\documentclass[10pt]{article}

\RequirePackage[amsthm,amsmath]{imsart}

\usepackage[mathscr]{eucal}
\usepackage{graphicx}
\usepackage{latexsym}
\usepackage{amssymb}
\usepackage{psfrag}
\usepackage{epsfig}
\usepackage{enumerate}
\DeclareMathAlphabet{\mathpzc}{OT1}{pzc}{m}{it}

\include{psbox}

\usepackage{amsfonts}
\usepackage{graphicx}
\usepackage{amsfonts}
\usepackage{color}
\usepackage[usenames,dvipsnames]{xcolor}
\usepackage[textsize=small]{todonotes}

\usepackage[numbers]{natbib}

\begin{document}


\newtheorem{proposition}{Proposition}[section]
\newtheorem{theorem}[proposition]{Theorem}
\newtheorem{corollary}[proposition]{Corollary}
\newtheorem{lemma}[proposition]{Lemma}
\newtheorem{conjecture}[proposition]{Conjecture}
\newtheorem{question}[proposition]{Question}
\newtheorem{definition}[proposition]{Definition}
\newtheorem{algorithm}[proposition]{Algorithm}
\newtheorem{assumption}[proposition]{Assumption}
\newtheorem{condition}[proposition]{Condition}
\renewcommand{\thefootnote}{\color{red}\arabic{footnote}}
\newcommand{\esquare}{\begin{flushright}$\square$\end{flushright}}
\numberwithin{equation}{section}
\numberwithin{proposition}{section}
\renewcommand{\theenumi}{\roman{enumi}}


%
%
\newcommand{\skp}{\vspace{\baselineskip}}
\newcommand{\noi}{\noindent}
\newcommand{\osc}{\mbox{osc}}
\newcommand{\lfl}{\lfloor}
\newcommand{\rfl}{\rfloor}

\theoremstyle{remark}
\newtheorem{example}{\bf Example}[section]
\newtheorem{remark}{\bf Remark}[section]

\newcommand{\img}{\imath}
\newcommand{\iy}{\infty}
\newcommand{\eps}{\epsilon}
\newcommand{\veps}{\varepsilon}
\newcommand{\del}{\delta}
\newcommand{\Rk}{\mathbb{R}^k}
\newcommand{\RR}{\mathbb{R}}
\newcommand{\spa}{\vspace{.2in}}
\newcommand{\V}{\mathcal{V}}
\newcommand{\E}{\mathbb{E}}
\newcommand{\I}{\mathbb{I}}
\newcommand{\PP}{\mathbb{P}}
\newcommand{\sgn}{\mbox{sgn}}
\newcommand{\ti}{\tilde}

\newcommand{\QQ}{\mathbb{Q}}

\newcommand{\XX}{\mathbb{X}}
\newcommand{\XXz}{\mathbb{X}^0}

\newcommand{\lan}{\langle}
\newcommand{\ran}{\rangle}
\newcommand{\lf}{\lfloor}
\newcommand{\rf}{\rfloor}
\def\wh{\widehat}
\newcommand{\defn}{\stackrel{def}{=}}
\newcommand{\txb}{\tau^{\epsilon,x}_{B^c}}
\newcommand{\tyb}{\tau^{\epsilon,y}_{B^c}}
\newcommand{\tilxb}{\tilde{\tau}^\eps_1}
\newcommand{\pxeps}{\mathbb{P}_x^{\eps}}
\newcommand{\non}{\nonumber}
\newcommand{\dist}{\mbox{dist}}

\newcommand{\Om}{\mathnormal{\Omega}}
\newcommand{\om}{\omega}
\newcommand{\vph}{\varphi}
\newcommand{\Del}{\mathnormal{\Delta}}
\newcommand{\Gam}{\mathnormal{\Gamma}}
\newcommand{\Sig}{\mathnormal{\Sigma}}

\newcommand{\tilyb}{\tilde{\tau}^\eps_2}
\newcommand{\beq}{\begin{eqnarray*}}
\newcommand{\eeq}{\end{eqnarray*}}
\newcommand{\beqn}{\begin{eqnarray}}
\newcommand{\eeqn}{\end{eqnarray}}
\newcommand{\ink}{\rule{.5\baselineskip}{.55\baselineskip}}

\newcommand{\bt}{\begin{theorem}}
\newcommand{\et}{\end{theorem}}
\newcommand{\deps}{\Del_{\eps}}
\newcommand{\dbl}{\mathbf{d}_{\tiny{\mbox{BL}}}}

\newcommand{\be}{\begin{equation}}
\newcommand{\ee}{\end{equation}}
\newcommand{\bes}{\begin{equation*}}
\newcommand{\ees}{\end{equation*}}
\newcommand{\ba}{\begin{aligned}}
\newcommand{\ea}{\end{aligned}}
\newcommand{\ac}{\mbox{AC}}
\newcommand{\BB}{\mathbb{B}}
\newcommand{\VV}{\mathbb{V}}
\newcommand{\DD}{\mathbb{D}}
\newcommand{\KK}{\mathbb{K}}
\newcommand{\HH}{\mathbb{H}}
\newcommand{\TT}{\mathbb{T}}
\newcommand{\CC}{\mathbb{C}}
\newcommand{\ZZ}{\mathbb{Z}}
\newcommand{\SSS}{\mathbb{S}}
\newcommand{\EE}{\mathbb{E}}
\newcommand{\NN}{\mathbb{N}}

\newcommand{\clg}{\mathcal{G}}
\newcommand{\clb}{\mathcal{B}}
\newcommand{\cls}{\mathcal{S}}
\newcommand{\clc}{\mathcal{C}}
\newcommand{\clj}{\mathcal{J}}
\newcommand{\clm}{\mathcal{M}}
\newcommand{\clx}{\mathcal{X}}
\newcommand{\cld}{\mathcal{D}}
\newcommand{\cle}{\mathcal{E}}
\newcommand{\clv}{\mathcal{V}}
\newcommand{\clu}{\mathcal{U}}
\newcommand{\clr}{\mathcal{R}}
\newcommand{\clt}{\mathcal{T}}
\newcommand{\cll}{\mathcal{L}}
\newcommand{\clz}{\mathcal{Z}}
\newcommand{\clq}{\mathcal{Q}}
\newcommand{\clo}{\mathcal{O}}

\newcommand{\cli}{\mathcal{I}}
\newcommand{\clp}{\mathcal{P}}
\newcommand{\cla}{\mathcal{A}}
\newcommand{\clf}{\mathcal{F}}
\newcommand{\clh}{\mathcal{H}}
\newcommand{\N}{\mathbb{N}}
\newcommand{\Q}{\mathbb{Q}}
\newcommand{\bfx}{{\boldsymbol{x}}}
\newcommand{\bfh}{{\boldsymbol{h}}}
\newcommand{\bfs}{{\boldsymbol{s}}}
\newcommand{\bfm}{{\boldsymbol{m}}}
\newcommand{\bff}{{\boldsymbol{f}}}
\newcommand{\bfb}{{\boldsymbol{b}}}
\newcommand{\bfw}{{\boldsymbol{w}}}
\newcommand{\bfz}{{\boldsymbol{z}}}
\newcommand{\bfu}{{\boldsymbol{u}}}
\newcommand{\bfell}{{\boldsymbol{\ell}}}
\newcommand{\bfn}{{\boldsymbol{n}}}
\newcommand{\bfd}{{\boldsymbol{d}}}
\newcommand{\bfbeta}{{\boldsymbol{\beta}}}
\newcommand{\bfzeta}{{\boldsymbol{\zeta}}}
\newcommand{\bfnu}{{\boldsymbol{\nu}}}

\newcommand{\bfl}{{ L}}
\newcommand{\bfc}{{C}}
\newcommand{\bfg}{{\bf G}}
\newcommand{\bfi}{{1}}
\newcommand{\bfk}{{\bf k}}
\newcommand{\bfDD}{{D}}

\newcommand{\bft}{{\bf T}}

\newcommand{\hatq}{\hat{Q}}
\newcommand{\hata}{\hat{A}}
\newcommand{\hatd}{\hat{D}}
\newcommand{\hats}{\hat{S}}
\newcommand{\hatx}{\hat{X}}
\newcommand{\haty}{\hat{Y}}
\newcommand{\hati}{\hat{I}}
\newcommand{\hatw}{\hat{W}}
\newcommand{\bard}{\bar{D}}
\newcommand{\barq}{\bar{Q}}
\newcommand{\bara}{\bar{A}}
\newcommand{\bars}{\bar{S}}
\newcommand{\bari}{\bar{I}}

\newcommand{\bart}{\bar{T}}
\newcommand{\Go}{\Rightarrow}

\newcommand{\curvz}{{\bf \mathpzc{z}}}
\newcommand{\curvx}{{\bf \mathpzc{x}}}
\newcommand{\curvi}{{\bf \mathpzc{i}}}
\newcommand{\curvs}{{\bf \mathpzc{s}}}
\newcommand{\blip}{\mathbb{B}_1}

\newcommand{\BM}{\mbox{BM}}

\newcommand{\tac}{\mbox{\scriptsize{AC}}}

\newcommand{\bfA}{{\boldsymbol{A}}}
\newcommand{\bfB}{{\boldsymbol{B}}}
\newcommand{\bfC}{{\boldsymbol{C}}}
\newcommand{\bfD}{D}
\newcommand{\bfG}{{\boldsymbol{G}}}
\newcommand{\bfH}{{\boldsymbol{H}}}
\newcommand{\vt}{\Theta}
\newcommand{\dmnk}{\mathcal{D}^{n,k}_m}
\newcommand{\hmnk}{\mathcal{H}^{n,k}_m}
%


\begin{frontmatter}
\title{ Construction of Asymptotically Optimal Control for a Stochastic Network from a Free Boundary Problem
}

 \runtitle{Asymptotically Optimal Control.}

\begin{aug}
\author{Amarjit Budhiraja, Xin Liu and Subhamay Saha\\ \ \\
}
\end{aug}

\today

\skp

\begin{abstract}
\noi
An asymptotic framework for optimal control of multiclass stochastic processing networks, using formal
diffusion approximations under suitable temporal and spatial scaling, by Brownian control problems (BCP)
and their equivalent workload formulations (EWF), has been developed by Harrison (1988). This  framework has been implemented in many works for 
constructing asymptotically optimal control policies for a broad range of stochastic network models.  To date all asymptotic optimality results for such networks correspond to settings where the solution of the EWF is a reflected Brownian motion in the positive orthant with normal reflections. In this work we consider a well studied stochastic network which is perhaps the
simplest example of a model with more than one dimensional workload process. In the regime considered here, the singular control
problem corresponding to the EWF does not have a simple form explicit solution, however by considering an associated free boundary problem one can give a representation for an optimal controlled process as a two dimensional reflected Brownian motion
in a Lipschitz domain whose boundary is determined by the solution of the free boundary problem. Using the form of the optimal solution we propose a sequence of control policies, given in terms of suitable thresholds, for the scaled stochastic network control problems and prove that this sequence of policies is asymptotically optimal. As suggested by the solution of the EWF, the policy we propose requires a server to idle under certain conditions which are specified in terms of the thresholds determined from the free boundary.
\\ \ \\

\noi {\bf AMS 2010 subject classifications:} Primary 60K25, 68M20, 90B36; secondary 60J70.\\ \ \\

\noi {\bf Keywords:} Stochastic networks, dynamic control,  heavy traffic, diffusion approximations, Brownian control problems,
singular control problems, reflected Brownian motions, free boundary problems, threshold policies, large deviations.
\end{abstract}

\end{frontmatter}

%
%
%
%
%
%
%
%
%
%

\section{Introduction}

Stochastic processing networks arise commonly in manufacturing, computer and communication systems.
Optimal regulation  is a key  objective in the design of such networks. Construction and implementation of an optimal control
can in general be challenging due to complex dynamics,  lack of simple Markovian state descriptors, and high frequency and throughput characteristics. An approach pioneered by Harrison \cite{harrison88} is to approximate the control problems for such complex networks, when the system is in {\em heavy traffic}, through certain control problems for Brownian motions.
These Brownian control problems (BCP) are quite non-standard in that the control processes may not even have bounded variation sample paths.  A key result of Harrison-van Mieghem \cite{HarVan} says that in quite general settings there are {\em equivalent
workload formulations} (EWF) of such BCP which correspond to  more tractable control problems.  In the EWF, controls are bounded variation processes and thus these problems fall within the classical framework of singular stochastic control, although here one has the
additional feature of state constraints in non-smooth domains (typically the state space is a convex polyhedral cone). Furthermore, in many examples the EWF is of much lower dimension than the original BCP, thus providing significant model simplification. In recent years there have been many works
that have developed asymptotically optimal control policies for a range of network models by analyzing the solutions of the associated BCP and EWF \cite{bellwill01, Kum, Har2, bellwill05, DaiLin, BudLiu}. 
Specifically, these works consider a sequence of control problems, indexed by a parameter $n$, for the underlying network
such that as $n$ becomes large the traffic intensity approaches criticality. An asymptotically optimal control is a sequence of
policies $\{T_n^*\}$ such that the (scaled) cost when $T_n^*$ is used in the $n$-th network is asymptotically the lowest that is achievable among all such control sequences.
One simplifying feature of all the models in the above papers is that the associated EWF is a one dimensional singular
control problem in $\mathbb{R}_+$ with a monotonic cost function. Such control problems have a simple solution given through
a one dimensional reflected Brownian motion and this explicit form plays a key role in the proofs.  

In \cite{AA} a well known 
queuing system with three buffers and two stations (see Figure 1), in heavy traffic, for which the associated EWF is a two dimensional singular control problem has been analyzed. This is perhaps the simplest non-trivial model with more than one-dimensional workload process.  The model, referred to in the literature as the {\em crisscross network}, has been previously analyzed in \cite{HarWei, YCY, MSS, KusMar2}. The network is of interest in its own right, but its analysis also gives insight for large networks with bottleneck sub-systems that have similar features as the criss-cross network. A brief description of the network is as follows -- there are $3$ classes of customers (corresponding to $3$ buffers) and $2$ servers; for $k=1,2,$ customers of Class $k$ arrive according to a renewal process  and receive service at Station $1$; Class $1$ customers leave the system once their service is completed; Class $2$ customers after being processed at Station $1$ proceeds to Station $2$, where they are re-designated as Class $3$ customers and get processed at Station 2 after which they leave the system. Precise descriptions
of the control problem and the cost criterion are given in Sections \ref{sec:schcont} and \ref{sec:netcont} respectively.
The form of an optimal control for this network depends on the underlying parameters, in particular on the 
strictly positive  holding cost vector $(c_1, c_2, c_3)$ and the (asymptotic) service rate vector $(\mu_1, \mu_2, \mu_3)$, and in general can be quite complex. Indeed the paper \cite{MSS} discusses two distinct regimes where the  structure of optimal control policies are expected to be quite different.   These regimes are as follows.  Case I:  $c_1\mu_1-c_2\mu_2+c_3\mu_2 \le 0$, and Case II:  $c_1\mu_1-c_2\mu_2+c_3\mu_2 > 0$. The paper \cite{MSS} further differentiates 
Case II into four sub-cases: Case IIA: $c_2\mu_2 - c_3\mu_2 \ge 0,\,\, c_2\mu_2-c_1\mu_1\geq 0$;
 Case IIB: $c_2\mu_2 - c_3\mu_2 <0,\,\, c_2\mu_2-c_1\mu_1\geq 0$;
 Case IIC: $c_2\mu_2 - c_3\mu_2 \ge 0,\,\, c_2\mu_2-c_1\mu_1< 0$;
 Case IID: $c_2\mu_2 - c_3\mu_2 <0,\,\, c_2\mu_2-c_1\mu_1 < 0$. 

Case I is the simplest to analyze and it  has been shown in \cite{YCY} that a simple priority policy (server 1 always gives priority to Buffer 1) is (asymptotically) optimal. Asymptotically optimal control policies for Case IIA have been constructed in \cite{MSS,AA}. In this case using certain monotonicity properties one can show that the optimal state process for the singular control problem is a two dimensional
reflected Brownian motion in $\RR_+^2$ with normal reflections.  Furthermore,  using results of \cite{HarVan} an explicit solution of the BCP can be given as well. The proposed policies and the proof of asymptotic optimality in \cite{MSS} and \cite{AA} are quite different -- the first paper uses quite technical machinery from viscosity solution analysis of Hamilton-Jacobi-Bellman(HJB) equations whereas the paper \cite{AA} proceeds by constructing a multiple threshold policy for which the
associated state process closely mimics the solution of the BCP. 

Cases IIB, IIC and IID, to date have remained unsolved.  One of the key obstacles in their analysis has been that in these regimes the singular control problem does not admit a simple form solution.  Indeed, although Harrison's framework has provided asymptotically optimal control policies for a broad range of models, all of the available results correspond to settings where the
solution of the EWF is a reflected Brownian motion in the positive orthant (with normal reflections). In \cite{AK} the authors studied the singular control problem that corresponds to the EWF for Cases IIB and IIC.  Typically, solutions of singular control problems are given in terms of an open set $\clo$ in the state space, referred to as the {\em continuation region}, such that starting within $\clo$ no control is applied until the boundary of $\clo$
is reached; if the initial condition is in $\bar{(\clo)^c}$ ({\em action region}), an instantaneous control in a pre-specified direction is applied to bring the state to $\partial \clo$; and once in $\bar \clo$, the state process is constrained in the set by suitable reflection at $\partial \clo$.  In terms of the associated HJB equation, in $\clo$ the value function satisfies a linear elliptic PDE while in $\clo^c$ a nonlinear first order PDE is satisfied; the boundary $\partial \clo$ separating these two regions is referred to as the {\em free boundary} for the system of PDE and determining this boundary is called a {\em free boundary problem}.  In \cite{AK} it was shown
that the value function $J^*$ of the control problem is $\mathcal{C}^1$ up to the boundary and the continuation region for the optimal control is given as 
$G = \{x\in \RR^2_+: x_1 \ge \Psi(x_2)\}$ where $\Psi: \RR_+ \to \RR_+$ is a Lipschitz non-decreasing function given as 
\begin{align}\label{imfcn}
\Psi(w_2)\doteq \sup\{w_1\geq 0:\partial_{w_1}J^*(w_1,w_2)=0\}\,.
\end{align}
One novel feature of this result is that here the principle of smooth fit ideas that have been used in previous works \cite{BSW, HarTak, ShSo1, ShSo2} are not applicable and in fact $C^2$-regularity of the value functions -- a crucial ingredient in these works -- is not available. The paper \cite{AK} constructs an optimal controlled process as a reflected Brownian motion in $G$ reflected at $\partial G = \partial_1G \cup \partial_2G$, where the
direction of reflection is $e_2 = (0,1)'$ on $ \partial_2G = \{x \in \RR_+^2: x\cdot e_2 =0\}$ and is $e_1=(1,0)'$ on $ \partial_1G = \{x \in \RR_+^2:\Psi(x_2)=x_1\}$.   The solution is not altogether explicit since in order
to determine $\Psi$ one need the value function of the control problem, however numerical methods for computing the
free boundary $\Psi$ are available \cite{MutKum2, KusMar, BudRos}.

The goal of this work is to use the solution of the free boundary problem from \cite{AK} (given by the function $\Psi$) to obtain an asymptotically optimal control policy for the corresponding crisscross network in regime IIB and IIC.
We only study Case IIB here since treatment of the other case is expected to be quite similar.   Recall that Case IIB corresponds to $ c_2\mu_2 - c_3\mu_2 <0,\,\, c_2\mu_2-c_1\mu_1\geq 0$ (note that it implies
$c_1\mu_1-c_2\mu_2+c_3\mu_2 >0$).
In this case serving Class $1$ customers reduces  cost of jobs in Buffer 1 at an (asymptotic) average rate of $c_1\mu_1$.
Also serving Class $2$ customers  reduces  cost of jobs in Buffer 2 at an (asymptotic) average rate of $c_2\mu_2$ and at the same time  increases  cost at an (asymptotic) average rate of $c_3\mu_2$ for Buffer $3$. The  condition $c_1\mu_1> c_2\mu_2-c_3\mu_2$ implies that overall cost is reduced at a higher rate if Server 1 processes Class $1$ customers instead
of Class $2$ customers.
The second condition $c_2\mu_2 < c_3\mu_2$  says that it is cheaper to keep jobs in Buffer $2$ than in Buffer $3$.
 The third condition $c_2\mu_2\ge c_1\mu_1$ means the  cost from the queues processed by Server 1 is reduced more rapidly
if jobs in Buffer 2 are processed.
The first condition suggests that a  priority policy that favors Class 1 customers should be used, however the third condition says that the minimization of immediate workload at Server 1 will be achieved by processing Class 2 customers.
Also, always giving high priority to Class 1 customers may lead to an undesirable underutilization of Server 2. Thus an optimal policy needs to suitably balance these opposing considerations. Additionally, even if there are jobs in the second queue (but say no jobs in Queue 1) it may be preferable for Server 1 to idle since holding costs in Queue 3 are higher than that in Queue 2.
Thus an optimal control is not expected to be a non-idling policy.  In Section \ref{sec:netcont} we  describe our proposed policy that suitably takes into account the various complex features of this parameter regime.
Furthermore (as in \cite{AA}) the policy is designed so that the associated state process closely mimics the solution of the BCP given in Section \ref{BCP}. The policy we propose will require Server 1 to idle under certain circumstances which are specified in terms of a threshold determined from the free boundary $\Psi$. 

For asymptotic optimality we will assume that the inter-arrival and service times have  finite moment generating functions in a neighborhood of $0$ (unlike \cite{AA} we do not assume these random variables to be Exponential).
These conditions allow the use of certain large deviation estimates that are key in the proof of asymptotic optimality.  Such large deviation techniques for obtaining asymptotically optimal control policies for stochastic networks were first introduced by Bell-Williams\cite{bellwill01} and later also used in \cite{Har2, AA, bellwill05}. Our main result is Theorem \ref{main-result} which under Assumptions \ref{rates-c}, \ref{htc}, \ref{ass:ass3} and \ref{regime-c} proves the asymptotic optimality of the control policy in Definition \ref{propolicy} with a suitable choice of threshold parameters $c, l_0, g_0$. This result treats Case IIB but, as noted earlier,  Case IIC can be treated in a similar fashion. Treatment of Case IID is  a challenging open problem. For this case even the solutions of the associated EWF and the BCP are currently unavailable. We have the following conjecture for the form of the solution to the EWF, which if resolved will be a key step forward in the construction of asymptotically optimal control policies in regime IID.
\begin{conjecture}
	\label{conj:conj1}
Let $B$ be the two dimensional Brownian motion on a filtered probability space $(\Om, \clf,$ $ \{\clf_t\}_{t\ge 0}, \PP)$  as in Definition \ref{def:ewf} and consider the EWF where
$\hat h$ is as in \eqref{LP-value} and $(\mu_1, \mu_2, \mu_3)$, $(c_1,c_2,c_3)$ are as in Case IID. Then there exist functions $\Psi_i: \RR_+ \to \RR_+$, $i=1,2$,
that are Lipschitz, strictly increasing, and $\Psi_i(x) \to \infty$ as $x \to \infty$, such that
there is a unique pair of $\{\clf_t\}$ adapted continuous processes $W_1^*, W_2^*$ with values in $\RR_+$ that solve
\begin{equation}
\begin{aligned}\label{workload-general-solu-conj}
{W}^*_1(t)&= B_1(t) + \sup_{0\leq s \leq t}\left[B_1(s)-\Psi_1({W}^*_2(s))\right]^-\,,\\
{W}^*_2(t)&= B_2(t)+\sup_{0\leq s \leq t}\left[B_2(s)-\Psi_2({W}^*_1(s))\right]^-\, ,
\end{aligned}
\end{equation}
and $I_1^*(t) = \sup_{0\leq s \leq t}\left[B_1(s)-\Psi_1({W}^*_2(s))\right]^-$,
$I_2^*(t) = \sup_{0\leq s \leq t}\left[B_2(s)-\Psi_2({W}^*_1(s))\right]^-$ is an optimal control for the EWF.
\end{conjecture}
Note that Case IIB corresponds to a setting where $\Psi_1 = \Psi$, with $\Psi$ as in \eqref{imfcn}, and $\Psi_2 =0$.

The paper is organized as follows.  In Section \ref{sec:qnet} we present the queuing network considered in this work. We also introduce our main assumptions on the arrival and service processes.  In Section \ref{sec:netcont} we introduce our scheduling policy and state the main result which gives asymptotic optimality of the policy with suitable choices of threshold parameters.
Section \ref{BCP} reviews results from \cite{AK}, in particular we present the solution of the Brownian control problem associated
with the network from Section \ref{sec:qnet}.  Section \ref{sec:sec5} contains the proof of our main result: Theorem \ref{main-result}.  Key steps in the proof of the theorem are contained in  Theorem \ref{impthm1}, Theorem \ref{impthm2}
and Lemma \ref{lower-bound-est}, the proofs of which are given in Section \ref{impthms}.  
Theorems \ref{impthm1} and \ref{impthm2} are the most technical parts of the paper.  The statement of these results are in the same vein as Theorems 4.8 and 4.9 in \cite{AA}.  However these latter results assume that the interarrival and service times are exponential and the proofs in the general distribution case treated here are substantially more technical.  In particular, the proof of Theorem 4.9 in \cite{AA} relies on  sample path large deviation estimates for Poisson processes, whereas in the proof of Theorem \ref{impthm2}, we   use fixed time large deviation estimates (as in Lemma \ref{LDP-renewal}) for renewal processes.
Finally the appendix summarizes some elementary facts about the one dimensional Skorohod map that are appealed to in our proofs.

The following notation will be used. For a Polish space $S$, $D([0,\infty):S)$ will denote the space of right continuous functions with left limits (RCLL) from $[0,\infty)$ to $S$ equipped with the usual Skorohod topology. Define
$\cld_1 = \{f \in D([0,\infty):\RR): f(0) \ge 0\}$.  All stochastic processes in this work will have RCLL sample paths.
A stochastic process $X$ with values in $S$ will be regarded as a random variable with values in 
$D([0,\infty):S)$.
Convergence in distribution of $S$ valued random variables $X^n$ to $X$ will be denoted as $X^n \Rightarrow X$.
A sequence $X^n$ of processes with sample paths in $D([0,\infty):S)$ is said to be $C$-tight if the corresponding sequence of probability laws is relatively compact (in the usual weak convergence topology) and any limit point is supported  on the space of $S$ valued continuous functions.

\section{Queueing Network Model}
\label{sec:qnet}

\subsection{Network structure}

Consider a sequence of  networks indexed by  $n\in \N$ of the form in Figure 1. The $n$th network consists of $3$ classes of customers (corresponding to $3$ buffers) and $2$ servers. For $k=1,2,$ customers of Class $k$ arrive according to a renewal process with rate $\lambda^n_k$ and receive service at Station $1$. Class $1$ customers leave the system once their service is completed. Class $2$ customers after being served at Station $1$ proceed to Station $2$, where they are re-designated as Class $3$ customers and get served at Station 2. The service rates for these $3$ classes of customers are denoted by $\mu^n_j, j=1,2,3.$ Within each class, customers are processed using the First-Come-First-Served discipline.  
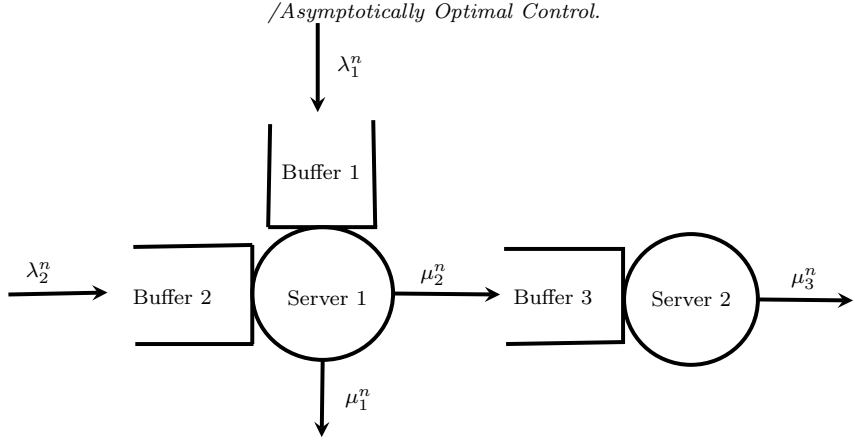
\begin{figure}
\ifx\du\undefined
  \newlength{\du}
\fi
\setlength{\du}{15\unitlength}
\begin{tikzpicture}
\pgftransformxscale{1.000000}
\pgftransformyscale{-1.000000}
\definecolor{dialinecolor}{rgb}{0.000000, 0.000000, 0.000000}
\pgfsetstrokecolor{dialinecolor}
\definecolor{dialinecolor}{rgb}{1.000000, 1.000000, 1.000000}
\pgfsetfillcolor{dialinecolor}
\definecolor{dialinecolor}{rgb}{1.000000, 1.000000, 1.000000}
\pgfsetfillcolor{dialinecolor}
\pgfpathellipse{\pgfpoint{19.925000\du}{11.575000\du}}{\pgfpoint{1.775000\du}{0\du}}{\pgfpoint{0\du}{1.675000\du}}
\pgfusepath{fill}
\pgfsetlinewidth{0.100000\du}
\pgfsetdash{}{0pt}
\pgfsetdash{}{0pt}
\pgfsetmiterjoin
\definecolor{dialinecolor}{rgb}{0.000000, 0.000000, 0.000000}
\pgfsetstrokecolor{dialinecolor}
\pgfpathellipse{\pgfpoint{19.925000\du}{11.575000\du}}{\pgfpoint{1.775000\du}{0\du}}{\pgfpoint{0\du}{1.675000\du}}
\pgfusepath{stroke}
\definecolor{dialinecolor}{rgb}{0.000000, 0.000000, 0.000000}
\pgfsetstrokecolor{dialinecolor}
\node at (19.925000\du,11.770000\du){};
\pgfsetlinewidth{0.100000\du}
\pgfsetdash{}{0pt}
\pgfsetdash{}{0pt}
\pgfsetbuttcap
{
\definecolor{dialinecolor}{rgb}{0.000000, 0.000000, 0.000000}
\pgfsetfillcolor{dialinecolor}
\definecolor{dialinecolor}{rgb}{0.000000, 0.000000, 0.000000}
\pgfsetstrokecolor{dialinecolor}
\draw (15.150000\du,10.400000\du)--(18.150000\du,10.350000\du);
}
\pgfsetlinewidth{0.100000\du}
\pgfsetdash{}{0pt}
\pgfsetdash{}{0pt}
\pgfsetbuttcap
{
\definecolor{dialinecolor}{rgb}{0.000000, 0.000000, 0.000000}
\pgfsetfillcolor{dialinecolor}
\definecolor{dialinecolor}{rgb}{0.000000, 0.000000, 0.000000}
\pgfsetstrokecolor{dialinecolor}
\draw (15.200000\du,12.850000\du)--(18.180000\du,12.850000\du);
}
\pgfsetlinewidth{0.100000\du}
\pgfsetdash{}{0pt}
\pgfsetdash{}{0pt}
\pgfsetbuttcap
{
\definecolor{dialinecolor}{rgb}{0.000000, 0.000000, 0.000000}
\pgfsetfillcolor{dialinecolor}
\definecolor{dialinecolor}{rgb}{0.000000, 0.000000, 0.000000}
\pgfsetstrokecolor{dialinecolor}
\draw (18.150000\du,10.350000\du)--(18.150000\du,12.850000\du);
}
\pgfsetlinewidth{0.100000\du}
\pgfsetdash{}{0pt}
\pgfsetdash{}{0pt}
\pgfsetbuttcap
{
\definecolor{dialinecolor}{rgb}{0.000000, 0.000000, 0.000000}
\pgfsetfillcolor{dialinecolor}
\definecolor{dialinecolor}{rgb}{0.000000, 0.000000, 0.000000}
\pgfsetstrokecolor{dialinecolor}
\draw (24.500000\du,10.450000\du)--(27.480000\du,10.450000\du);
}
\pgfsetlinewidth{0.100000\du}
\pgfsetdash{}{0pt}
\pgfsetdash{}{0pt}
\pgfsetbuttcap
{
\definecolor{dialinecolor}{rgb}{0.000000, 0.000000, 0.000000}
\pgfsetfillcolor{dialinecolor}
\definecolor{dialinecolor}{rgb}{0.000000, 0.000000, 0.000000}
\pgfsetstrokecolor{dialinecolor}
\draw (24.550000\du,12.850000\du)--(27.510000\du,12.800000\du);
}
\pgfsetlinewidth{0.100000\du}
\pgfsetdash{}{0pt}
\pgfsetdash{}{0pt}
\pgfsetbuttcap
{
\definecolor{dialinecolor}{rgb}{0.000000, 0.000000, 0.000000}
\pgfsetfillcolor{dialinecolor}
\definecolor{dialinecolor}{rgb}{0.000000, 0.000000, 0.000000}
\pgfsetstrokecolor{dialinecolor}
\draw (27.500000\du,10.400000\du)--(27.500000\du,12.850000\du);
}
\definecolor{dialinecolor}{rgb}{0.000000, 0.000000, 0.000000}
\pgfsetstrokecolor{dialinecolor}
\node[anchor=west] at (14.900000\du,11.662500\du){Buffer $2$};
\pgfsetlinewidth{0.100000\du}
\pgfsetdash{}{0pt}
\pgfsetdash{}{0pt}
\pgfsetbuttcap
{
\definecolor{dialinecolor}{rgb}{0.000000, 0.000000, 0.000000}
\pgfsetfillcolor{dialinecolor}
\pgfsetarrowsend{stealth}
\definecolor{dialinecolor}{rgb}{0.000000, 0.000000, 0.000000}
\pgfsetstrokecolor{dialinecolor}
\draw (21.700000\du,11.575000\du)--(24.400000\du,11.600000\du);
}
\pgfsetlinewidth{0.100000\du}
\pgfsetdash{}{0pt}
\pgfsetdash{}{0pt}
\pgfsetbuttcap
{
\definecolor{dialinecolor}{rgb}{0.000000, 0.000000, 0.000000}
\pgfsetfillcolor{dialinecolor}
\pgfsetarrowsend{stealth}
\definecolor{dialinecolor}{rgb}{0.000000, 0.000000, 0.000000}
\pgfsetstrokecolor{dialinecolor}
\draw (30.900000\du,11.725000\du)--(33.300000\du,11.750000\du);
}
\pgfsetlinewidth{0.100000\du}
\pgfsetdash{}{0pt}
\pgfsetdash{}{0pt}
\pgfsetbuttcap
{
\definecolor{dialinecolor}{rgb}{0.000000, 0.000000, 0.000000}
\pgfsetfillcolor{dialinecolor}
\pgfsetarrowsend{stealth}
\definecolor{dialinecolor}{rgb}{0.000000, 0.000000, 0.000000}
\pgfsetstrokecolor{dialinecolor}
\draw (12.000000\du,11.600000\du)--(14.500000\du,11.550000\du);
}
\definecolor{dialinecolor}{rgb}{0.000000, 0.000000, 0.000000}
\pgfsetstrokecolor{dialinecolor}
\node[anchor=west] at (24.500000\du,11.650000\du){Buffer $3$};
\definecolor{dialinecolor}{rgb}{0.000000, 0.000000, 0.000000}
\pgfsetstrokecolor{dialinecolor}
\node[anchor=west] at (18.800000\du,11.675000\du){Server $1$};
\definecolor{dialinecolor}{rgb}{0.000000, 0.000000, 0.000000}
\pgfsetstrokecolor{dialinecolor}
\node[anchor=west] at (29.319781\du,10.763566\du){};
\definecolor{dialinecolor}{rgb}{0.000000, 0.000000, 0.000000}
\pgfsetstrokecolor{dialinecolor}
\node[anchor=west] at (27.200000\du,11.500000\du){};
\pgfsetlinewidth{0.100000\du}
\pgfsetdash{}{0pt}
\pgfsetdash{}{0pt}
\pgfsetbuttcap
{
\definecolor{dialinecolor}{rgb}{0.000000, 0.000000, 0.000000}
\pgfsetfillcolor{dialinecolor}
\definecolor{dialinecolor}{rgb}{0.000000, 0.000000, 0.000000}
\pgfsetstrokecolor{dialinecolor}
\draw (18.600000\du,7.300000\du)--(18.550000\du,9.900000\du);
}
\pgfsetlinewidth{0.100000\du}
\pgfsetdash{}{0pt}
\pgfsetdash{}{0pt}
\pgfsetbuttcap
{
\definecolor{dialinecolor}{rgb}{0.000000, 0.000000, 0.000000}
\pgfsetfillcolor{dialinecolor}
\definecolor{dialinecolor}{rgb}{0.000000, 0.000000, 0.000000}
\pgfsetstrokecolor{dialinecolor}
\draw (18.550000\du,9.900000\du)--(21.300000\du,9.900000\du);
}
\pgfsetlinewidth{0.100000\du}
\pgfsetdash{}{0pt}
\pgfsetdash{}{0pt}
\pgfsetbuttcap
{
\definecolor{dialinecolor}{rgb}{0.000000, 0.000000, 0.000000}
\pgfsetfillcolor{dialinecolor}
\definecolor{dialinecolor}{rgb}{0.000000, 0.000000, 0.000000}
\pgfsetstrokecolor{dialinecolor}
\draw (21.250000\du,9.900000\du)--(21.200000\du,7.200000\du);
}
\pgfsetlinewidth{0.100000\du}
\pgfsetdash{}{0pt}
\pgfsetdash{}{0pt}
\pgfsetbuttcap
{
\definecolor{dialinecolor}{rgb}{0.000000, 0.000000, 0.000000}
\pgfsetfillcolor{dialinecolor}
\pgfsetarrowsend{stealth}
\definecolor{dialinecolor}{rgb}{0.000000, 0.000000, 0.000000}
\pgfsetstrokecolor{dialinecolor}
\draw (19.925000\du,13.250000\du)--(19.900000\du,15.200000\du);
}
\definecolor{dialinecolor}{rgb}{0.000000, 0.000000, 0.000000}
\pgfsetstrokecolor{dialinecolor}
\node[anchor=west] at (18.650000\du,8.500000\du){Buffer $1$};
\definecolor{dialinecolor}{rgb}{0.000000, 0.000000, 0.000000}
\pgfsetstrokecolor{dialinecolor}
\node[anchor=west] at (16.050000\du,11.650000\du){};
\definecolor{dialinecolor}{rgb}{0.000000, 0.000000, 0.000000}
\pgfsetstrokecolor{dialinecolor}
\node[anchor=west] at (19.925000\du,11.575000\du){};
\definecolor{dialinecolor}{rgb}{0.000000, 0.000000, 0.000000}
\pgfsetstrokecolor{dialinecolor}
\node[anchor=west] at (29.319781\du,10.763566\du){};
\definecolor{dialinecolor}{rgb}{1.000000, 1.000000, 1.000000}
\pgfsetfillcolor{dialinecolor}
\pgfpathellipse{\pgfpoint{29.212500\du}{11.725000\du}}{\pgfpoint{1.687500\du}{0\du}}{\pgfpoint{0\du}{1.650000\du}}
\pgfusepath{fill}
\pgfsetlinewidth{0.100000\du}
\pgfsetdash{}{0pt}
\pgfsetdash{}{0pt}
\pgfsetmiterjoin
\definecolor{dialinecolor}{rgb}{0.000000, 0.000000, 0.000000}
\pgfsetstrokecolor{dialinecolor}
\pgfpathellipse{\pgfpoint{29.212500\du}{11.725000\du}}{\pgfpoint{1.687500\du}{0\du}}{\pgfpoint{0\du}{1.650000\du}}
\pgfusepath{stroke}
\definecolor{dialinecolor}{rgb}{0.000000, 0.000000, 0.000000}
\pgfsetstrokecolor{dialinecolor}
\node at (29.212500\du,11.920000\du){};
\definecolor{dialinecolor}{rgb}{0.000000, 0.000000, 0.000000}
\pgfsetstrokecolor{dialinecolor}
\node[anchor=west] at (27.975000\du,11.720000\du){Server $2$};
\pgfsetlinewidth{0.100000\du}
\pgfsetdash{}{0pt}
\pgfsetdash{}{0pt}
\pgfsetbuttcap
{
\definecolor{dialinecolor}{rgb}{0.000000, 0.000000, 0.000000}
\pgfsetfillcolor{dialinecolor}
\pgfsetarrowsend{stealth}
\definecolor{dialinecolor}{rgb}{0.000000, 0.000000, 0.000000}
\pgfsetstrokecolor{dialinecolor}
\draw (19.800000\du,4.750000\du)--(19.800000\du,7.000000\du);
}
\definecolor{dialinecolor}{rgb}{0.000000, 0.000000, 0.000000}
\pgfsetstrokecolor{dialinecolor}
\node[anchor=west] at (20.050000\du,5.850000\du){$\lambda^n_1$};
\definecolor{dialinecolor}{rgb}{0.000000, 0.000000, 0.000000}
\pgfsetstrokecolor{dialinecolor}
\node[anchor=west] at (12.200000\du,11.050000\du){$\lambda^n_2$};
\definecolor{dialinecolor}{rgb}{0.000000, 0.000000, 0.000000}
\pgfsetstrokecolor{dialinecolor}
\node[anchor=west] at (22.150000\du,11.100000\du){$\mu^n_2$};
\definecolor{dialinecolor}{rgb}{0.000000, 0.000000, 0.000000}
\pgfsetstrokecolor{dialinecolor}
\node[anchor=west] at (20.250000\du,14.250000\du){$\mu^n_1$};
\definecolor{dialinecolor}{rgb}{0.000000, 0.000000, 0.000000}
\pgfsetstrokecolor{dialinecolor}
\node[anchor=west] at (31.500000\du,11.200000\du){$\mu^n_3$};
\end{tikzpicture}
\caption{The $n$th crisscross network.}
\end{figure}

A precise mathematical description is as follows. Let $(\Omega, \mathcal{F},\mathbb{P})$ be a complete probability space; all random variables and stochastic processes described in this work are, without loss of generality, defined on this common probability space. 
For $k=1,2,$ and $j=1,2,3$, let $\{u_k(i):i=1,2,\cdots\}$ be a sequence of i.i.d. random variables with mean $1$ and standard deviation $\sigma_k$, and $\{v_j(i):i=1,2,\cdots\}$ a sequence of i.i.d. random variables with mean $1$ and standard deviation $\varsigma_j$. In the $n$th network, for $k=1,2$ and $j=1,2,3$, the inter-arrival times $\{u^n_k(i):i=1,2,\cdots\}$ for Class $k$ customers, and the service times $\{v^n_j(i):i=1,2,\cdots\}$ for Class $j$ customers are given by 
\bes
u^n_k(i) = \frac{1}{\lambda^n_k} u_k(i), \ \ v^n_j(i) = \frac{1}{\mu^n_j} v_j(i),
\ees
where $\lambda^n_k, \mu^n_j \in (0,\infty)$ are the arrival and service rates. We further assume that the sequences of inter-arrival times and service times are mutually independent for each $n\in\NN$. Define
\begin{align*}\xi^n_k(l)=\sum_{i=1}^lu_k^n(i),\; \eta^n_j(l)=\sum_{i=1}^lv_j^n(i)\;\;\mbox{for}\,\,l=1,2,\cdots,\,\,k=1,2, \,\,j=1,2,3.
\end{align*}
Then the arrival and service processes can be described as follows:
\begin{align*}
A^n_k(t)=\sup\{l\geq 0:\xi^n_k(l)\leq t\},\; S^n_j(t)=\sup\{l\geq 0:\eta^n_j(l)\leq t\},\,\,\,t\geq 0,\,\,k=1,2,\,\,j=1,2,3.
\end{align*}
Thus $A^n_k(t)$ represents the numbers of customers of Class $k$ who have arrived up to time $t$ and $S^n_j(t)$ represents the number of customers of Class $j$ who would have finished service up to time $t$ if the corresponding server had continuously served Class $j$ customers during time interval $[0,t]$. We make the following assumptions on the arrival and service rates.
\begin{assumption}\label{rates-c}
For $k=1,2$ and $j=1,2,3$, there exist $\lambda_k, \mu_j \in (0,\infty)$ such that 
$\lim_{n\rightarrow \infty}\lambda_k^n=\lambda_k$,  $\lim_{n\rightarrow \infty}\mu_j^n=\mu_j$.
\end{assumption}
The following is our main heavy traffic assumption.
\begin{assumption} \label{htc} The following relations hold for the arrival and service rate parameters:
\begin{align}\label{htc-1}
\frac{\lambda_1}{\mu_1}+\frac{\lambda_2}{\mu_2}=1,\,\,\,\frac{\lambda_2}{\mu_3}=1\,,
\end{align}
and there exist $b_i \in \mathbb{R}$, $i=1,2,3$, such that 
\begin{align}\label{htc-2}
\lim_{n\to\infty} \sqrt{n}\biggl(\frac{\lambda^n_i}{\mu^n_i}-\frac{\lambda_i}{\mu_i}\biggr) = b_i,\,\,\,i=1,2,\,\,\,\lim_{n\to\infty} \sqrt{n}\biggl(\frac{\lambda^n_2}{\mu^n_3}-1\biggr) = b_3\,.
\end{align}
\end{assumption}
Condition \eqref{htc-1} says the traffic intensities at both stations converge to $1$ as $n\to\infty$, while the convergence rates of the traffic intensities are characterized in \eqref{htc-2}.

Finally, we make the following assumption on the logarithmic moment
generating functions for interarrival and service times that enables certain large deviation estimates  for the renewal processes $A^n_k, S^n_j$ (see Lemma \ref{LDP-renewal}).
\begin{assumption}
	\label{ass:ass3}
There is a non-empty open neighborhood $\clo$ of $0\in\RR$ such that for all $l\in\clo$, 
$$
\Lambda_{a,k}\doteq \log \EE(e^{{l u_k(1)}/{\lambda_k}}) < \infty, \; \Lambda_{s,j}\doteq \log \EE(e^{{l v_j(1)}/{\mu_j}}) < \infty,\; \ k=1,2,\;\ j=1,2, 3.
$$
\end{assumption}

\subsection{Scheduling control}
\label{sec:schcont}
Scheduling control for the $n$th network is described by a vector-valued allocation process
$$T^n(t)=(T^n_1(t),T^n_2(t),T^n_3(t))', \,\,\,t\geq 0\,,$$ where for $j=1,2,3,$ $T^n_j(t)$ represents the cumulative amount of service time devoted to Class $j$ customers in the time interval $[0,t]$. The idle time processes at two servers are defined as follows:
$$I^n_1(t)=t-T^n_1(t)-T^n_2(t),\,\,\,\,I^n_2(t)=t-T^n_3(t),\,\,\,t\geq 0\,.$$ 
For simplicity, we assume the system is initially empty. Then the queue-length processes corresponding to the three types of customers can be described as follows. For $t\geq 0$,
\be\label{queue-length}
Q^n_i(t) =A_i^n(t)-S^n_i(T^n_i(t)),\,\,\,i=1,2,\; \ \ Q^n_3(t)=S^n_2(T^n_2(t))-S^n_3(T^n_3(t)).
\ee
We write $A^n=(A^n_1, A^n_2)'$, $\lambda^n=(\lambda^n_1, \lambda^n_2)'$. The quantities $S^n, I^n, Q^n$, $\mu^n, \lambda, \mu$ are defined similarly.

In order to precisely formulate the family of admissible control policies, we need to introduce the notion of multi-parameter filtrations and stopping times (cf. Section 2.8 of \cite{EthKur}).
Define for $a = (a_1, a_2)'\in \NN^2, b = (b_1, b_2, b_3)' \in \NN^3$, 
\bes
\clg^n(a, b) = \sigma\left\{u^n_i(\tilde a_i), v^n_j(\tilde b_j): \tilde a_i \le a_i, \tilde b_j \le b_j, \ i=1,2, j=1,2,3\right\}.
\ees
Then $\{\clg^n(a,b): a\in\NN^2, b \in \NN^3\}$ is a {\em multiparameter filtration} with the following (partial) ordering
\bes
(\tilde a, \tilde b) \le (a, b) \ \mbox{if and only if} \ \tilde a_i \le a_i, \tilde b_j \le b_j, \ i=1,2, j=1,2,3.
\ees
A $\{\clg^n(a,b): a\in\NN^2, b \in \NN^3\}$ {\em multiparameter stopping time} is a random variable $\bft$ which takes values in $\bar\NN^5$, where $\bar\NN = \NN \cup \{\infty\}$, such that 
\bes
\{\bft = (a, b)\} \in \clg^n(a,b), \ \ \mbox{for all $a\in\NN^2, b \in \NN^3.$}
\ees
The $\sigma$-field associated with such a stopping time is given by 
\bes
\clg^n_\bft = \left\{B\in \clf: B\cap \{\bft = (a, b)\}\in \clg^n(a,b), \ \mbox{for all $a\in\NN^2, b\in\NN^3$}\right\}. 
\ees
The scheduling control process $\{T^n(t)\}$ is required to satisfy the following  conditions.  
\begin{itemize}
\item[\rm (i)] For $j=1,2,3$ and $n \in \NN$, $T^n_j$ is a continuous non-decreasing process with $T^n_j(0)=0$.
\item[\rm (ii)] For $k=1,2$ and $n \in \NN$, $I^n_k$ is a continuous non-decreasing process with $I^n_k(0)=0$.
\item[\rm (iii)] For $j=1,2,3$, $t\ge 0$, and $n \in \NN$, $Q^n_j(t)\geq 0$.
\item[\rm (iv)] Define for each $t \ge 0$ and $n \in \NN$, a $\NN^5$ valued random variable
$$\sigma_0^n(t) \doteq \left(A_i^n(nt)+1, i=1,2;\; S^n_j(T^n_j(nt))+1, j=1,2,3\right).$$
Then, for each $t\ge 0$, $\sigma^n_0(t)$ is a $\{\clg^n(a,b): a\in\NN^2, b \in \NN^3\}$ stopping time.  

Define the filtration $\{\clg^n_1(t): t\ge 0\}$ as $\clg^n_1(t) = \clg^n_{\sigma^n_0(t)}$. Then $I^n(nt)$ is $\clg^n_1(t)$ measurable for every $t\ge 0$.
\end{itemize} 
From (i) and (ii), we see that for all $j=1,2,3,$
\be\label{T-lip}
\mbox{ $T_j^n$ is uniformly (in $n$) Lipschitz continuous with Lipschitz constant bounded by $1$.}
\ee 
Although condition (iv) above appears somewhat technical, it is a natural non-anticipativity property and Theorem 5.4 of \cite{AA1} shows that the condition is satisfied for a very broad family of control processes.

Any process $T^n$ satisfying the above conditions will be referred to as an admissible control policy for the $n$th network.


\subsection{Scaled processes}\label{scaled}
Now we define fluid-scaled and diffusion-scaled processes corresponding to the processes described above. 
For each $n \in \N$, 
define for $t\geq 0$, $\bar{T}^n(t)  \doteq n^{-1}T^n(nt)$. Processes $\bar{I}^n, \bar{A}^n, \bar{S}^n, \bar{Q}^n$ are defined similarly. These will occasionally be referred to as fluid scaled processes.


We also defined diffusion-scaled processes:
\begin{equation*}
\begin{aligned}
\hat{A}^n(t) & \doteq n^{-1/2}(A^n(nt)- n\lambda^n t),\;
\hat{S}^n(t)  \doteq n^{-1/2}(S^n(nt)- n\mu^n t),\\
\hat{T}^n(t) & \doteq n^{-1/2}T^n(nt),\;
 \hat{I}^n(t)  \doteq n^{-1/2}I^n(nt),\;
\hat{Q}^n(t)  \doteq n^{-1/2}Q^n(nt). 
\end{aligned}
\end{equation*}
We next define the workload process $W^n =\{(W^n_1(t),W^n_2(t))': t\geq 0\}$, which measures the amount of service  needed for customers that are in the system at time $t$. More precisely, for $t\ge 0$, define
\be\label{workload}
W^n_1(t)= \frac{Q^n_1(t)}{\mu^n_1}+\frac{Q^n_2(t)}{\mu^n_2},\;
W^n_2(t) = \frac{Q^n_2(t)}{\mu^n_3}+\frac{Q^n_3(t)}{\mu^n_3}.
\ee
The fluid and diffusion scaled workload processes are defined as follows: For $t\ge 0$,
\bes
\bar W^n(t) \doteq n^{-1} W^n(nt),\; \hat W^n(t) \doteq n^{-1/2} W^n(nt). 
\ees

\section{Main Result}
\label{sec:netcont}

We consider an expected infinite horizon discounted cost, associated with an admissible control $T^n$ and the corresponding normalized queue-length process $\hat{Q}^n$, given as follows:
\begin{align}\label{cf}
\hat{J}^n(T^n)=\mathbb{E}\biggl(\int_0^{\infty}e^{-\gamma t}c\cdot\hat{Q}^n(t)dt\biggr)\,,
\end{align}where $\gamma \in (0, \infty)$ is the discount factor and $c\equiv (c_1,c_2,c_3)'$ is a strictly positive vector of holding costs.  The aim is to find a sequence $\{T^n\}$ of scheduling control policies which is asymptotically optimal, namely it satisfies
\begin{align*}
\lim_{n\rightarrow \infty}\hat{J}^n(T^n)=\inf \liminf_{n\rightarrow \infty}\hat{J}^n(\tilde{T}^n)\,,
\end{align*}where the infimum is over the set of all admissible control policy sequences $\{\tilde{T}^n\}_{n\in\NN}$.
We make the following assumption on the service rates and the holding cost parameters.
\begin{assumption}\label{regime-c} 
$ c_2\mu_2 - c_3\mu_2 <0,\,\, c_2\mu_2-c_1\mu_1\geq 0$.
\end{assumption}
As noted in the Introduction, this parameter regime is the Case IIB considered in \cite{MSS} and an optimal policy needs to suitably balance several considerations such as, overall cost is reduced at a higher rate if Server 1 processes Class $1$ customers instead
of Class $2$ customers; it is cheaper to keep jobs in Buffer $2$ than in Buffer $3$;  the  cost of jobs in the queues processed by Server 1 is reduced more rapidly
if jobs in Buffer 2 are processed; and that it is undesirable to have Server 2 idle when there is work in Buffer 2.

We now describe our proposed policy which suitably takes into account the various complex features of this parameter regime.
The policy is motivated by the form of the solution of the Brownian control problem associated with this control problem.
This point will be explained further in Section \ref{BCP} (see comments below Corollary \ref{optimal-solns}) where the Brownian control problem is studied in detail.
 
Fix $c, l_0 \in (1, \infty)$ and $g_0 \in (0,\infty)$. Define $L^n\doteq \lfloor l_0\log n\rfloor$ and $C^n\doteq \lfloor c_0 \log n\rfloor$, where $c_0=cl_0$. Since we are only interested in asymptotic optimality, we will assume without any loss of generality, that $n\geq \bar{n}$, where $\bar{n}$ is such that for all $n\geq \bar{n}$, $C^n-L^n-1\geq 1$ and $\frac{\mu_1^n}{\mu_2^n}(C^n-L^n+2)\geq 1$. 
The control policy will be specified in terms of the {\em free boundary}
$\{(x_1,x_2): x_1 = \Psi(x_2)\}$ where $\Psi$ is as specified in \eqref{imfcn} with $J^*$ as in \eqref{valuebcp}.

\begin{definition}[Control policy]\label{propolicy}
The policy is as follows. Server $2$ processes jobs from Buffer 3 whenever the buffer is nonempty. The sequencing policy for Server $1$ is as follows. At time $s\in [0,\infty),$
\begin{itemize}
\item[] if $Q^n_3(s)-\frac{\mu^n_2}{\mu^n_1}Q^n_1(s)< L^n$,
\begin{itemize}
\item[] serve Class $1$ customers (provided the queue is non-empty) if either $Q^n_3(s) \geq C^n-1$ or $Q^n_2(s)=0$,
\item[] serve Class $2$ customers if $Q^n_3(s)<C^n-1$ and $Q^n_2(s) > 0$;
\end{itemize}
\item[] if $Q^n_3(s)-\frac{\mu^n_2}{\mu^n_1}Q^n_1(s)\geq L^n$,
\begin{itemize}
\item[] serve Class $1$ customers  if either $Q^n_1(s)\geq \frac{\mu_1^n}{\mu_2^n}(C^n-L^n+2)$ or $(Q^n_2(s)=0 \mbox{ and } Q^n_1(s) > 0)$,
\item[] serve Class $2$ customers if $Q^n_1(s)< \frac{\mu_1^n}{\mu_2^n}(C^n-L^n+2)$, $Q^n_2(s) > 0$
and $W^n_1(s) -\sqrt{n}\Psi(W^n_2(s)/\sqrt{n})\\\ge g_0$,
\item[] idle Server 1 if $Q^n_1(s)< \frac{\mu_1^n}{\mu_2^n}(C^n-L^n+2)$, $Q^n_2(s) > 0$
and $W^n_1(s) -\sqrt{n}\Psi(W^n_2(s)/\sqrt{n}) < g_0$.
\end{itemize}
\end{itemize}
\end{definition}

\begin{remark}
The free boundary $\Psi$ is given in terms of the value function $J^*$ of the {\em Workload Control Problem} in
Section \ref{BCP}.  Although this function does not have a closed form expression, there are well developed numerical
methods for solving such free boundary problems (see \cite{KusMar, MutKum1, MutKum2, BudRos}). 
\end{remark}

\begin{remark}
We note that the proposed control policy addresses many of the complex features of the parameter regime. In particular
the policy says that if Queue 1 is sufficiently large then it receives priority, unless there is  a risk of Server 2
idling despite there being jobs in Queue 2. Furthermore, Server 1 idles when Queue 1 is (asymptotically) negligible and there
is enough work in Queue 3, so as to keep jobs in the buffer with a lower holding cost.  This latter property is enforced by the last line in the control policy that involves the free boundary function $\Psi$.
\end{remark}

\begin{remark}
	From Theorem 5.4 of \cite{AA1} it follows that the control policy in Definition \ref{propolicy} is an admissible control policy in the sense of Section \ref{sec:schcont}.
	An explicit representation for the policy in terms of a vector allocation process $T^n$ can be given in a manner analogous to Remark 3.8 of \cite{AA}.
\end{remark}


The following is our main result, which gives the asymptotic optimality of the  policy proposed in Definition \ref{propolicy}. The limit of the cost under the proposed policy is  characterized in Theorem \ref{main2} in terms of the solution of a Brownian control problem. 
\begin{theorem}\label{main-result}
Let $g_0\in(0,\infty)$. There exist $c, \bar{l} \in (1,\infty)$ such that the sequence of scheduling controls $\{T^n\}$ defined in Definition \ref{propolicy} with threshold parameters $c, l_0, g_0$, with $l_0\geq \bar{l}$, satisfies 
\bes
\lim_{n\to \infty} \hat{J}^n(T^n) = \inf\liminf_{n\to\infty} \hat{J}^n(\tilde T^n),
\ees
where the infimum is taken over all admissible control policy sequences $\{\tilde{T}^n\}$.
\end{theorem}

\begin{remark}
The parameters $c$ and $\bar l$ in Theorem \ref{main-result} can be chosen as follows. 

For $i=1,2,3, j=1,2$, and $l=1,2$, define $\Theta^{s,i}_l$ and $\Theta^{a,j}_l$ as $\Theta_l$ in Lemma \ref{LDP-renewal} with $\Lambda^*$ replaced by $\Lambda^*_{s,i}$ and $\Lambda^*_{a,j}$, respectively, where $\Lambda^*_{s,i}$ and $\Lambda^*_{s,j}$ are the Legendre-Fenchel transforms of $\Lambda_{s,i}$ and $\Lambda_{a,j}$, respectively. Now let $\theta_4$ be as in Theorem \ref{impthm1}. From \eqref{final-est-1} one can see that, $\theta_4$ can be chosen as follows:
\[
\theta_4 =\left( \inf_{n\ge 1} \frac{\mu^n_1}{2\mu^n_2(\lambda^n_1+\epsilon)}\right) \min \left\{\lambda_1 \Theta^{a,1}_1(\lambda_1, \epsilon), \ \mu_1 \Theta^{s,1}_2(\mu_1, \epsilon), \ p_0\epsilon/(2\mu_1) \right\},
\]
where $\epsilon \in (0, (\mu_1-\lambda_1)/4)$ and $p_0\in \mathcal{O}$. Let $c= 1+ \frac{4}{\theta_4}$. 

Next let $\gamma_4$ be as in Theorem \ref{impthm2}. From \eqref{final-est-2} we see that $\gamma_4$ can be taken to be
{\small\begin{align*}
\gamma_4 & = \min\left\{ d\lambda_2\Theta^{a,2}_1(\lambda_2, \epsilon_1)/K, \ d(\theta+1)(\mu_i-2\epsilon_1)\Theta^{s,i}_2(\mu_i, \epsilon_1)/K, \ d\theta\mu_i\Theta^{s,i}_1(\mu_i,\epsilon_1)/K, \ dp_0(\theta+1)\epsilon_1/(2\mu_iK),\right.\\
& \quad\quad\quad  \left. \mu_i\Theta^{s,i}_1(\mu_i,\epsilon_1)/(4\mu_3), \ \lambda_2\Theta^{a,2}_1(\lambda_2,\epsilon_1)/(4\mu_3), \ (\mu_i-2\epsilon_1)\Theta^{s,i}_2(\mu_i,\epsilon_1)/(4\mu_3), \right. \\
& \quad \quad\quad  \left.  (\lambda_2-2\epsilon_1)\Theta^{a,2}_2(\lambda_2,\epsilon_1)/(4\mu_3), \ p_0\epsilon_1/(8\mu_3\mu_i), \ p_0\epsilon_1/(8\mu_3\lambda_2), \ \ i=2,3  \right\},
\end{align*}}
where $\epsilon_1, d, K, \theta$ are as in \eqref{imp-parameters}. Finally, let $\bar l > \max\{1, \frac{3}{\gamma_4} \}$. 
\end{remark}

\section{Brownian Control Problem and Equivalent Workload Formulation} \label{BCP}

We now introduce the Brownian control problem associated with the control problem from Section \ref{sec:netcont}. Roughly speaking, the BCP is obtained by taking a formal limit of the sequence of queueing control problems. Using the scaling defined in Section \ref{scaled}, we have, from \eqref{queue-length},  for a given sequence of admissible control policies $\{T^n\}$, and for all $t\geq 0$, 
\begin{equation}\label{scaled1}
\begin{aligned}
\hat{Q}^n_i(t)&=\hat{A}^n_i(t)-\hat{S}^n_i(\bar{T}^n_i(t))+ \sqrt{n}(\lambda^n_it-\mu^n_i\bar{T}^n_i(t))\\
& = \hat{A}^n_i(t)-\hat{S}^n_i(\bar{T}^n_i(t)) + \sqrt{n}\mu^n_i\left(\frac{\lambda^n_i}{\mu^n_i} -\frac{\lambda_i}{\mu_i}\right) t + \sqrt{n} \mu^n_i \left(\frac{\lambda_i}{\mu_i} t - \bar{T}^n_i(t) \right), \ \ i=1,2, \\
\hat{Q}^n_3(t)&=\hat{S}^n_2(\bar{T}^n_2(t))-\hat{S}^n_3(\bar{T}^n_3(t)) + \sqrt{n}(\mu^n_2\bar{T}^n_2(t)-\mu^n_3\bar{T}^n_3(t)) \\
& = \hat{S}^n_2(\bar{T}^n_2(t))-\hat{S}^n_3(\bar{T}^n_3(t)) + \sqrt{n} \left(\mu^n_2 \frac{\lambda_2}{\mu_2} - \mu^n_3 \right) t
- \sqrt{n}\mu^n_2\left(\frac{\lambda_2}{\mu_2} t - \bar{T}^n_2(t) \right) + \sqrt{n}\mu^n_3(t - \bar{T}^n_3(t)).
\end{aligned}
\end{equation}
For $t\geq 0$, let
\begin{equation}\label{scaled2}
\begin{aligned}
\hat{X}^n_i(t)&\doteq \hat{A}^n_i(t)-\hat{S}^n_i(\bar{T}^n_i(t)) + \sqrt{n}\mu^n_i\left(\frac{\lambda^n_i}{\mu^n_i} -\frac{\lambda_i}{\mu_i}\right) t ,\,\,\,i=1,2,\\
\hat{X}^n_3(t)&\doteq \hat{S}^n_2(\bar{T}^n_2(t))-\hat{S}^n_3(\bar{T}^n_3(t)) + \sqrt{n}\left(\mu^n_2\frac{\lambda_2}{\mu_2}t-\mu^n_3t\right),
\end{aligned}
\end{equation}
and 
\begin{equation}\label{scaled3}
\hat{Y}^n_i(t)\doteq  \sqrt{n} \left(\frac{\lambda_i}{\mu_i} t - \bar{T}^n_i(t)\right),\,\,\,i=1,2,\;\ \ \hat{Y}^n_3(t)\doteq \sqrt{n}(t - \bar{T}^n_3(t)).
\end{equation}
From \eqref{scaled1} -- \eqref{scaled3}, we get the following relationships:
\be\label{scaled4}
\hatq^n_i(t)  = \hatx^n_i(t) + \mu^n_i \haty^n_i(t), \ \ i=1,2,\; \ \
\hatq^n_3(t)  = \hatx^n_3(t)  - \mu^n_2\haty^n_2(t) + \mu^n_3\haty^n_3(t). 
\ee
Define for $t\geq 0$,
\be\label{T-star}
\bart^*(t) = \left(\frac{\lambda_1}{\mu_1}, \frac{\lambda_2}{\mu_2}, 1\right)' t.
\ee
Write $\hatx^n = (\hatx^n_1, \hatx^n_2, \hatx^n_3)'$. 
It can be argued (see for example Lemma 3.3 in \cite{AA1}) that, under `reasonable' control policies, $\bart^n\Go\bart^*$ and consequently, using functional central limit theorem for renewal processes, under such policies
\be\label{BM-conv}
\hatx^n\Go X,
\ee 
where $X$ is a three-dimensional Brownian motion that starts from the origin and has drift $(\mu_1b_1,\mu_2b_2,\mu_3b_3-\mu_2b_2)$ and covariance matrix
\begin{align*}
\begin{pmatrix}
\sigma^2_1\lambda_1 + \varsigma^2_1\lambda_1 & 0 & 0\\
0 & \sigma_2^2\lambda_2 + \varsigma^2_1\lambda_2 & -\varsigma^2_1\lambda_2\\
0 & -\varsigma^2_1\lambda_2 & \varsigma^2_2\lambda_2 + \varsigma_3^2 \mu_3
\end{pmatrix}\,.
\end{align*}
Also,
\be\label{deviation-idle}
\hati^n_1(t) = \haty^n_1(t) + \haty^n_2(t), \ \mbox{and} \ \hati^n_2(t) = \haty^n_3(t), \ t\ge 0,
\ee
which are nondecreasing processes starting from $0$.

Thus taking a formal limit as $n\rightarrow \infty$ in \eqref{scaled4},  we arrive at the following BCP.
\begin{definition}\label{bcp}
Let $X$ be a three dimensional Brownian motion as in \eqref{BM-conv}, given on some filtered probability space
$(\Om, \clf, \{\clf_t\}_{t\ge 0}, \PP)$. The Brownian control problem is to find an $\mathbb{R}^3$-valued $\{\clf_t\}$-adapted stochastic process $\tilde{Y} =(\tilde{Y}_1, \tilde{Y}_2, \tilde{Y}_3)'$,  which minimizes
\begin{align*}
\mathbb{E}\biggl(\int_0^{\infty}e^{-\gamma t}c\cdot\tilde{Q}(t)dt\biggr)\,,
\end{align*}
subject to the following conditions. For all $t\geq 0$,
\begin{equation}\label{bcp1}
0  \le  \tilde{Q}_i(t) \doteq {X}_i(t) + \mu_i\tilde{Y}_i(t),\; i=1,2, \; 
0  \le  \tilde{Q}_3(t) \doteq {X}_3(t) + \mu_3\tilde{Y}_3(t)-\mu_2\tilde{Y}_2(t),
\end{equation}
and
\begin{equation}\label{bcp2}
\tilde{I}_1 \doteq \tilde{Y}_1 + \tilde{Y}_2,\;
\tilde{I}_2 \doteq \tilde{Y}_3\; \mbox{ are non-decreasing, and }\,\,\tilde{I}_i(0)=0,\; i=1,2.
\end{equation}
We will refer to any $\{\clf_t\}$-adapted process $\tilde{Y} = (\tilde Y_1, \tilde Y_2, \tilde Y_3)'$ satisfying \eqref{bcp1} and \eqref{bcp2} as an admissible control for the BCP and an admissible control that achieves the minimum cost as an optimal control for the BCP.
\end{definition}
We now introduce an equivalent workload formulation  of the above BCP that makes use of a certain static deterministic
 linear programming (LP) problem. Recall the workload process defined in \eqref{workload}.
Using \eqref{scaled4} and \eqref{deviation-idle}, we have for $t\ge 0$, 
\be\label{scaled-workload}\ba
\hatw^n_1(t)&= \frac{\hatq^n_1(t)}{\mu^n_1}+\frac{\hatq^n_2(t)}{\mu^n_2} = \frac{\hatx^n_1(t)}{\mu^n_1}+\frac{\hatx^n_2(t)}{\mu^n_2} + \hati^n_1(t), \\
\hatw^n_2(t)&= \frac{\hatq^n_2(t)}{\mu^n_3}+\frac{\hatq^n_3(t)}{\mu^n_3}  = \frac{\hatx^n_2(t)}{\mu^n_3}+\frac{\hatx^n_3(t)}{\mu^n_3} + \hati^n_2(t). 
\ea\ee
 Fix $w_1,w_2 \in [0,\infty)$. Consider the LP problem  defined as:
\be\label{LP}\begin{aligned}
&\mbox{minimize}_{q_1,q_2,q_3}\,\,\, c_1q_1+c_2q_2+c_3q_3\\
&\mbox{subject to}\,\,\,\frac{q_1}{\mu_1}+\frac{q_2}{\mu_2}=w_1,\;
\frac{q_2}{\mu_3}+\frac{q_3}{\mu_3}=w_2,\;
q_1,q_2,q_3\geq 0.
\end{aligned}\ee
A straightforward calculation using the fact that $c_1\mu_1-c_2\mu_2+c_3\mu_2 > 0$ shows that the value of the LP is
\begin{align}\label{LP-value}
\hat{h}(w_1,w_2)=\begin{cases}(c_1\mu_1)w_1+\frac{\mu_3}{\mu_2}(c_2\mu_2-c_1\mu_1)w_2, \,\,\,\mbox{when} \,\,\,\mu_3w_2\leq \mu_2w_1, \\
(c_2\mu_2-c_3\mu_2)w_1+(c_3\mu_3)w_2, \,\,\,\mbox{when} \,\,\,\mu_3w_2\geq \mu_2w_1,
\end{cases}
\end{align}
and the optimal solution is 
\be\label{LP-soln}\ba
& q^*_1 =\frac{\mu_1}{\mu_2}(\mu_2w_1-\mu_3w_2), & q^*_2 = \mu_3 w_2, \ \ & q^*_3 = 0, & \mbox{if} \ \mu_3w_2\le \mu_2w_1,\\
& q^*_1 =0, & q^*_2 = \mu_2 w_1, \ \ & q^*_3 = \mu_3w_2-\mu_2w_1, & \mbox{if} \ \mu_3w_2\ge \mu_2w_1.
\ea\ee
Using \eqref{LP-value}, and taking a formal limit in \eqref{scaled-workload}, we arrive at the following  control problem, which is usually referred to as the equivalent workload formulation (for the BCP in Definition \ref{bcp}). 
\begin{definition}
	\label{def:ewf}
Let $X$ and $(\Om, \clf, \{\clf_t\}_{t\ge 0}, \PP)$ be as in Definition \ref{bcp}. The equivalent workload formulation (EWF) is to find an $\mathbb{R}^2$-valued $\{\clf_t\}$-adapted stochastic process $\tilde{I}=(\tilde{I}_1,\tilde{I}_2)'$,  which minimizes
\begin{align}\label{workload-obj}
\mathbb{E}\biggl(\int_0^{\infty}e^{-\gamma t}\hat{h}(\tilde{W}(t))dt\biggr)\,,
\end{align}
subject to the following conditions. For all $t\geq 0$
\begin{equation}\label{workload-EWF}
\begin{aligned}
&0 \leq \tilde{W}_1(t)\doteq B_1(t)+\tilde{I}_1(t),\;
0 \leq \tilde{W}_2(t)\doteq B_2(t)+\tilde{I}_2(t), \\
&\tilde{I}_1, \ \tilde{I}_2 \,\,\mbox{are nondecreasing, and}\,\,\,\tilde{I}_1(0) = \tilde{I}_2(0)=0\, ,
\end{aligned}
\end{equation}
where
\begin{equation}
	\label{eq:eq911}
	B_1(t) = \frac{{X}_1(t)}{\mu_1}+\frac{{X}_2(t)}{\mu_2}, \;\; B_2(t) = \frac{{X}_2(t)}{\mu_3}+\frac{{X}_3(t)}{\mu_3}.
\end{equation}
We will refer to any $\{\clf_t\}$-adapted process $\tilde{I} = (\tilde I_1, \tilde I_2)'$ satisfying \eqref{workload-EWF}  as an admissible control for the EWF and an admissible control that achieves the minimum cost as an optimal control for the EWF.
\end{definition}

The BCP and EWF of the above form were first introduced by Harrison in \cite{harrison88}, and they have been used extensively in the study of optimal scheduling for multiclass queuing networks in heavy traffic 
(see \cite{HarVan, bellwill01, bellwill05, Har2, Kum,  AA, DaiLin, AA1, AA2, BudLiu}). In particular, the BCP and EWF introduced here are identical to those in \cite{AA}. 
The following lemma says that in order to solve the BCP it suffices to solve the associated EWF. The proof of this lemma is straightforward from \eqref{LP-value} and \eqref{LP-soln} and we refer the reader to Section 3.1 of \cite{AA} 
for details. 
\begin{lemma}\label{BCP-EWF}
Suppose $\tilde I^*$ is an optimal control of the EWF. Denote by $\tilde W^*$ the corresponding optimal workload (which is defined by \eqref{workload-EWF} with $\tilde I$ replaced by $\tilde I^*$). Define when $\mu_3\tilde{W}^*_2(t)< \mu_2\tilde{W}^*_1(t)$,
\begin{align}
\tilde{Y}^*_1(t) & \doteq -\frac{X_3(t)}{\mu_2} + \tilde I^*_1(t) - \frac{\mu_3}{\mu_2} \tilde I^*_2(t), \ \tilde Y^*_2(t)  \doteq \frac{X_3(t)}{\mu_2} + \frac{\mu_3}{\mu_2} \tilde I^*_2(t), \ \tilde Y^*_3(t)  \doteq \tilde I^*_2(t), \label{eq:eq1235a}
\end{align}
 and when $\mu_3\tilde{W}^*_2(t)\geq \mu_2\tilde{W}^*_1(t)$,
 \begin{align}
\tilde{Y}^*_1(t) & \doteq -\frac{X_1(t)}{\mu_1}, \ \tilde Y^*_2(t) \doteq \frac{X_1(t)}{\mu_1} +  \tilde I^*_1(t), \ \tilde Y^*_3(t)  \doteq\tilde I^*_2(t). \label{eq:eq1235b}
\end{align}
Then $\tilde Y^*$ is an optimal control of the BCP. 
\end{lemma}
We note that if $\tilde Y^*$ is an optimal solution of the BCP, then the corresponding optimal queue length is 
\begin{align}
\tilde{Q}^*_1(t)&=\frac{\mu_1}{\mu_2}(\mu_2\tilde{W}^*_1(t)-\mu_3\tilde{W}^*_2(t)), \ \tilde{Q}^*_2(t)=\mu_3\tilde{W}^*_2(t), \ \tilde{Q}^*_3(t)=0, \label{eq:eq1237a}
\end{align}
if $\mu_3\tilde{W}^*_2(t) < \mu_2\tilde{W}^*_1(t)$, and 
 \begin{align}
\tilde{Q}^*_1(t)&=0,\ \tilde{Q}^*_2(t) =\mu_2\tilde{W}^*_1(t),\ \tilde{Q}^*_3(t) =\mu_3\tilde{W}^*_2(t)-\mu_2\tilde{W}^*_1(t), 
\label{eq:eq1237b}
\end{align}
if $\mu_3\tilde{W}^*_2(t) \ge \mu_2\tilde{W}^*_1(t)$.

Although the BCP and EWF here are the same as those in \cite{AA},  the solutions to these problems are much less straightforward than 
in \cite{AA}.  The latter paper considers the Case IIA, where, in particular, $c_2\mu_2-c_3\mu_2 \geq 0$ and
$c_2\mu_2-c_1\mu_1 \geq 0$. In their setting, $\hat{h}$ is a non-decreasing function of both its arguments and consequently, the solution of the EWF is trivial, in that it is given by the solution of the one-dimensional Skorohod problem (see Appendix \ref{SP} for the definition and properties of Skorohod problem). However for the regime considered here, $\hat{h}$ no longer has the above monotonicity property and thus a simple closed form solution of the EWF or BCP is not available. The BCP in Cases IIB
and IIC has been investigated in \cite{AK} where using certain optimal stopping problems a solution of the EWF has been provided
in terms of the free boundary associated with the control problem.
Below we summarize some key results from \cite{AK} that will be needed here. 

Consider the workload control problem in Definition \ref{def:ewf} corresponding to an arbitrary initial condition.
 More precisely, letting $B$ and
$(\Om, \clf, \{\clf_t\}_{t\ge 0}, \PP)$ be as in Definition \ref{def:ewf}, and $w = (w_1,w_2)' \in \mathbb{R}_+^2$, the control problem is to find an $\RR_+^2$-valued $\{\clf_t\}$-adapted non-decreasing process $\tilde I^w = (\tilde I_1^w, \tilde I_2^w)'$ to minimize 
\be\label{objective-general}
\mathbb{E}\biggl(\int_0^{\infty}e^{-\gamma t}\hat{h}(\tilde{W}^w(t))dt\biggr)
\ee
subject to the following conditions: For $t\ge 0,$
\begin{equation}\label{workload-general}
\begin{aligned}
&0 \le \tilde{W}_1^w(t)\doteq w_1+ B_1(t)+\tilde{I}_1^w(t),\\
&0 \le \tilde{W}_2^w(t)\doteq w_2+B_2(t)+\tilde{I}^w_2(t).
\end{aligned}
\end{equation}
We refer to  $\tilde{I}^w = (\tilde{I}_1^w, \tilde{I}_2^w)'$ as an admissible control for the initial condition $w$.
Define the optimal value function as:
\begin{align}\label{valuebcp}
J^*(w)= \inf_{\tilde{I}^w}\mathbb{E}\biggl(\int_0^{\infty}e^{-\gamma t}\hat{h}(\tilde{W}^w(t))dt\biggr),\; w \in \RR_+^2,
\end{align}
where the infimum is taken over all admissible controls for the initial condition $w$.

From Theorem 3.1. of \cite{AK}, $J^*$ is a $C^1$ function on $\RR_+^2$. Now define for 
$w_2 \in \mathbb{R}_+$, $\Psi: \RR_+ \to \RR_+$ as in \eqref{imfcn}.
The following result is taken from \cite{AK}.
\begin{lemma}[Lemma 5.1. of \cite{AK}]\label{imp-fcn} The function $\Psi$ has the following properties.
\begin{itemize}
\item[\rm (i)] For all $w_2 \ge 0$, $0\leq \Psi(w_2)\leq \frac{\mu_3}{\mu_2}w_2$.
\item[\rm (ii)] $\Psi$ is non-decreasing and Lipschitz continuous with Lipschitz constant bounded by $\frac{\mu_3}{\mu_2}$.
\item[\rm (iii)] $\displaystyle\lim_{w_2\rightarrow \infty}\Psi(w_2)=\infty$. 
\end{itemize}
\end{lemma}
The following result from \cite{AK} gives an optimal solution of the workload control problem in Definition \ref{def:ewf}. 
\begin{theorem}[Theorem 5.2. of \cite{AK}]\label{larger-control} For $t\ge 0$, define 
\begin{equation}
\begin{aligned}\label{workload-general-solu}
{W}^*_1(t)&= B_1(t) + \sup_{0\leq s \leq t}\left[B_1(s)-\Psi({W}^*_2(s))\right]^-\,,\\
{W}^*_2(t)&= B_2(t)+\sup_{0\leq s \leq t}\left[B_2(s)\right]^-\, ,
\end{aligned}
\end{equation}
where for $z \in \RR$, $z^{-} = -\min\{0,z\}$. Then the minimum value of \eqref{workload-obj} over all admissible controls
is given as 
\begin{align}\label{optimal-limit-soln}
J^*(0)=\mathbb{E}\biggl(\int_0^{\infty}e^{-\gamma t}\hat{h}(W^*(t))dt\biggr)\,.
\end{align}
Thus 
\be\label{optimal-idle}
{I}^*_1(t)  = \sup_{0\leq s \leq t}\left[B_1(s)-\Psi({W}^*_2(s))\right]^-,\; 
{I}^*_2(t)  = \sup_{0\leq s \leq t}\left[B_2(s)\right]^-. 
\ee
is an optimal control for the EWF.
\end{theorem}
Consider the one dimensional Skorohod map $\Gamma: \cld_1 \to D([0,\infty):\RR_+)$ defined as 
$$\Gamma(f)(t) \doteq f(t) + \sup_{0\leq s\leq t} (f(s))^{-}, \; f \in \cld_1, \; t \ge 0.$$
 Then for $t\ge 0,$
\begin{align}\label{eq:eq1100}
{W}^*_2(t)& \doteq\Gamma (B_2)(t),\;
{W}^*_1(t) \doteq\Gamma(B_1-\Psi({W}^*_2))(t)+\Psi({W}^*_2(t))\,.
\end{align}
Recall
$
G = \{x\in \RR^2_+: x_1 \ge \Psi(x_2)\}.
$
Clearly,
\be\label{optimal-workload}
W^*(t)\in G \ \ \mbox{for all $t\ge 0.$}
\ee 
Roughly speaking, the process $W^*$ behaves like the  Brownian motion $(B_1, B_2)'$ in the interior of $G$ and it is reflected on the boundary of $G$, where the directions of reflection on $\{x\in \RR^2_+: x \cdot e_2 =0\}$ and $\{x\in \RR^2_+: x_1 =\Psi(x_2)\}$ are $e_2$ and $e_1$, respectively.

Combining Theorem \ref{larger-control} with Lemma \ref{BCP-EWF} we have the following corollary.
\begin{corollary}\label{optimal-solns}
Let $W^* = (W^*_1, W^*_2)'$ and $I^*=(I^*_1, I^*_2)'$ be as in Theorem \ref{larger-control}. Then $Y^*$  defined by \eqref{eq:eq1235a} - \eqref{eq:eq1235b} (replacing $\tilde I^*$ there with $I^*$) is an optimal control for the BCP and
$Q^*$ defined by \eqref{eq:eq1237a} - \eqref{eq:eq1237b} (replacing $\tilde W^*$ there by $W^*$) is the corresponding optimally controlled state process.
\end{corollary}
The solution to the BCP given above suggests the following control policy for the $n$-th network. Note that the set
$\{Q_3^*(t) < \frac{\mu_2}{\mu_1} Q_1^*(t)\}$ equals $\{\mu_3 W_2^*(t) < \mu_2 W_1^*(t)\}$.  Since on this set $Q_3^*(t) =0$,
a good policy for the $n$-th network should keep $\hat Q^n_3(t)$ close to $0$ when $Q^n_3(t) < \frac{\mu_2^n}{\mu_1^n}Q_1^n(t)$.
Similarly, when $Q^n_3(t) \ge \frac{\mu_2^n}{\mu_1^n}Q_1^n(t)$ the policy should ensure that $\hat Q^n_1(t)$ is close to $0$.
This motivates the thresholds $L^n = \lfloor l_0 \log n\rfloor$ and $C^n = \lfloor c_0 \log n\rfloor$ introduced above Definition \ref{propolicy}. When $Q_3^n(t) - \frac{\mu_2^n}{\mu_1^n}Q_1^n(t) < L^n$, under the policy in Definition  \ref{propolicy}, 
Server 1 processes Class 2 customers (leading to an increase in $Q_3^n(t)$) only when $Q_3^n(t) < C^n-1$.  Thus $C^n$
can be interpreted as the level of `safety stock'  that prevents idleness of Server 2.  Since in the diffusion scaling the 
safety stock levels approach $0$ (i.e. $C^n/\sqrt{n} \to 0$ as $n \to \infty$) the policy ensures that $\hat Q^n_3$ is close to $0$ in this regime.  Similarly when $Q_3^n(t) - \frac{\mu_2^n}{\mu_1^n}Q_1^n(t) \ge L^n$, Server 1 processes Class 1 jobs as soon as
$Q_1^n(t) \ge \frac{\mu_1^n}{\mu_2^n} (C^n-L^n+2)$, ensuring that in this regime $\hat Q^n_1$ remains close to $0$.
We refer the reader to Corollary \ref{impcor} for a convergence result that makes these statements precise.  Finally, \eqref{optimal-workload} suggests that under a near optimal policy the condition $\hat W^n(t) \in G$ for all $t$ should be
satisfied approximately for large $n$.  As shown in Lemma \ref{lower-bound-est}, the policy in Definition \ref{propolicy}
satisfies this property. The proof of this lemma relies on the key idleness property formulated in the last line of
Definition  \ref{propolicy}.

\section{Proof of  asymptotic optimality }
\label{sec:sec5}
In this section we  prove Theorem \ref{main-result} which gives the asymptotic optimality of the policy proposed in Definition \ref{propolicy} for a suitable choice of threshold parameters $g_0,c$ and $l_0$. We begin with the
following result which is an immediate consequence of Theorem 3.1 of \cite{AA1}.
\begin{theorem}\label{main1}
Let for $n\ge 1$, $\tilde T^n$ be an admissible control policy for the $n$-th network. Then with $J^*(0)$ as in \eqref{optimal-limit-soln}, we have,
$\liminf_{n \rightarrow \infty}\hat{J}^n(\tilde T^n)\geq J^*(0)\,.$
\end{theorem}
The above theorem  says that the optimal cost of the BCP is a lower bound for the asymptotic cost for any sequence of admissible
control policies.  Thus it suffices to show that the sequence of policies in Definition \ref{propolicy} (with a suitable choice
of threshold parameters) asymptotically achieves the optimal cost of the BCP. 
This is done in the theorem below.
\begin{theorem}\label{main2}
There exist $c, \bar{l} \in (1,\infty)$ such that for any $g_0\in(0,\infty)$ and $l_0\in [\bar l, \infty)$, the sequence of  control policies $\{T^n\}$  in Definition \ref{propolicy} with threshold parameters $c, l_0, g_0$, satisfies the following:

\textup{(i)} $(\hat{W}^n,\hat{I}^n)\Rightarrow ({W}^*, {I}^*) \,\,\,\mbox{as}\,\,\,n\rightarrow \infty,$

\textup{(ii)} $\hat{J}^n(T^n) \to J^*(0) \,\,\,\mbox{as}\,\,\,n\rightarrow \infty,$

\noindent where $(\hat{W}^n,\hat{I}^n)$ are defined as in Section \ref{sec:qnet} using the above sequence of control policies, and  $W^*, I^*$,  $J^*(0)$ are as in \eqref{workload-general-solu}, \eqref{optimal-idle}, and \eqref{optimal-limit-soln} respectively. 
\end{theorem}
Proof of Theorem \ref{main-result} is immediate from Theorems \ref{main1} and \ref{main2}:
\begin{proof}[Proof of Theorem \ref{main-result}]
	From Theorem \ref{main1}
	$$\inf \liminf_{n\to \infty} \hat J^n(\tilde T^n) \ge J^*(0)$$ 
	where the infimum is taken over all admissible control sequences $\{\tilde T^n\}$.
	Also, with $\{T^n\}$ as in Theorem \ref{main2}, 
	$$\inf \liminf_{n\to \infty} \hat J^n(\tilde T^n) \le \liminf_{n\to \infty} \hat J^n( T^n) = J^*(0).$$
	Combining the above two inequalities, we have $\inf \liminf_{n\to \infty} \hat J^n(\tilde T^n) = J^*(0) = \lim_{n\to \infty} \hat J^n(T^n)$.
\end{proof}
Rest of this section is devoted to the  proof of Theorem \ref{main2}.  The proof  relies on three technical results: Theorem \ref{impthm1}, Theorem \ref{impthm2}
and Lemma \ref{lower-bound-est}, the proofs of which are postponed to Section \ref{impthms}. Throughout this section $\{T^n\}$ will denote 
the sequence of  control policies   in Definition \ref{propolicy} with some choice of threshold parameters.
  We begin with the following lemma
from  \cite{AA}. 
\begin{lemma}[Lemma 4.7 of \cite{AA}]\label{techlem}
Let $\{f_n\}$ and $\{g_n\}$ be sequences of functions in $D([0,\infty):\RR)$, and let $f$ and $g$ be continuous functions from $[0,\infty)$ to $\mathbb{R}$, such that $f_n\rightarrow f$, $g_n\rightarrow g$ in $D([0,\infty):\RR)$ as $n\rightarrow \infty$. Suppose that 
$\int_{[0,\infty)}e^{-\gamma t}1_{\{g(t)=0\}}dt=0\,.$
Let $\epsilon_n$ be a sequence of non-negative real numbers converging to $0$. Then for all $T > 0$, the following hold:
$$\int_0^Te^{-\gamma t}f_n(t)1_{\{g_n(t)\geq \epsilon_n\}}dt\rightarrow \int_0^Te^{-\gamma t}f(t)1_{\{g(t)\geq 0\}}dt\,\, \ \mbox{as}\,\,n\rightarrow \infty\,,$$
$$\int_0^Te^{-\gamma t}f_n(t)1_{\{g_n(t)\leq \epsilon_n\}}dt\rightarrow \int_0^Te^{-\gamma t}f(t)1_{\{g(t)\leq 0\}}dt\,\, \ \mbox{as}\,\,n\rightarrow \infty\,.$$
\end{lemma}

Given $c \in (1,\infty)$, let 
\begin{align}
\kappa(c) \doteq \max \biggl\{\frac{4\mu_1}{\mu_2},\frac{4}{c-1},\frac{c}{c-1}, 4\biggr\}.
\end{align}
For $n\in \N$, $\kappa \ge \kappa(c)$, and $t\geq 0$, define the events:
$$\bfA(n,t) \doteq \left\{{Q}^n_3(t)-\frac{\mu_2^n}{\mu_1^n}{Q}^n_1(t) < L^n\right\},$$ 
\begin{equation}
	\label{eq:eq816}\mathcal{E}_{\kappa}(n,t)\doteq \left\{\sup_{0\leq s \leq t}\hat{Q}^n_3(s)1_{\bfA(n,ns)}> \frac{\kappa(C^n-L^n+1)}{\sqrt{n}}\right\}
\cup \left\{\sup_{0\leq s \leq t}\hat{Q}^n_1(s)1_{\bfA(n,ns)^c}> \frac{\kappa(C^n-L^n+1)}{\sqrt{n}}\right\}.
\end{equation}

Proofs of the following two results are given in Section \ref{impthms}.
\begin{theorem}\label{impthm1}
There exist $\theta_i \in (0,\infty)$, $i=1,2,3,4$ and $n_0 \in \NN$ such that for the sequence $\{T^n\}$ with threshold parameters $l_0>1, c>1, g_0>0$ and with $\mathcal{E}_{\kappa}(n,t)$ defined as in \eqref{eq:eq816}
\begin{align}\label{prob1}
\PP(\cle_{\kappa}(n,t)) \le \theta_1 (nt+1)^2 e^{-\theta_2 nt} + \theta_3 (nt +1)^3 n^{-\theta_4 (c-1)l_0},
\end{align}
whenever $\kappa \ge \kappa(c)$, $n \ge n_0$ and $nt \ge 2$.
\end{theorem}
\begin{theorem}\label{impthm2}
 There exist $n_1 \in \NN$, $\epsilon \in (0,1)$, $\gamma_i =\gamma_i(c) \in (0,\infty)$, $i=1,2,3,4$, and $d = d(c) \in (0, \infty)$, such that for the sequence $\{T^n\}$ of control policies with threshold parameters $c>1$ and arbitrary $l_0 > 1$, $g_0 > 0$, 
	\begin{align}\label{prob2}
	\mathbb{P}&\biggl[\int_{[0,t)}1_{\bfB_d(n,s)}d\hat{I}^n_2(s)\neq 0\biggr]
	\le \gamma_1(nt+1)^2 e^{-\gamma_2 nt} + \gamma_3 (nt+1)^3 n^{-\gamma_4 l_0},
	\end{align}
	whenever $n\geq n_1$ and $nt\ge 2/\epsilon$,
	where 
$$ \bfB_d(n,t) \doteq \left\{\hat{Q}^n_2(t)\geq \frac{dl_0\log n}{\sqrt{n}}\right\}.$$
\end{theorem}
An immediate consequence of the above two theorems is the following.
\begin{corollary}\label{impcor}
Let $\theta_4$ be as in Theorem \ref{impthm1}. Let $c= 1+ \frac{4}{\theta_4}$, and let $\gamma_i \in (0,\infty)$,
	$i=1,2,3,4$ and $d$ be as in Theorem \ref{impthm2}.  Choose $\bar l \in (1,\infty)$ to be large enough so that $\gamma_4\bar{l}>3$. Fix $t\ge 0$. Then for all $l_0 \ge \bar l$, $g_0 >0$
	and sequence $\{T^n\}$ of control policies with threshold parameters $c, l_0, g_0$
	the probabilities in \eqref{prob1} and \eqref{prob2} tend to $0$ as $n\rightarrow \infty$ for all $\kappa \ge \kappa(c)$. In particular, as $n\rightarrow \infty$,
\begin{align*}
\hat{Q}^n_1(\cdot)1_{\bfA(n,n\cdot)^c}\Rightarrow 0,\; \ \
\hat{Q}^n_3(\cdot)1_{\bfA(n,n\cdot)}\Rightarrow 0,\; \  \
\int_{[0,\cdot]}1_{\bfB_d(n,s)}d\hat{I}^n_2(s)\Rightarrow 0.
\end{align*}
\end{corollary}
For the rest of this section we fix threshold parameters $c, l_0, g_0$ and constants $d$, $\kappa$ as in Corollary \ref{impcor}.
We will suppress $\kappa$ and $d$ in the notation for $\bfB_d(n,s), \cle_{\kappa}(n,s)$.
We next provide a lower bound for $W^n_1(t) - \sqrt{n}\Psi(W^n_2(t)/\sqrt{n}), t\ge 0$, which will ensure that $\hat W^n(t) \in G$ asymptotically. The proof is  given in Section \ref{impthms}.
\begin{lemma}\label{lower-bound-est}
There exist $C_1, C_2 \in (0,\infty)$ such that for $t\ge 0$, we have 
\begin{align}
& W^n_1(t) - \sqrt{n}\Psi(W^n_2(t)/\sqrt{n})  \nonumber\\
&\quad \ge  - \frac{1}{\mu^n_1}  \left|1- \frac{\mu_3\mu^n_2}{\mu^n_3\mu_2}\right| Q^n_1(t) -  \frac{1}{\mu^n_2}  \left|1- \frac{\mu_3\mu^n_2}{\mu^n_3\mu_2}\right| (2Q^n_2(t)+ 1) - C_1(C^n-L^n+2)+g_0 - C_2.
\label{eq:eq141}
\end{align}
\end{lemma}
The following asymptotic property of the sequence $\{T^n\}$ will play a key role. Recall $\bar T^*$ defined in \eqref{T-star}
and the fluid scaled processes $\bar T^n$ introduced in Section \ref{scaled}.
\begin{lemma}\label{implem}
As $n\to \infty$, 
$\bar{T}^n \Rightarrow \bar{T}^*$  and $\bar Q^n \Rightarrow 0$.
\end{lemma}
\begin{proof}  From the second expression in \eqref{scaled-workload}, we have
	 for $t\ge 0,$
\bes
\hatw^n_2(t) = \frac{\hatq^n_2(t)}{\mu^n_3} + \frac{\hatq^n_3(t)}{\mu^n_3} = \frac{\hatx^n_2(t)}{\mu^n_3} + \frac{\hatx^n_3(t)}{\mu^n_3} + \hati^n_2(t),
\ees
and thus 
{\small\be\label{sko-1}\ba
 \frac{\hatq^n_2(t)}{\mu^n_3} 1_{\bfB(n,t)} + \frac{\hatq^n_3(t)}{\mu^n_3} 
& = \frac{\hatx^n_2(t)}{\mu^n_3} + \frac{\hatx^n_3(t)}{\mu^n_3} - \frac{\hatq^n_2(t)}{\mu^n_3} 1_{\bfB(n,t)^c}  + \int_0^t 1_{\bfB(n,s)} d\hati^n_2(s)  +  \int_0^t 1_{\bfB(n,s)^c} d\hati^n_2(s). 
\ea\ee}
Note that the last term on the right hand side of \eqref{sko-1} is equal to $0$ when $t=0$, and is nondecreasing, and increases only when the term on the left hand side of \eqref{sko-1} is $0$. Therefore the left side of \eqref{sko-1} can be represented in terms of the one-dimensional Skorohod map $\Gamma$ (see  Proposition \ref{sp-property} (i)) and we have for $t\ge 0$, 
\be\label{exp1}
\hat{W}^n_2(t)=\Gamma\biggl(\frac{\hat{X}^n_2(\cdot)}{\mu^n_3}+\frac{\hat{X}^n_3(\cdot)}{\mu^n_3}-\frac{\hat{Q}^n_2(\cdot)}{\mu^n_3}1_{\bfB(n,\cdot)^c}+\int_0^\cdot 1_{\bfB(n,s)}d\hat{I}^n_2(s)\biggr)(t)
+\frac{\hat{Q}^n_2(t)}{\mu^n_3}1_{\bfB(n,t)^c}.
\ee
Thus using the Lipschitz continuity property of the Skorohod map (see  Proposition \ref{sp-property} (ii)), we have for $t\ge 0,$
\begin{align*}
 \sup_{0\le s \le t} \bar W^n_2(s) &
\le  \frac{1}{\sqrt{n}}\sup_{0\le s \le t} \Gamma\biggl(\frac{\hat{X}^n_2(\cdot)}{\mu^n_3}+\frac{\hat{X}^n_3(\cdot)}{\mu^n_3}-
\frac{\hat{Q}^n_2(\cdot)}{\mu^n_3}1_{\bfB(n,\cdot)^c}+\int_0^\cdot 1_{\bfB(n,s)}d\hat{I}^n_2(s)\biggr)(s)\\
& \quad + \frac{1}{\sqrt{n}} \sup_{0\le s \le t} \frac{\hat{Q}^n_2(s)}{\mu^n_3}1_{\bfB(n,s)^c} \\
& \le \frac{2}{\sqrt{n}} \sup_{0\le s \le t} \bigg|\frac{\hat{X}^n_2(s)}{\mu^n_3}+\frac{\hat{X}^n_3(s)}{\mu^n_3}-\frac{\hat{Q}^n_2(s)}{\mu^n_3}
1_{\bfB(n,s)^c}+\int_0^s 1_{\bfB(n,u)}d\hat{I}^n_2(u)\bigg| 
 + \frac{1}{\sqrt{n}} \cdot \frac{d l_0\log n}{\mu^n_3\sqrt{n}} \\
& \le  \frac{2}{\sqrt{n}} \sup_{0\le s \le t} \left|\frac{\hat{X}^n_2(s)}{\mu^n_3}+\frac{\hat{X}^n_3(s)}{\mu^n_3}\right| +   \frac{2d l_0\log n}{\mu^n_3{n}} + \frac{2}{\sqrt{n}} \int_0^t 1_{\bfB(n,u)}d\hat{I}^n_2(u) 
 +  \frac{d l_0\log n}{\mu^n_3 {n}}.
\end{align*}
From functional central limit theorem for renewal processes, $\hat A^n$ and $\hat S^n$ converge weakly to Brownian motions.
Combining this with the fact that $\bar T_i^n(t) \le t$ for $n \ge 1$ and $t \ge 0$, we have 
for all $t\ge 0,$
\be
\frac{2}{\sqrt{n}} \sup_{0\le s \le t} \left|\frac{\hat{X}^n_2(s)}{\mu^n_3}+\frac{\hat{X}^n_3(s)}{\mu^n_3}\right| \to 0, \ \mbox{in probability.} \label{eq:eq450}
\ee
Next from Corollary \ref{impcor} we have for $t\ge 0,$
$
\int_0^t 1_{\bfB(n,u)}d\hat{I}^n_2(u) \to 0$ in probability.
Finally, since
$
 \frac{d l_0\log n}{\mu^n_3{n}}  \to 0$ as $n\to \infty$, we have
\be\label{fluid-workload-1}
\sup_{0\le s \le t} \bar W^n_2(s) \to 0, \ \mbox{in probability.}
\ee
Next using the representation for $\hat{W}^n_1$ from \eqref{scaled-workload}, we have
\begin{align*}
&\hatw^n_1(t) -\Psi(\hat{W}^n_2(t))=\frac{\hat{X}^n_1(t)}{\mu^n_1}+\frac{\hat{X}^n_2(t)}{\mu^n_2}+\hat{I}^n_1(t)-\Psi(\hat{W}^n_2(t)),
\end{align*}
which implies 
\begin{align}
 (\hatw^n_1(t)-\Psi(\hat W^n_2(t)))1_{\bfC(n,t)^c } 
& = \frac{\hat{X}^n_1(t)}{\mu^n_1}+\frac{\hat{X}^n_2(t) }{\mu^n_2} - \Psi(\hatw^n_2(t)) \nonumber \\
& \quad -  (\hatw^n_1(t)-\Psi(\hat W^n_2(t)))1_{\bfC(n,t) } +\hat{I}^n_1(t),
\label{sko-2}\end{align}
where $\bfC(n,t) = \{\hatw^n_1(t)-\Psi(\hat W^n_2(t)) < g_0/\sqrt{n}\}$.
Note that the scheduling policy described in Definition \ref{propolicy} is such that $\hat{I}^n_1$ is equal $0$ when $t=0$, is non-decreasing, and increases only if the left hand side of the above equation is $0$. Thus using the characterizing property of the one dimensional Skorohod map we have for $t\ge 0,$
\be\label{exp3}\ba
\hatw^n_1(t) & = \Gamma\left( \frac{\hat{X}^n_1(\cdot)}{\mu^n_1}+\frac{\hat{X}^n_2(\cdot) }{\mu^n_2} - \Psi(\hatw^n_2(\cdot))  -  (\hatw^n_1(\cdot)-\Psi(\hat W^n_2(\cdot)))1_{\bfC(n,\cdot) }\right)(t) \\
& \quad + (\hatw^n_1(t)-\Psi(\hat W^n_2(t)))1_{\bfC(n,t)} + \Psi(\hatw^n_2(t)).
\ea\ee
By the Lipschitz property of $\Gamma$ and $\Psi$ we now have,
\begin{align}\label{estimate}\nonumber
 \sup_{0\le s\le t} \bar W^n_1(s)  & \le \frac{2}{\sqrt{n}}\sup_{0 \leq s \leq t}\left|\frac{\hat{X}^n_1(s)}{\mu^n_1}+\frac{\hat{X}^n_2(s)}{\mu^n_2}\right| + \frac{3}{\sqrt{n}}
\sup_{0 \leq s \leq t}\Psi(\hatw^n_2(s))\nonumber \\
& \quad+ \frac{3}{\sqrt{n}}\sup_{0 \leq s \leq t}\left|(\hatw^n_1(s)-\Psi(\hat W^n_2(s)))
1_{\bfC(n,s) }\right|\nonumber  \\
& \le \frac{2}{\sqrt{n}}\sup_{0 \leq s \leq t}\left|\frac{\hat{X}^n_1(s)}{\mu^n_1}+\frac{\hat{X}^n_2(s)}{\mu^n_2}\right| + \frac{6\mu_3}{\mu_2}\sup_{0 \leq s \leq t}\bar W^n_2(s) + \frac{3g_0}{n}.
\end{align}
As for \eqref{eq:eq450}, we have 
\bes
\frac{2}{\sqrt{n}}\sup_{0 \leq s \leq t}\left|\frac{\hat{X}^n_1(s)}{\mu^n_1}+\frac{\hat{X}^n_2(s)}{\mu^n_2}\right|\rightarrow 0\,\,\,\,\mbox{in probability}.
\ees
Using this along with \eqref{fluid-workload-1}, we now get
\be\label{fluid-workload-2}
\sup_{0\le s \le t} \bar{W}^n_1(s) \rightarrow 0 \ \ \mbox{as $n\to\infty$}.
\ee
From \eqref{fluid-workload-1},\eqref{fluid-workload-2} and \eqref{scaled-workload}, we have $\bar Q^n \Go 0$ as $n\to\infty$. 
Finally, using functional central limit theorem for renewal processes again, we have
$$n^{-1/2}\hat{A}^n_i(\cdot) \Rightarrow 0,\,\,\,\,n^{-1/2}\hat{S}^n_j(\bar{T}^n_j(\cdot))\Rightarrow 0, \,\,\,i=1,2\,\,j=1,2,3.$$ 
Hence, the conclusion follows from \eqref{scaled1} and the fact that $\bar Q^n \Rightarrow 0$.
\end{proof}

The following theorem gives certain uniform integrability properties that will be needed to  prove Theorem \ref{main2}.
\begin{theorem}\label{impthm3}
For $i=1,2$,
\begin{align}\label{uniint}
\limsup_{n \rightarrow \infty}\int_0^{\infty}e^{-\gamma t}\mathbb{E}\biggl[\sup_{0\leq s\leq t}\hat{W}^n_i(s)\biggr]^2dt < \infty\,,
\end{align}
and
\begin{align}\label{uniint1}
\limsup_{T\rightarrow \infty}\limsup_{n \rightarrow \infty}\int_T^{\infty}e^{-\gamma t}\mathbb{E}\biggl[\sup_{0\leq s\leq t}\hat{W}^n_i(s)\biggr]^2dt =0 \,.
\end{align}
\end{theorem}
\begin{proof}
	We only prove \eqref{uniint1}. The proof of \eqref{uniint} is similar. For the case $i=2$
	the proof of \eqref{uniint1} is identical to that of Theorem 4.11 of \cite{AA}, with one  modification. 
	Unlike \cite{AA}, here  we consider arrival and service processes that are general renewal processes rather than Poisson processes. Thus one cannot directly apply Doob's maximal inequality  to  bound $\mathbb{E}\bigl\{\displaystyle\sup_{0\leq s \leq t}\big|\hat{A}^n_i(s)\big|\bigr\}^2$ and $\mathbb{E}\bigl\{\displaystyle\sup_{0\leq s \leq t}\big|\hat{S}^n_i(s)\big|\bigr\}^2$. 
 However, these quantities can be bounded using Lorden's inequality  as in \cite{bellwill01} (see Equation (172) therein) and we omit the details. Consider now $i=1$.
	From \eqref{estimate} 
	\begin{align}\label{estimate1020}
	 \sup_{0\le s\le t} \hat W^n_1(s)  &  \le 2\sup_{0 \leq s \leq t}\left|\frac{\hat{X}^n_1(s)}{\mu^n_1}+\frac{\hat{X}^n_2(s)}{\mu^n_2}\right| + \frac{6\mu_3}{\mu_2}\sup_{0 \leq s \leq t}\hat W^n_2(s) + \frac{3g_0}{\sqrt{n}}.
	\end{align}
	
	The result for $i=1$ now follows from the above estimate along with the property 
	\eqref{uniint1} for $\hat W^n_2$.
\end{proof}

We now prove the main result of this section, Theorem \ref{main2}.

\begin{proof}[Proof of Theorem \ref{main2}.] Let $c, \bar l$ be as in Corollary \ref{impcor} and consider the 
sequence $\{T^n\}$ with threshold parameters $l_0 \ge \bar l$, $g_0 > 0$ and $c$ as above.	
 From Lemma \ref{implem} and functional central limit theorem for renewal processes, it follows that
\be\label{BM-conv2}
\hat{X}^n\Rightarrow {X},
\ee
where ${X}$ is as introduced below \eqref{BM-conv}. 
Define processes $\hat Z^n_i$, $i=1,2$ as
\begin{align*}
	\hat Z^n_1(t) &\doteq	\frac{\hat{X}^n_1(t)}{\mu^n_1}+\frac{\hat{X}^n_2(t) }{\mu^n_2} - \Psi(\hatw^n_2(t))  -  (\hatw^n_1(t)-\Psi(\hat W^n_2(t)))1_{\bfC(n,t)},\\
\hat Z^n_2(t) &\doteq \frac{\hat{X}^n_2(t)}{\mu^n_3}+\frac{\hat{X}^n_3(t)}{\mu^n_3}-\frac{\hat{Q}^n_2(t)}{\mu^n_3}1_{\bfB(n,t)^c}+\int_0^t 1_{\bfB(n,s)}d\hat{I}^n_2(s).
\end{align*}
Also let $\cli: D([0,\infty):\RR) \to D([0,\infty):\RR)$ be the identity map.
From \eqref{exp1} and \eqref{exp3}  we have for $t\ge 0,$
\begin{align}
	\hati^n_1(t) & = (\Gamma - \cli)(\hat Z^n_1)(t), \label{eq:eq1043}\\
	\hati^n_2(t) & =  (\Gamma - \cli)(\hat Z_2^n)(t) + \int_0^t1_{\bfB(n,s)}d\hat{I}^n_2(s).\label{eq:eq1043b}
\end{align}
Next, from \eqref{BM-conv}, 
\be\label{main-conv-1}
\frac{\hat{X}^n_2}{\mu^n_3}+\frac{\hat{X}^n_3}{\mu^n_3} \Go \frac{{X}_2}{\mu_3}+\frac{{X}_3}{\mu_3} = B_2, \;\;\; \frac{\hat{X}^n_1}{\mu^n_1}+\frac{\hat{X}^n_2}{\mu^n_2} \Go \frac{{X}_1}{\mu_1}+\frac{{X}_2 }{\mu_2} = B_1.
\ee
Applying \eqref{main-conv-1} and the third convergence statement in Corollary \ref{impcor} to \eqref{eq:eq1043b}, and recalling $(W^*_2, I^*_2)$ defined in \eqref{workload-general-solu} and 
\eqref{optimal-idle}, we see that
\begin{equation}
	\label{eq:eq1101}
	(\hat W^n_2, \hat I^n_2) \Rightarrow (W_2^*, I_2^*).
\end{equation}
Also from Lemma \ref{lower-bound-est}, for some $c_1, c_2\in (0,\infty),$
\begin{align}
& |\hatw^n_1(t)-\Psi(\hat W^n_2(t))|1_{\bfC(n,t) } \nonumber\\
& \le \frac{g_0}{\sqrt{n}} + \frac{\sqrt{n}}{\mu^n_1}  \left|1- \frac{\mu_3\mu^n_2}{\mu^n_3\mu_2}\right| \barq^n_1(t) +  \frac{\sqrt{n}}{\mu^n_2}  \left|1- \frac{\mu_3\mu^n_2}{\mu^n_3\mu_2}\right| (2\barq^n_2(t)+ \frac{1}{n}) + \frac{c_1(C^n-L^n+2)}{\sqrt{n}}+\frac{c_2}{\sqrt{n}}.\label{eq:eq1049}
\end{align}
From Assumption \ref{htc} we have
\[
\sqrt{n}\left|1- \frac{\mu_3\mu^n_2}{\mu^n_3\mu_2}\right|  \to \frac{|\mu_3b_3 - \mu_2b_2|}{\mu_3}, \]
and from Lemma \ref{implem}, $\barq^n \Go 0$.
Using these observations in \eqref{eq:eq1049}
we see that 
\be\label{main-conv-3}
(\hatw^n_1(\cdot)-\Psi(\hat W^n_2(\cdot)))1_{\bfC(n,\cdot) } \Go 0.
\ee
Applying the above result and \eqref{main-conv-1}, \eqref{eq:eq1101} to \eqref{eq:eq1043}, we now get
$(\hat W^n_1, \hat I^n_1) \Rightarrow (W_1^*, I_1^*).$
In fact we have shown $(\hat W^n, \hat I^n) \Rightarrow (W^*, I^*)$ in $D([0,\infty): \RR^4)$.  This proves part (i) of the theorem.
For the second part of the theorem, we observe that from Theorem \ref{impthm3} and first part of this theorem,
\begin{align}\label{conclusion1}
\int_0^{\infty}e^{-\gamma t}\mathbb{E}(\hat{W}^n_i(t))dt \rightarrow \int_0^{\infty}e^{-\gamma t}\mathbb{E}({W}^*_i(t))dt,\,\,\,\,i=1,2.
\end{align}
From \eqref{scaled-workload} we see that
$$\bfA(n,nt) = \left\{\hat{Q}^n_3(t)-\frac{\mu_2^n}{\mu_1^n}\hat{Q}^n_1(t) < \frac{L^n}{\sqrt{n}}\right\}
=  \left\{\mu_3^n\hat{W}^n_2(t)-\mu_2^n\hat{W}^n_1(t) < \frac{L^n}{\sqrt{n}}\right\}.$$
Combining \eqref{conclusion1} with Lemma \ref{techlem} and using the definition of $Q_2^*$ we now have 
exactly as in the proof of 
Theorem 4.2 of \cite{AA},
\begin{equation}
\int_0^{\infty}e^{-\gamma t}\mathbb{E} (\mu_2^n\hat{W}^n_1(t) 1_{\bfA(n,nt)^c}) dt +
\int_0^{\infty}e^{-\gamma t}\mathbb{E} (\mu_3^n\hat{W}^n_2(t) 1_{\bfA(n,nt)}) dt 
\to \int_0^{\infty}e^{-\gamma t}\mathbb{E} (Q^*_2(t)) dt.
\label{conclusion2}
\end{equation}
Finally using \eqref{conclusion1}, \eqref{conclusion2}, \eqref{scaled-workload} and Corollary \ref{impcor} it immediately follows that
\begin{align}
\int_0^{\infty}e^{-\gamma t}\mathbb{E}(\hat{Q}^n_i(t))dt \rightarrow \int_0^{\infty}e^{-\gamma t}\mathbb{E}({Q}^*_i(t))dt, \,\,\,i=1,2,3.
\end{align}
The result now follows from the definitions of $\hat{J}^n(T^n)$ and $J^*(0)$.  
\end{proof}


\section{Proofs of Theorems \ref{impthm1} and \ref{impthm2}, and Lemma \ref{lower-bound-est}}\label{impthms}
We begin with the following large deviations estimate for renewal processes, which will be extensively used in our proofs.
In the form stated below, the result can be found in \cite{bellwill05} (see also \cite{bellwill01}).
\begin{lemma}[Lemma 6.7 of \cite{bellwill05}]\label{LDP-renewal}
Let $\{\eta_i\}_{i=1}^\infty$ be a sequence of independent strictly positive random variables, where $\{\eta_i\}_{i=2}^\infty$ are identically distributed with finite mean $1/\nu$, and $\eta_1$ may have a different distribution from $\eta_i, i\ge 2$. Assume that there is a nonempty open neighborhood $\clo$ of $0\in\RR$ such that  
\be\label{eq:eq240}
\Lambda(l) \doteq \log \EE (e^{l\eta_i}) < \infty \ \mbox{for all $l\in \clo$ and $i\ge 2$.}
\ee
For each $n\in\NN$, let $\nu^n > 0$ be such that $\lim_{n\to\infty}\nu^n = \nu$, and for each $n\in\NN$ and $i=1, 2,3,\ldots$, let 
$
\eta^n_i = \frac{\nu}{\nu^n} \eta_i.
$
Given $\epsilon \in (0, \nu/2)$, let $n(\{\nu^n\}, \epsilon) \in \NN$ be such that when $n\ge n(\{\nu^n\}, \epsilon),$
\begin{align*}
 |\nu^n - \nu|  < \epsilon, \;\;\;\;\;
 \frac{\nu^n}{\nu} \frac{1}{\nu^n + \frac{\epsilon}{2}}  \le \frac{1}{\nu} \frac{1}{1 + \frac{\epsilon}{3\nu}} < \frac{1}{\nu},
\;\;\;\;\;   \frac{1}{\nu} \left( 1+ \frac{\epsilon}{2(\nu^n - \epsilon)} \right) \ge \frac{1}{\nu} \left( 1+ \frac{\epsilon}{2\nu}\right) > \frac{1}{\nu}.
\end{align*}
%
For each $n\in\NN$ and $t\ge 0$, let 
\bes
N^n(t) = \sup\left\{k \ge 0: \sum_{i=1}^k \eta_i^n \le t \right\}.
\ees
Let
\bes
\Lambda^*(x) = \sup_{l\in \RR} (lx - \Lambda(l)), \; x \in \RR
\ees
be the Legendre-Fenchel transform of $\Lambda$.
Let $\vt_1(\nu, \eps) = \Lambda^*\left(\frac{1}{\nu} \frac{1}{1 + \frac{\epsilon}{3\nu}}  \right)$
and $\vt_2(\nu, \eps) = \Lambda^*\left(\frac{1}{\nu} \left( 1+ \frac{\epsilon}{2\nu}\right)\right)$.
Then $\Theta_i(\nu,\eps) > 0$ for $i=1,2$, and for $n\ge n(\{\nu^n\}, \epsilon)$ and $t\ge 2/\epsilon$, 
\be\label{LDP-upper}\ba
\PP(N^n(t) > (\nu^n + \epsilon) t)  \le \exp\left(-[(\nu^n + \epsilon)t -1] \vt_1(\nu,\eps) \right)
 \le \exp\left(-(\nu t -1) \vt_1(\nu,\eps)   \right),
\ea\ee
and for $n \ge n(\{\nu^n\}, \epsilon)$ and $t\ge 0$, 
\be\label{LDP-lower}\ba
\PP(N^n(t) < (\nu^n - \epsilon) t)  &\le \exp\left(-(\nu^n - \epsilon)t \vt_2(\nu,\eps) \right) + \PP\left(\eta^n_1 > \frac{\epsilon}{2\nu^n} t \right)\\
& \le \exp\left(-(\nu  - 2\epsilon) t  \vt_2(\nu,\eps) \right) + \PP\left(\eta^n_1 > \frac{\epsilon}{2\nu^n} t \right),
\ea\ee
Furthermore, if $\eta_1$ has the same distribution as $\eta_i, i\ge 2$, then for each $n\ge 1, t\ge 0$, and $0< p_0\in \clo$ and any $m\in \NN$, 
\be\label{LDP-lower-2}
\PP\left( \max_{i=1, \ldots, m} \eta^n_i > \frac{\epsilon}{2\nu^n} t\right) \le m \exp\left(-\frac{p_0 \epsilon t}{2\nu} \right) \exp(\Lambda(p_0)).
\ee

\end{lemma}

For $k=1,2$ and $j=1,2,3$, denote by $\Lambda^*_{a,k}$ and $\Lambda^*_{s,j}$ the Legendre-Fenchel transform of $\Lambda_{a,k}$ and $\Lambda_{s,j}$, respectively, where the latter functions are as introduced in Assumption \ref{ass:ass3}.
These transforms along with Lemma \ref{LDP-renewal} applied with sequences $\{v_j(i)/\mu_j\}_{i \in \NN}$,
$\{u_k(i)/\lambda_k\}_{i\in \NN}$, $j=1,2,3$, $k=1,2$, will play a key role in the proofs of Theorems \ref{impthm1} and \ref{impthm2}.

 Let $\{\clg^n_1(t)\}_{t\ge 0}$ be the filtration in Section \ref{sec:schcont} associated with the sequence $\{T^n\}$. We next introduce 
a family of $\{\clg^n_1(t)\}$-stopping times as follows: For $k\in\NN,$
\be\label{stopping-times}\ba
\tau^n_0 & \doteq 0, \\
\tau^n_{2k-1} & \doteq \inf\left\{ t > \tau^n_{2k-2}: Q^n_3(t) -  \frac{\mu^n_2}{\mu^n_1}Q^n_1(t) \ge \bfl^n\right\}, \\
\tau^n_{2k} & \doteq \inf\left\{ t > \tau^n_{2k-1}: Q^n_3(t) -  \frac{\mu^n_2}{\mu^n_1}Q^n_1(t) < \bfl^n, Q^n_3(t) < \bfc^n -1\right\}.
\ea\ee

\subsection{Proof of Theorem \ref{impthm1}}
Throughout this section $\{T^n\}$ will denote 
the sequence of  control policies   in Definition \ref{propolicy} with some choice of threshold parameters.
Note that on $[\tau^n_{2k-2}, \tau^n_{2k-1}), k\in\NN$, $Q^n_3(\cdot)$ starts from below $\bfc^n -1$, and whenever $Q^n_3(\cdot)$ becomes larger than or equal to $\bfc^n -1$, Server 1 stops serving Buffer 2, which causes $Q^n_3(\cdot)$ to decrease. Thus we have 
\bes
Q^n_3(t) \le \bfc^n, \ \mbox{for all $t\in [\tau^n_{2k-2}, \tau^n_{2k-1})$, and $k\in\NN.$} 
\ees
Since $\kappa \ge \frac{c}{c-1}$, $C^n \le \kappa (C^n-L^n+1)$ and so we have for $s\in [\tau^r_{2k-2}, \tau^n_{2k-1}), k\in\NN,$
\be\label{eqn:1}
Q^n_3(s) \bfi_{\bfA(n,s)} \le \kappa(\bfc^n - \bfl^n +1). 
\ee
We also note that for $s\in [\tau^r_{2k-2}, \tau^n_{2k-1}), k\in\NN,$ $Q^n_3(s) - \frac{\mu^n_2}{\mu^n_1}Q^n_1(s) < \bfl^n$. Thus
\be\label{eqn:2}
Q^n_1(s) \bfi_{\bfA(n,s)^c} =0.
\ee
In view of \eqref{eqn:1} and \eqref{eqn:2}, to estimate $\PP(\cle(n,t))$, it suffices to focus on $[\tau^n_{2k-1}, \tau^n_{2k}), k\in\NN$. Note that each $\tau^n_{2k-1}$ corresponds to a up-crossing of $Q^n_3(\cdot) - \frac{\mu^n_2}{\mu^n_1}Q^n_1(\cdot)$ from smaller than $\bfl^n$ to become larger than or equal to $\bfl^n$. Each up-crossing requires at least one service completion of Server 1. Let 
\be\label{number}
\bfk^n = \lfloor n t (\lambda^n_1 + \lambda^n_2+2)\rfloor +1.
\ee 
Then we have that 
\bes\ba
\PP(\tau^n_{2\bfk^n-1} \le n t) & \le \PP(S^n_1(T^n_1(n t)) + S^n_2(T^n_2(n t)) \ge \bfk^n) \\
& \le \PP(A^n_1(n t) + A^n_2(n t) \ge \bfk^n) \\
& \le \PP(A^n_1(n t) + A^n_2(n t) \ge n t (\lambda^n_1 + \lambda^n_2+2)) \\
& \le  \PP(A^n_1(n t)  \ge n t (\lambda^n_1 + 1)) + \PP(A^n_2(n t) \ge n t ( \lambda^n_2+1)).
\ea\ees
From \eqref{LDP-upper} in Lemma \ref{LDP-renewal}, when $n\ge \max\{n(\{\lambda^n_1\},1), n(\{\lambda^n_2\}, 1)\}$ and $nt \ge 2$, we have 
\be\label{LDP-estimate-1}\ba
 \PP(\tau^n_{2\bfk^n-1} \le n t) \le \exp\left\{- (\lambda_1 nt -1) \Theta^{a,1}_1(\lambda_1,1) \right\} + \exp\left\{- (\lambda_2 nt -1) \Theta^{a,2}_1(\lambda_2,1)\right\},
\ea\ee
where for $\eps>0$, $\Theta^{a,i}_1(\lambda_i, \eps)$, $i=1,2$, are defined as $\Theta_1$ in Lemma \ref{LDP-renewal} on replacing $\Lambda^*$
with $\Lambda^*_{a,i}$.
Also,
\begin{align}
  &\PP(\tau^n_{2\bfk^n-1} > n t, \ \cle(n, t)) \nonumber\\
& \le  \sum_{k=1}^{\bfk^n} \PP\left(\sup_{s \in [\tau^n_{2k-1},  \tau^n_{2k}\wedge nt]} Q^n_3(s) \bfi_{\bfA(n,s)} > \kappa(\bfc^n - \bfl^n +1), \tau^n_{2k-1} \le nt \right) \nonumber\\
 &  \quad + \sum_{k=1}^{\bfk^n}  \PP\left( \sup_{s\in [\tau^n_{2k-1}, \tau^n_{2k}\wedge nt]} Q^n_1(s) \bfi_{\bfA(n,s)^c} > \kappa(\bfc^n - \bfl^n +1), \tau^n_{2k-1} \le nt\right) \nonumber\\
 & \le \sum_{k=1}^{\bfk^n} \PP\left(Q^n_3(s) > \kappa(\bfc^n - \bfl^n +1) \mbox{ and }  \bfA(n,s), \mbox{ for some $s \in [\tau^n_{2k-1},  \tau^n_{2k}\wedge nt]$ } \right) \nonumber\\
 &  \quad +  \sum_{k=1}^{\bfk^n} \PP\left(Q^n_1(s) > \kappa(\bfc^n - \bfl^n +1) \mbox{ and } \bfA(n,s)^c,  \mbox{ for some $s \in [\tau^n_{2k-1},  \tau^n_{2k}\wedge nt]$ } \right) \label{eq:eq519}
\end{align}
Next, for each $k\in\NN$, we define a sequence of stopping times within $[\tau^n_{2k-1}, \tau^n_{2k}).$ 
Let 
$$\bfH(n,t) = \left\{Q^n_1(t) \ge \frac{\mu^n_1}{\mu^n_2} (\bfc^n-\bfl^n+2)\right\}, \;\;\; \bfG(n,t) = \{Q^n_3(t) \ge (\bfc^n-1)\}.$$
For $l\in\NN,$ 
\be\label{substopping-times}\ba
\eta^{n,k}_{0} & \doteq \tau^n_{2k-1}, \\
\eta^{n,k}_{2l-1} & \doteq \tau^{n}_{2k} \wedge \inf\left\{t\ge \eta^{n,k}_{2l-2}:  
\left(\bfA(n,t) \cap \bfG(n,t)\right) 
\mbox{or} \ \left(\bfA(n,t)^c \cap \bfH(n,t)\right)\right\},\\
\eta^{n,k}_{2l} & = \tau^{n}_{2k} \wedge \inf\left\{t\ge \eta^{n,k}_{2l-1}:  \bfA(n,t)^c \cap \bfH(n,t)^c\right\}.
\ea\ee
%
Note that 
\begin{align}
  \{\omega: t\in [\eta^{n,k}_{2l-2}, \eta^{n,k}_{2l-1})\} &\subset
 \left\{\bfA(n,t)\cap \bfG(n,t) \right\}^c    \cap \left\{ \bfA(n,t)^c\cap \bfH(n,t)\right\}^c  \cap \left\{ \bfA(n,t)\cap\bfG(n,t)^c\right\}^c \nonumber \\
&\quad =    \bfA(n,t)^c\cap \bfH(n,t)^c. \label{no-effect}
\end{align}
Also, there exists $n_1\in\NN$ such that when $n\ge n_1$, we have $\frac{2\mu_1}{\mu_2}> \frac{\mu^n_1}{\mu^n_2}$. Since $\kappa \ge \frac{4\mu_1}{\mu_2}$, we  have from \eqref{no-effect} that when $n\ge n_1$,
\begin{align*}
\Big\{\sup_{s\in [\eta^{n,k}_{2l-2}, \eta^{n,k}_{2l-1})} Q^n_3(s)1_{\bfA(n,s)} > \kappa (C^n-L^n+1)\Big\} &= \emptyset ,\\
\Big\{\sup_{s\in [\eta^{n,k}_{2l-2}, \eta^{n,k}_{2l-1})} Q^n_1(s)1_{\bfA(n,s)^c} > \kappa (C^n-L^n+1)\Big\} &= \emptyset .
\end{align*}
Thus to  estimate $\PP(\cle(n,t))$, it suffices to consider the intervals $[\eta^{n,k}_{2l-1}, \eta^{n,k}_{2l}), l, k \in \NN$. We now estimate how many such subintervals are within $[\tau^n_{2k-1}, \tau^n_{2k}\wedge nt)$. We observe that each  $\eta^{n,k}_{2l-1}$ corresponds to at least one additional arrival to Buffer 1 or one additional job completion for Buffer 3. Recall $\bfk^n$ defined in \eqref{number}.  As in the proof of \eqref{LDP-estimate-1} we have, for all $n\ge \max\{n(\{\lambda^n_1\},1), n(\{\lambda^n_2\}, 1)\}$ and $nt > 2$, 
\begin{align}
 \PP(\eta^{n,k}_{2\bfk^n-1} \le n t) 
& \le \PP(A^n_1(n t) + S^n_3(T^n_3(nt)) \ge \bfk^n) \nonumber\\
& \le \PP (A^n_1(n t)+A^n_2(nt) \ge \bfk^n) \nonumber\\
& \le  
  \exp\left\{- (\lambda_1 nt -1) \Theta^{a,1}_1(\lambda_1,1) \right\} + \exp\left\{- (\lambda_2 nt -1) \Theta^{a,2}_1(\lambda_2,1)\right\}. \label{LDP-estimate-2}
\end{align}
Thus from \eqref{LDP-estimate-1}, \eqref{eq:eq519} and \eqref{LDP-estimate-2} we have for $n\ge \max\{n(\{\lambda^n_1\},1), n(\{\lambda^n_2\}, 1), n_1\}$ and $nt \ge 2$,
\begin{align*}
\PP(\cle(n,t)) & \le  \sum_{k=1}^{\bfk^n}\sum_{l=1}^{\bfk^n} \PP\left(Q^n_3(s) > \kappa(\bfc^n - \bfl^n +1) \ \mbox{and} \ \bfA(n,s) 
 \mbox{ for some $s \in [\eta^{n,k}_{2l-1},  \eta^{n,k}_{2l}\wedge nt)$ } \right) \\
&+ \sum_{k=1}^{\bfk^n} \sum_{l=1}^{\bfk^n} \PP\left(Q^n_1(s) > \kappa(\bfc^n - \bfl^n +1) \ \mbox{and} \ \bfA(n,s)^c
\mbox{ for some $s \in [\eta^{n,k}_{2l-1},  \eta^{n,k}_{2l}\wedge nt)$ } \right) \\
& + (\bfk^n+1) \left(\exp\left\{- (\lambda_1 nt -1) \Theta^{a,1}_1(\lambda_1,1) \right\} + \exp\left\{- (\lambda_2 nt -1) \Theta^{a,2}_1(\lambda_2,1)\right\}\right).
\end{align*}
Next, on the set
$\{Q^n_3(s) > \kappa(\bfc^n - \bfl^n +1) \ \mbox{and} \ \bfA(n,s)\}$
\begin{align*}
Q^n_1(s) & > \frac{\mu^n_1}{\mu^n_2}(Q^n_3(s) - \bfl^n) > \frac{\mu^n_1}{\mu^n_2}(\kappa(\bfc^n-\bfl^n+1)-\bfl^n) \\
& = \frac{\mu^n_1}{\mu^n_2}\frac{3\kappa}{4}(\bfc^n-\bfl^n+1) + \frac{\mu^n_1}{\mu^n_2}\left(\frac{\kappa}{4}(\bfc^n-\bfl^n+1)-\bfl^n\right) \\
& \ge \frac{\mu^n_1}{\mu^n_2}\frac{3\kappa}{4}(\bfc^n-\bfl^n+1) + \frac{\mu^n_1}{\mu^n_2}\left(\frac{\kappa}{4}\left(c-1\right)-1\right)l_0\log n \\
& > \frac{\mu^n_1}{\mu^n_2}\frac{3\kappa}{4}(\bfc^n-\bfl^n+1), 
\end{align*}
where the last inequality follows on noting that $\kappa > \frac{4}{c-1}$.
Let $n_2\in\NN$ be such that when $n\ge n_2$, we have $\frac{2\mu_1}{3\mu_2} \le \frac{\mu_1^n}{\mu_2^n}$. Letting $\kappa' = \min\{\kappa, \frac{\kappa}{2} \frac{\mu_1}{\mu_2}\}$, we  have that, when $n\ge n_2$,
\bes
Q^n_1(s) > \kappa'(\bfc^n-\bfl^n+1).
\ees
Note that since $\kappa \ge \max\{\frac{2\mu_1}{\mu_2}, 4\}$,  we have that
\be\label{kappa-prime}
\kappa' \ge \frac{2\mu_1}{\mu_2}.
\ee
Since $\kappa \ge \kappa'$, for $n\ge \max\{n(\{\lambda^n_1\},1), n(\{\lambda^n_2\}, 1), n_1, n_2\}$ and $nt \ge 2$,
\be\label{estimate-final-2}\ba
\PP(\cle(n,t)) & \le 2\sum_{k=1}^{k^n} \sum_{l=1}^{k^n} \PP\biggm(Q^n_1(s) > \kappa'(\bfc^n - \bfl^n +1) 
 \mbox{ for some $s \in [\eta^{n,k}_{2l-1},  \eta^{n,k}_{2l}\wedge nt)$ } \biggm)\\
& \quad + (\bfk^n+1)\left(\exp\left\{- (\lambda_1 nt -1) \Theta^{a,1}_1(\lambda_1,1) \right\} + \exp\left\{- (\lambda_2 nt -1) \Theta^{a,2}_1(\lambda_2,1)\right\}\right).
\ea\ee
Next note that
\begin{align}
   &\left\{ \bfA(n,s)^c \cap \bfH(n,s)^c\right\}^c 
 \cap \left\{ \bfA(n,s)\cap \bfG(n,s)^c\right\}^c \nonumber\\
&\quad =  \left\{ \bfA(n,s)\cap \bfG(n,s)\right\} 
  \cup \left\{ \bfH(n,s)\cap \bfG(n,s)\right\} 
\cup  \left\{\bfH(n,s)\cap \bfA(n,s)^c\right\}. \label{effect-1}
\end{align}
From this, we see that for $s\in [\eta^{n,k}_{2l-1}, \eta^{n,k}_{2l})$, $Q^n_1(s) \ge \frac{\mu^n_1}{\mu^n_2}(\bfc^n - \bfl^n -1) > 0$, and Server 1 works on Buffer 1 continuously. 
Also, from \eqref{no-effect}, we have $Q^n_1(\eta^{n,k}_{2l-1}-) < \frac{\mu^n_1}{\mu^n_2}(\bfc^n - \bfl^n +2)$, and so $Q^n_1(\eta^{n,k}_{2l-1}) < \lfloor\frac{\mu^n_1}{\mu^n_2}(\bfc^n - \bfl^n +2)\rfloor+1$.
Using this fact, the probability in \eqref{estimate-final-2}
can be estimated by analyzing a $GI/GI/1$ queue with interarrival times $\tilde u^n_1, u^n_1(k), k=2, 3, \ldots$, service times $\tilde v^n_1, v^n_1(k), k=2, 3, \ldots$, and initial queue length $\lfloor\frac{\mu^n_1}{\mu^n_2}(\bfc^n - \bfl^n +2)\rfloor+1$,
where $\tilde u_1^n$ and $\tilde v^n_1$ are residual interarrival and service times at time $\eta^{n,k}_{2l-1}$ in Buffer 1. Let $\{\clq^n(t)\}_{t\ge 0}$ be this $GI/GI/1$ queue length process and define
\bes
\beta^n = \inf\left\{s > 0: \clq^n(s) < \frac{\mu^n_1}{\mu^n_2}(\bfc^n - \bfl^n -1) \right\}.
\ees
Noting that when $n\ge n_1$, $\frac{2\mu_1}{\mu_2} \ge \frac{\mu^n_1}{\mu^n_2}$ and using \eqref{kappa-prime}, we  see that
\be\label{larger-event}\ba
& \PP\left(Q^n_1(s) > \kappa'(\bfc^n - \bfl^n +1) \ \mbox{for some $s \in [\eta^{n,k}_{2l-1},  \eta^{n,k}_{2l}\wedge nt]$} \right) \\
& \le \PP\left(  \clq^n(s) > \kappa'(\bfc^n - \bfl^n +1) 
\mbox{ for some $s \in [0, \beta^n]$ }\right). 
\ea\ee
Let $\epsilon \in (0, (\mu_1-\lambda_1)/4)$, and define 
\bes
t^n = \frac{\left\lfloor \frac{\mu^n_1}{\mu^n_2}(\bfc^n - \bfl^n -1) -2\right\rfloor}{2(\lambda_1^n + \epsilon)}
\ees
Let $\tilde A^n$ and $\tilde S^n$ be the arrival and service processes of the $GI/GI/1$ queue, and consider
\bes
\clh^n \doteq \left\{ \tilde A^n(t^n) < (\lambda^n_1 + \epsilon) t^n, \ \tilde S^n(t^n) > (\mu^n_1 - \epsilon) t^n \right\}.
\ees
From \eqref{LDP-upper} and \eqref{LDP-lower} in Lemma \ref{LDP-renewal}, for $n \ge \max\{n(\{\lambda^n_1\}, \epsilon), n(\{\mu^n_1\}, \epsilon)\}$, and $t^n \ge 2/\epsilon$, we have 
\bes\ba
 \PP((\clh^n)^c) & \le \PP( \tilde A^n(t^n) \ge (\lambda^n_1 + \epsilon) t^n) + \PP(\tilde S^n(t^n) \le (\mu^n_1 - \epsilon) t^n) \\
 & \le \exp\left(-(\lambda_1 t^n -1) \Theta^{a,1}_1(\lambda_1,\eps) \right) 
 + \exp\left(-(\mu_1  - 2\epsilon) t^n  \Theta^{s,1}_2(\mu_1,\eps) \right) + \PP\left(\tilde v^n_1 > \frac{\epsilon t^n}{2\mu_1^n}  \right),
\ea\ees
where $\Theta^{s,1}_2$ is defined as $\Theta_2$ in Lemma \ref{LDP-renewal} by replacing $\Lambda^*$ with $\Lambda^*_{s,1}$.

Next using \eqref{LDP-upper} and \eqref{LDP-lower-2} in Lemma \ref{LDP-renewal}, 
\be\label{lower-prob-estimate}\ba
\PP\left(\tilde v^n_1 > \frac{\epsilon t^n}{2\mu_1^n}  \right)  
& \le \PP \left(\max_{k=1, \ldots, S^n_1(nt)} v^n_1(k) > \frac{\epsilon t^n}{2\mu_1^n}  \right) \\
& \le \PP \left(S^n_1(nt) > (\mu^n_1 + \epsilon) nt\right)+ \PP \left(\max_{k=1, \ldots, \lfloor (\mu^n_1 + \epsilon) nt \rfloor} v^n_1(k) > \frac{\epsilon t^n}{2\mu_1^n}  \right) \\
& \le \exp\left(-(\mu_1 nt -1) \Theta^{s,1}_1(\mu_1,\eps) \right)  + (\mu^n_1 + \epsilon) nt \exp\left(-\frac{p_0 \epsilon t^n}{2\mu_1} \right) \exp(\Lambda_{s,1}(p_0)),
\ea\ee
where  $\Theta^{s,1}_1$ is defined as $\Theta_1$ in Lemma \ref{LDP-renewal} by replacing $\Lambda^*$ with $\Lambda^*_{s,1}$
and $0 < p_0\in \clo$. Thus for $n \ge \max\{n(\{\lambda^n_1\}, \epsilon), n(\{\mu^n_1\}, \epsilon)\}$, and $t^n \ge 2/\epsilon$, we have 
\be\label{estimate-final-3}\ba
\PP((\clh^n)^c) & \le \exp\left(-(\lambda_1 t^n -1) \Theta^{a,1}_1(\lambda_1,\eps)\right)
+ \exp\left(-(\mu_1  - 2\epsilon) t^n  \Theta^{s,1}_2(\mu_1,\eps) \right) \\
& \quad + \exp\left(-(\mu_1 nt -1) \Theta^{s,1}_1(\mu_1,\eps) \right)  + (\mu^n_1 + \epsilon) nt \exp\left(-\frac{p_0 \epsilon t^n}{2\mu_1} \right) \exp(\Lambda_{s,1}(p_0)).
\ea\ee
Next note that on $\mathcal{H}^n$,
\bes\ba
\clq^n(t^n) & = \clq^n(0) + \tilde A^n(t^n) - \tilde S^n(t^n) \\
& < \frac{\mu^n_1}{\mu^n_2}(\bfc^n - \bfl^n +2)+1 + (\lambda^n_1-\mu^n_1 + 2\epsilon) t^n \\
& = \frac{\mu^n_1}{\mu^n_2}(\bfc^n - \bfl^n +2)+1 - (\mu^n_1-\lambda^n_1 - 2\epsilon)\frac{\left\lfloor \frac{\mu^n_1}{\mu^n_2}(\bfc^n - \bfl^n -1) -2\right\rfloor}{2(\lambda_1^n + \epsilon)}.
\ea\ees
Recall that $\epsilon \in (0, (\mu_1-\lambda_1)/4)$. Thus there exists $n_3\in \NN$ such that when $n\ge n_3$, we have 
\bes\ba
& \clq^n(t^n) - \frac{\mu^n_1}{\mu^n_2}(\bfc^n - \bfl^n -1) \\
& < \frac{3\mu^n_1}{\mu^n_2} + 1 - (\mu^n_1-\lambda^n_1 - 2\epsilon)\frac{\left\lfloor \frac{\mu^n_1}{\mu^n_2}(\bfc^n - \bfl^n -1) -2\right\rfloor}{2(\lambda_1^n + \epsilon)} \\
& < 0. 
\ea\ees
Since $\clq^n(t) \ge \frac{\mu^n_1}{\mu^n_2}(\bfc^n - \bfl^n -1) -1$ for all $t\in [0,\beta^n]$,
we see that for all $n \ge n_3$, on $\mathcal{H}^n$, $t^n \ge \beta^n$.
 Next, on this set, for $s\in [0, \beta^n]$ and $n \ge n_3$, we have 
\bes\ba
\clq^n(s) & = \clq^n(0) + \tilde A^n(s) - \tilde S^n(s) \\
& \le \clq^n(0) + \tilde A^n(s) \\
& < \frac{\mu^n_1}{\mu^n_2}(\bfc^n - \bfl^n +2)+1 + (\lambda^n_1 + \epsilon) t^n \\
& \le  \frac{\mu^n_1}{\mu^n_2}(\bfc^n - \bfl^n +2) + \frac{\mu^n_1}{2\mu^n_2}(\bfc^n - \bfl^n -1)\\
& = \frac{3\mu^n_1}{2\mu^n_2}(\bfc^n - \bfl^n +1).
\ea\ees
Recall from \eqref{kappa-prime} that $\kappa' > \frac{2\mu_1}{\mu_2}$, and so there exists $n_4\ge n_3$ such that when $n\ge n_4$, we have $\kappa' >  \frac{3\mu^n_1}{2\mu^n_2}$, and so for such $n$, on $\mathcal{H}^n$,
\bes
\clq^n(s) < \kappa' (\bfc^n - \bfl^n +1) \ \ \mbox{for all $s\in [0, \beta^n]$}.
\ees
Let $\cli^n =\{\clq^n(s) > \kappa'(\bfc^n - \bfl^n +1) 
\mbox{ for some $s \in [0, \beta^n]$ } \}$.
Then for  $n \ge n_4$, $\PP(\clh^n \cap \cli^n) =0$, and so for $n\ge \max\{n(\{\lambda^n_1\},\epsilon), n(\{\lambda^n_2\}, \epsilon), n_1, n_4\}$, the probability in \eqref{larger-event} is bounded by $\PP((\clh^n)^c).$ Combining this with \eqref{estimate-final-2} and \eqref{estimate-final-3}, we finally have for sufficiently large $n$  and $t\ge 2/n$, 
\be\label{final-est-1}\ba
\PP(\cle(n,t)) & \le (\bfk^n+1)\left(\exp\left\{- (\lambda_1 nt -1) \Theta^{a,1}_1(\lambda_1,1) \right\} + \exp\left\{- (\lambda_2 nt -1) \Theta^{a,2}_1(\lambda_2,1)\right\}\right)\\
& \quad + 2 (\bfk^n)^2 \exp\left(-(\lambda_1 t^n -1) \Theta^{a,1}_1(\lambda_1,\eps) \right) 
 + 2 (\bfk^n)^2\exp\left(-(\mu_1  - 2\epsilon) t^n  \Theta^{s,1}_2(\mu_1,\eps) \right) \\
& \quad + 2 (\bfk^n)^2 \exp\left(-(\mu_1 nt -1) \Theta^{s,1}_1(\mu_1,\eps) \right) 
 + 2 (\bfk^n)^2 (\mu^n_1 + \epsilon) nt \exp\left(-\frac{p_0 \epsilon t^n}{2\mu_1} \right) \exp(\Lambda_{s,1}(p_0)).
\ea\ee
Thus we have shown that, there exist $\theta_1, \theta_2, \theta_3, \theta_4 \in (0,\infty)$ such that for sufficiently large $n$  and $t\ge 2/n$, we have 
\bes
\PP(\cle(n,t)) \le \theta_1 (nt+1)^2 \exp\{-\theta_2 (nt + 1)\} + \theta_3 (nt +1)^3 n^{-\theta_4 (c-1)l_0}. 
\ees
\qed

\subsection{Proof of Theorem \ref{impthm2}}
Throughout this section $\{T^n\}$ will denote 
the sequence of  control policies   in Definition \ref{propolicy} with some choice of threshold parameters.

We now introduce a sequence of $\{\clg^n_1(t)\}_{t\ge 0}$ stopping times, which will be  used in the proof of Theorem \ref{impthm2}. First recall the stopping times $\tau^n_l, l\in \NN_0$ in \eqref{stopping-times}. As in \cite{AA}, for $k\in\NN$, we define a sequence of stopping times within $[\tau^n_{2k-2}, \tau^n_{2k-1})$. For $d, l_0 > 0$, let $\bfD^n = d l_0\log n$. For $m\in\NN,$
\be\label{substopping-timesb}\ba
\tilde\eta^{n,k}_0 & \doteq \tau^n_{2k-2}, \\
\tilde\eta^{n,k}_{2m-1} & \doteq \min\left\{\tau^n_{2k-1}, \inf\{s>\tilde\eta^{n,k}_{2m-2}| Q^n_3(s) \ge \bfc^n-1\} \right\} \\
\tilde\eta^{n,k}_{2m} & \doteq \min\left\{\tau^n_{2k-1}, \inf\{s>\tilde\eta^{n,k}_{2m-1}| Q^n_3(s) < \bfc^n-1\} \right\} \\
\beta^{n,k}_m & \doteq \min\left\{\tilde\eta^{n,k}_{2m-1}, \inf\{s>\tilde\eta^{n,k}_{2m-2}| Q^n_2(s) \ge \frac{\bfD^n}{4}\} \right\}. 
\ea\ee
Recall the multi-parameter filtration $\{\clg^n(a,b): a\in\NN^2, b \in \NN^3\}$ and multi-parameter stopping times introduced in Section \ref{sec:schcont}.
Lemma \ref{multi-2} below is taken from \cite{bellwill01}.
\begin{lemma}[Lemma 7.6 of \cite{bellwill01}]\label{multi-2}
Let $\bft = (\bft_1, \bft_2, \bft_3, \bft_4, \bft_5)'$ be a $\{\clg^n(a, b): a\in \NN^2, b\in\NN^3\}$ multi-parameter stopping time. Then 
\bes
(u^n_1(\bft_1), u^n_2(\bft_2), v^n_1(\bft_3), v^n_2(\bft_4), v^n_3(\bft_5))' \in \clg^n_\bft,
\ees
and on $\{\bft \in \NN^5\}$ the conditional distribution of $\{(u^n_1(\bft_1+k), u^n_2(\bft_2+k), v^n_1(\bft_3+k), v^n_2(\bft_4+k), v^n_3(\bft_5+k))': k\in \NN\}$ given $\clg^n_\bft$ is the same as the distribution of $\{(u^n_1(k), u^n_2(k), v^n_1(k), v^n_2(k),$
$ v^n_3(k))': k\in \NN\}$.
\end{lemma}
The following lemma follows along the lines of Lemma 7.5 of \cite{bellwill01}. The proof is omitted.
\begin{lemma}\label{multi-1}
For $n, k, m\in \NN$, 
\bes
\bft^{n,k}_m \doteq (A^n_1(\beta^{n,k}_m), A^n_2(\beta^{n,k}_m), S^n_1(\beta^{n,k}_m), S^n_2(\beta^{n,k}_m), S^n_3(\beta^{n,k}_m))'
\ees
is a $\{\clg^n(a, b): a\in \NN^2, b\in\NN^3\}$ multiparameter stopping time. 
\end{lemma}

\begin{proof}[Proof of Theorem \ref{impthm2}]  From the definition of  stopping times  in \eqref{stopping-times} we see that, for $s\in [\tau^n_{2k-1}, \tau^n_{2k}), k\in\NN,$  
\bes
Q^n_3(s) - \frac{\mu^n_2}{\mu^n_1}Q^n_1(s) \ge \bfl^n, \ \mbox{or} \ Q^n_3(s) \ge \bfc^n -1,
\ees
and therefore for such $s$, $Q^n_3(s) > 0.$ Consequently, since Server 2 does not idle unless Buffer 3 is empty, we have
\bes
\int_{[\tau^n_{2k-1}, \tau^n_{2k})} \bfi_{\{Q^n_2(s)\ge \bfD^n\}} d I^n_2(s) = 0. 
\ees
Thus we only need to consider time intervals $[\tau^n_{2k-2}, \tau^n_{2k-1}), k\in\NN.$ 
We observe that for $s\in [\tilde\eta^{n,k}_{2m-1}, \tilde\eta^{n,k}_{2m}),$ $Q^n_3(s) \ge \bfc^n -1 >0.$ Thus $Q^n_3(s) =0$ is possible only for $s\in[\tilde\eta^{n,k}_{2m-2}, \tilde\eta^{n,k}_{2m-1}).$ We next estimate how many such subintervals are within $[\tau^n_{2k-2}\wedge nt, \tau^n_{2k-1}\wedge nt)$. Each  $\tilde\eta^{n,k}_{2m-1}$ corresponds to at least one additional job completion for Buffer 3. Let $\bfk^n_1 = \lfloor (\mu^n_2+1)nt \rfloor + 1$. From \eqref{LDP-upper} in Lemma \ref{LDP-renewal}, we have when $n \ge n(\{\mu^n_2\}, 1)$ and $nt >2$, 
\be\label{LDP-estimate-3}\ba
& \PP(\tilde\eta^{n,k}_{2\bfk^n_1-1} \le nt) \le \PP(S^n_2(nt) \ge \bfk^n_1)  \le \exp\left(-(\mu_2 nt -1) 
\Theta_1^{s,2}(\mu_2,1)\right).
\ea\ee
Let 
\bes
\clr(n,t) \doteq \left\{ \int_{[0, nt)} \bfi_{\{Q^n_2(s)\ge \bfD^n\}} d I^n_2(s) \neq 0 \right\}.
\ees 
Then from  \eqref{LDP-estimate-3}, we have when $n\ge \max\{ n(\{\mu_2^n\},1)\}$, and $nt \ge 2$, 
\begin{align}
\PP(\clr(n,t), \tau^n_{2\bfk^n-1}> nt)  &\le
 \sum_{k=1}^{\bfk^n} \PP\left(\int_{[\tau^n_{2k-2}, \ \tau^n_{2k-1}\wedge nt)} \bfi_{\{Q^n_2(s)\ge \bfD^n\}} d I^n_2(s) \neq 0 \right) \nonumber \\
& \le \sum_{k=1}^{\bfk^n} \sum_{m=1}^{\bfk_1^n} \PP\left(\int_{[\tilde\eta^{n,k}_{2m-2},  \tilde\eta^{n,k}_{2m-1}\wedge nt)} \bfi_{\{Q^n_2(s)\ge \bfD^n\}} d I^n_2(s) \neq 0 \right)  \label{estimate-new} \\
 &+ \bfk^n \exp\left(-(\mu_2 nt -1) \Theta_1^{s,2}(\mu_2,1) \right). \nonumber
\end{align}
By the definition of $\beta^{n,k}_m$, for $s\in [\tilde\eta^{n,k}_{2m-2}, \beta^{n,k}_m)$, $Q^n_2(s) < \frac{\bfD^n}{4} $, and so a typical summand in \eqref{estimate-new} is equal to 
\begin{align}
 & \ \PP\left(\beta^{n,k}_m > \tilde\eta^{n,k}_{2m-2}, \ \int_{[\beta^{n,k}_m, \ \tilde\eta^{n,k}_{2m-1}\wedge nt)} \bfi_{\{Q^n_2(s)\ge \bfD^n\}} d I^n_2(s) \neq 0 \right)  \label{estimate-new-1}  \\ 
  & + \  \PP\left(\beta^{n,k}_m = \tilde\eta^{n,k}_{2m-2}, \ \int_{[\beta^{n,k}_m, \ \tilde\eta^{n,k}_{2m-1}\wedge nt)} \bfi_{\{Q^n_2(s)\ge \bfD^n\}} d I^n_2(s) \neq 0 \right).  \label{estimate-new-2} 
\end{align}
In the following, we estimate the probabilities in \eqref{estimate-new-1} and \eqref{estimate-new-2} separately. We will use the following constants:
\be\label{imp-parameters}\ba
\epsilon_1 & \in \left(0, \min\left\{1, \frac{\mu_2-\mu_3}{8},  \frac{\mu_2}{2}, \frac{\mu_3}{2}\right\} \right), \\
K & = 32\mu_2+\frac{4\mu_2(\mu_2-\mu_3)}{\mu_3}, \\
d & = \frac{2 c K}{\mu_2 - \mu_3}, \\
\theta & =  \min\left\{\frac{1}{2}, \frac{\mu_2-\mu_3}{32 c \mu_3} \right\}.
\ea\ee
There exists $n_1\in\NN$ such that when $n\ge n_1$, also noting that $\lambda_2=\mu_3$ from Assumption \ref{htc}, we have 
\begin{align*}
& \mu^n_2 + \epsilon_1 < 2\mu_2, \ \mu^n_3 + \epsilon_1 < 2 \mu_3, \ \lambda^n_2 + \epsilon_1 < 2 \lambda_2, \ \mu^n_2 - \mu^n_3 - 2 \epsilon_1 > \frac{\mu_2-\mu_3}{2}, \\
& \mu^n_2 - \lambda^n_2 + 2 \epsilon_1 < 2(\mu_2 - \lambda_2).
\end{align*}
For the rest of the proof we  assume $n\ge n_1$. Define for $s\ge 0$, 
\begin{align*}
\tilde A^n_2(s)  = A^n_2(\beta^{n,k}_m + s) - A^n_2(\beta^{n,k}_m), \;\;
\tilde S^n_i(s)  = S^n_i(T^n_i(\beta^{n,k}_m) + s) - S^n_i(T^n_i(\beta^{n,k}_m)), \; i = 2,3.
\end{align*}
Now denote by $\clb^{n,k}_m$ the event in \eqref{estimate-new-1}, i.e.,
\begin{align*}
\clb^{n,k}_m & \doteq \left\{\beta^{n,k}_m > \tilde\eta^{n,k}_{2m-2}, \ \int_{[\beta^{n,k}_m, \ \tilde\eta^{n,k}_{2m-1}\wedge nt)} \bfi_{\{Q^n_2(s)\ge \bfD^n\}} d I^n_2(s) \neq 0 \right\}. \nonumber 
\end{align*}
We claim that, with $\dmnk = \{\sup_{0\le s \le \frac{\theta\bfD^n}{K} } \left| Q^n_2(\beta^{n,k}_m +s) - \frac{\bfD^n}{4}\right| \ge \frac{1}{2} \frac{\bfD^n}{4}\}$,
\bes\ba
 \PP\left( \dmnk \cap \ \clb^{n,k}_m\right) 
 \le \exp\left\{- ( \frac{\mu_2\theta\bfD^n}{K} -1) \Theta_1^{s,2}(\mu_2, \eps_1) \right\}  
+ \exp\left\{-( \frac{\lambda_2\theta\bfD^n}{K} -1) \Theta_1^{a,2}(\lambda_2, \eps_1) \right\}.
\ea\ees
To see this,  note that 
\begin{align}
 \PP\left( \dmnk \cap \clb^{n,k}_m\right) 
& \le \PP\left(Q^n_2(\beta^{n,k}_m +s) \le \frac{1}{2} \frac{\bfD^n}{4} \ \mbox{for some $s\in [0, \frac{\theta\bfD^n}{K} ]$}, \ \clb^{n,k}_m\right)  \nonumber\\
& \quad + \PP\left(Q^n_2(\beta^{n,k}_m +s) \ge \frac{3}{2} \frac{\bfD^n}{4} \ \mbox{for some $s\in [0, \frac{\theta\bfD^n}{K} ]$}, \ \clb^{n,k}_m\right).  \nonumber
\end{align}
By the definition of $\beta^{n,k}_m$, we have $Q^n_2(\beta^{n,k}_m)\ge \frac{\bfD^n}{4}$. For $Q^n_2(\beta^{n,k}_m + s) \le \frac{1}{2} \frac{\bfD^n}{4}$ for some $s\in [0,\frac{\theta\bfD^n}{K}]$, we need at least $\frac{1}{2} \frac{\bfD^n}{4}$ service completions for jobs in Buffer 2 in $\frac{\theta\bfD^n}{K}$ units of time. On the other hand, noting that $\beta^{n,k}_m > \tilde\eta^{n,k}_{2m-2}$, we have $Q^n_2(\beta^{n,k}_m-)<\frac{\bfD^n}{4}$ and so $Q^n_2(\beta^{n,k}_m)<\frac{\bfD^n}{4}+1$. In order for $Q^n_2(\beta^{n,k}_m + s)$ to be greater than or equal to  $\frac{3}{2} \frac{\bfD^n}{4}$ for some $s\in [0,\frac{\theta\bfD^n}{K}]$, we need at least $\frac{1}{2} \frac{\bfD^n}{4} -1$ arrivals to Buffer 2 in $\frac{\theta\bfD^n}{K}$ time. Therefore, we have 
\begin{align*}
 \PP\left( \dmnk \cap \clb^{n,k}_m\right) 
 \le \PP\left(\tilde S^n_2(\frac{\theta\bfD^n}{K}) \ge \frac{1}{2} \frac{\bfD^n}{4}, \beta^{n,k}_m \le nt \right) + \PP\left(\tilde A^n_2(\frac{\theta\bfD^n}{K}) \ge \frac{1}{2} \frac{\bfD^n}{4}-1, \beta^{n,k}_m \le nt \right).
\end{align*}
Next note that 
\be\label{time-estimate}\ba
\frac{1}{2} \frac{\bfD^n}{4}  & > \frac{2\mu_2 }{K}\bfD^n    \ge (\mu^n_2 + \epsilon_1)\frac{\bfD^n}{K}
\ge (\mu^n_2 + \epsilon_1)\frac{\theta\bfD^n}{K}, \\
 \frac{1}{2} \frac{\bfD^n}{4} -1 & > \frac{2\lambda_2}{K} \bfD^n  \ge (\lambda^n_2 + \epsilon_1)\frac{\bfD^n}{K}
\ge (\lambda^n_2 + \epsilon_1)\frac{\theta\bfD^n}{K}.
\ea\ee
From \eqref{LDP-upper} in Lemma \ref{LDP-renewal},  Lemmas \ref{multi-1}, and \ref{multi-2}, and a conditioning argument,
when $nt \ge 2/ \epsilon_1$ and $n \ge \max\{n(\{\mu^n_2\}, \epsilon_1), n(\{\lambda^n_2\}, \epsilon_1), n_1\}$ 
\begin{align}
& \PP\left(\tilde S^n_2(\frac{\theta\bfD^n}{K}) \ge \frac{1}{2} \frac{\bfD^n}{4}, \beta^{n,k}_m \le nt \right) + \PP\left(\tilde A^n_2(\frac{\theta\bfD^n}{K}) \ge \frac{1}{2} \frac{\bfD^n}{4}, \beta^{n,k}_m \le nt \right) \nonumber\\
 &\le  \  \PP\left( S^n_2(\frac{\theta\bfD^n}{K}) \ge (\mu_2^n+\epsilon_1) \frac{\theta\bfD^n}{K}\right)  + \PP\left(A^n_2(\frac{\theta\bfD^n}{K}) \ge (\lambda_2^n+\epsilon_1) \frac{\theta\bfD^n}{K}\right) \nonumber\\
& \le   \exp\left\{- ( \frac{\mu_2\theta\bfD^n}{K} -1) \Theta^{s,2}_1(\mu_2,   \eps_1) \right\}  
 + \exp\left\{-( \frac{\lambda_2\theta\bfD^n}{K} -1) \Theta^{a,2}_1(\lambda_2,\eps_1) \right\}.\label{LDP-estimate-4}
\end{align}
This proves the claim.
Next,
\begin{align}
\PP(\clb^{n,k}_m\cap (\dmnk)^c ) & \le \PP\left( \tilde\eta^{n,k}_{2m-1}-\beta^{n,k}_m -\frac{\theta\bfD^n}{K} \le \frac{\bfD^n}{K},  (\dmnk)^c,  \clb^{n,k}_m\right)\nonumber  \\
& + \PP\left( \tilde\eta^{n,k}_{2m-1}-\beta^{n,k}_m -\frac{\theta\bfD^n}{K} > \frac{\bfD^n}{K}, (\dmnk)^c, \ \clb^{n,k}_m\right) \nonumber \\
& \le   \PP\left( \tilde\eta^{n,k}_{2m-1}-\beta^{n,k}_m -\frac{\theta\bfD^n}{K} \le \frac{\bfD^n}{K}, 
 \sup_{0\le s \le \frac{\theta\bfD^n}{K}} Q^n_2(\beta^{n,k}_m +s)  \le \frac{3}{2} \frac{\bfD^n}{4}, \ \clb^{n,k}_m\right)\label{sup-1} \\
& + \PP\left( \tilde\eta^{n,k}_{2m-1}-\beta^{n,k}_m -\frac{\theta\bfD^n}{K} > \frac{\bfD^n}{K},  \inf_{0\le s \le \frac{\theta\bfD^n}{K}} Q^n_2(\beta^{n,k}_m +s) \ge \frac{1}{2} \frac{\bfD^n}{4}, \ \clb^{n,k}_m\right).\label{inf-1} 
\end{align}
On the event in \eqref{sup-1}, $Q^n_2(\beta^{n,k}_m +s)  \le \frac{3}{2} \frac{\bfD^n}{4}$ for all $s\in [0, \frac{\theta\bfD^n}{K}]$, and $Q^n_2(\beta^{n,k}_m +s)\ge \bfD^n$ for some $s\in [\frac{\theta\bfD^n}{K}, \tilde\eta^{n,k}_{2m-1}-\beta^{n,k}_m] \subset [\frac{\theta\bfD^n}{K}, \frac{\bfD^n}{K}+\frac{\theta\bfD^n}{K}].$ Thus there must be at least $ \frac{1}{2} \frac{\bfD^n}{4}$ arrivals to Buffer 2 in $\frac{\bfD^n}{K}$ time, and so \eqref{sup-1} is bounded by 
\bes
\PP\left( A^n_2\left(\frac{\bfD^n}{K}\right) \ge \frac{1}{2} \frac{\bfD^n}{4}\right).
\ees
From \eqref{time-estimate}, 
and using similar argument as in \eqref{LDP-estimate-4}, we have the following upper bound for the probability
in \eqref{sup-1},
\begin{align}
 \PP\left( A^n_2\left(\frac{\bfD^n}{K}\right) \ge  (\lambda_2^n + \epsilon_1)\frac{\bfD^n}{K}\right) 
 \le \exp\left\{-( \frac{\lambda_2\bfD^n}{K} -1) \Theta^{a,2}_1(\lambda_2, \eps_1) \right\}. 
\label{eq:eq646}
\end{align}
Next consider the  event
\bes
\cla^{n,k}_m \doteq \bigcup_{i=2,3}\left\{ \left|\tilde S^n_i\left( \frac{(\theta + 1)\bfD^n}{K} \right) - \mu_i^n \frac{(\theta + 1)\bfD^n}{K} \right| \ge  \epsilon_1\frac{(\theta + 1)\bfD^n}{K}\right\}.
\ees
Then
\be\label{ldp-estimate-3}\ba
 \PP(\cla^{n,k}_m) 
 & \le \sum_{i=2, 3} \PP\left(\tilde S^n_i\left( \frac{(\theta + 1)\bfD^n}{K} \right) \ge (\mu_i^n+\epsilon_1) \frac{(\theta + 1)\bfD^n}{K}  \right) \\
 & \quad + \sum_{i=2, 3} \PP\left(\tilde S^n_i\left( \frac{(\theta + 1)\bfD^n}{K} \right) \le (\mu_i^n-\epsilon_1) \frac{(\theta + 1)\bfD^n}{K}  \right).
\ea\ee
From Lemmas \ref{multi-2}, \ref{multi-1}  and \eqref{LDP-upper} in Lemma \ref{LDP-renewal}, when $n\ge \max\{n(\{\mu^n_2\}, \epsilon_1), n(\{\mu^n_3\}, \epsilon_1)\}$ and $nt \ge 2/\epsilon_1$,
\bes\ba
 \sum_{i=2, 3} \PP\left(\tilde S^n_i\left( \frac{(\theta + 1)\bfD^n}{K} \right) \ge (\mu_i^n+\epsilon_1) \frac{(\theta + 1)\bfD^n}{K}  \right)  \le \sum_{i=2,3} \exp\left\{-( \frac{\mu_i(\theta+1)\bfD^n}{K} -1) \Theta^{s,i}_1(\mu_i,\eps_1) \right\}.
\ea\ees
Let $\tilde v_2^n$ and $\tilde v_3^n$ denote the residual service times for the jobs in service at time $\beta^{n,k}_m$ in Buffers 2 and 3, respectively. Then from \eqref{LDP-lower} in Lemma \ref{LDP-renewal}, and once again using Lemmas \ref{multi-2}, \ref{multi-1}, we have 
\bes\ba
& \sum_{i=2, 3} \PP\left(S^n_i\left( \frac{(\theta + 1)\bfD^n}{K} \right) \le (\mu_i^n-\epsilon_1) \frac{(\theta + 1)\bfD^n}{K}  \right) \\
& \le \sum_{i=2,3} \exp\left(-(\mu_i  - 2\epsilon_1) \frac{(\theta + 1)\bfD^n}{K}   \Theta^{s,i}_2(\mu_i, \eps_1)  \right) + \sum_{i=2,3}\PP\left(\tilde v^n_i >  \frac{\eps_1(\theta + 1)\bfD^n}{2\mu_i^nK}  \right).
\ea\ees
Now using \eqref{LDP-lower-2}, similar to \eqref{lower-prob-estimate}, we have for $i=2,3$,
\begin{align}
 \PP\left(\tilde v^n_i >  \frac{\eps_1(\theta + 1)\bfD^n}{2\mu_i^nK}  \right) 
& \le \exp\left(-(\mu_i nt -1) \Theta^{s,i}_1(\mu_i, \eps_1) \right)\nonumber\\
&\quad  + (\mu^n_i + \epsilon_1) nt \exp\left(-\frac{p_0 \epsilon_1 (\theta+1)D^n}{2\mu_iK} \right) \exp(\Lambda_{s,i}(p_0)),
\label{lower-prob-estimate-2}\end{align}
where $0 < p_0\in\clo.$ Combining the above estimates, we have 
\begin{align}
\PP(\cla^{n,k}_m) & \le \sum_{i=2,3} \exp\left\{-( \frac{\mu_i(\theta+1)\bfD^n}{K} -1) \Theta^{s,i}_1(\mu_i, \eps_1) \right\} \nonumber \\
& \quad + \sum_{i=2,3} \exp\left(-(\mu_i  - 2\epsilon_1) \frac{(\theta + 1)\bfD^n}{K}   \Theta^{s,i}_2(\mu_i, \eps_1)  \right)+ \sum_{i=2,3} \exp\left(-(\mu_i nt -1) \Theta^{s,i}_1(\mu_i, \eps_1) \right) \nonumber\\
& \quad +\sum_{i=2,3} (\mu^n_i + \epsilon_1) nt \exp\left(-\frac{p_0 \epsilon_1 (\theta+1)D^n}{2\mu_iK} \right) \exp(\Lambda_{s,i}(p_0)).\label{eq:eq701}
\end{align}
We next observe that on the intersection of $(\cla^{n,k}_m)^c$ and the event in \eqref{inf-1}, for $s\in [0, \frac{(\theta + 1)\bfD^n}{K}]$, $Q^n_2(\beta^{n,k}_m+s)$ is always nonzero, because 
$Q_2^n(\beta^{n,k}_m) \ge \bfD^n/4$ and 
\bes
\tilde S^n_2\left(  \frac{(\theta + 1)\bfD^n}{K} \right)  \le (\mu_2^n+\epsilon_1)\frac{(\theta + 1)\bfD^n}{K} \le 2 \mu_2 \frac{2\bfD^n}{K}  < \frac{\bfD^n}{8}.  
\ees
Also on this set, $\tilde \eta^k_{2m-2} < \frac{(\theta + 1)\bfD^n}{K} + \beta^{n,k}_m < \tilde \eta^k_{2m-1}$
and so 
\begin{equation}
	Q^n_3(\frac{(\theta + 1)\bfD^n}{K} + \beta^{n,k}_m) < C^n-1.\label{eq:eq711}
	\end{equation} 
Consequently, according to our policy, Server 1 works on Buffer 2 continuously over the interval 
$[\beta^{n,k}_m, \beta^{n,k}_m+ \frac{(\theta + 1)\bfD^n}{K}]$. Thus, on this set,
\begin{align*}
 Q^n_3\left( \beta^{n,k}_m+  \frac{(\theta + 1)\bfD^n}{K}\right) 
& \ge \tilde S^n_2 \left( \frac{(\theta + 1)\bfD^n}{K} \right) - \tilde S^n_3\left( \frac{(\theta + 1)\bfD^n}{K}\right) \\
& \ge (\mu^n_2 - \mu^n_3 - 2\epsilon_1) \frac{(\theta + 1)\bfD^n}{K}\\
& \ge \frac{\mu_2 - \mu_3}{2} \frac{\bfD^n}{K} \\
& \ge \bfc^n.
\end{align*}
However this contradicts \eqref{eq:eq711} and so we must have that the intersection of $(\cla^{n,k}_m)^c$ and the event in \eqref{inf-1} is empty.
Thus the probability in \eqref{inf-1} 
can be bounded by $\PP(\cla^{n,k}_m)$, and combining this observation with \eqref{eq:eq701} and the bound on the probability  in \eqref{sup-1} obtained
in \eqref{eq:eq646},
\be\label{estimate-second-1}\ba
&\PP(\clb^{n,k}_m) \\
& \le 
 \exp\left\{-( \frac{\lambda_2\bfD^n}{K} -1) \Theta^{a,2}_1(\lambda_2, \eps_1) \right\}  +  \sum_{i=2,3} \exp\left\{-( \frac{\mu_i(\theta+1)\bfD^n}{K} -1) \Theta^{s,i}_1(\mu_i,\eps_1) \right\} \\
& \quad + \sum_{i=2,3} \exp\left(-(\mu_i  - 2\epsilon_1) \frac{(\theta + 1)\bfD^n}{K}   \Theta^{s,i}_2(\mu_i, \eps_1)  \right) + \sum_{i=2,3} \exp\left(-(\mu_i nt -1) \Theta^{s,i}_1(\mu_i,\eps_1) \right) \\
&\quad +\sum_{i=2,3} (\mu^n_i + \epsilon_1) nt \exp\left(-\frac{p_0 \epsilon_1 (\theta+1)D^n}{2\mu_iK} \right) \exp(\Lambda_{s,i}(p_0)).
\ea \ee
We now consider the event in \eqref{estimate-new-2}, and denote it by $\clc^{n,k}_m$, i.e.,
\bes\ba
\clc^{n,k}_m & \doteq \left\{\beta^{n,k}_m = \tilde\eta^{n,k}_{2m-2}, \ \int_{[\beta^{n,k}_m, \ \tilde\eta^{n,k}_{2m-1}\wedge nt)} \bfi_{\{Q^n_2(s)\ge \bfD^n \}} d I^n_2(s) \neq 0 \right\}.
\ea\ees
When $k=m=1$, on $\clc^{n,1}_1$, we have $\beta^{n,1}_1 = \tilde\eta^{n,1}_0 = 0$, and so $Q^n_2(\beta^{n,1}_1)=0$, which is a contradiction to the definition of $\beta^{n,1}_1$. Thus $\PP(\clc^{n,1}_1)=0$. Consider $(k,m)\neq (1,1)$. 
Let
$$\hmnk =  \left\{\inf_{0\le s \le \frac{\theta\bfD^n}{K} } Q^n_3(\beta^{n,k}_m +s) \le \frac{\bfl^n}{2} 
 \mbox{ or }  \inf_{0\le s \le \frac{\theta\bfD^n}{K} } Q^n_2(\beta^{n,k}_m +s) \le \frac{1}{2}\frac{\bfD^n}{4}\right\}.$$
We now estimate $\PP( \hmnk \cap \clc^{n,k}_m).$ 
Observe that on $\{ \tilde\eta^{n,k}_{2m-2} < \tau^n_{2k-1} \wedge nt\},$ for $k\ge 2, m=1$,
\bes
Q^n_3(\tilde\eta^{n,k}_{2m-2}-) = Q^n_3(\tau^n_{2k-2}-) \ge \min\{\bfl^n, \bfc^n-1\} = \bfl^n,
\ees
and for $k\ge 1, m\ge 2$, 
\bes
Q^n_3(\tilde\eta^{n,k}_{2m-2}-) \ge \bfc^n -1 > \bfl^n.
\ees
Thus for $(k,m)\neq (1,1)$, on $\clc^{n,k}_m$, we have $Q^n_3(\beta^{n,k}_m) = Q^n_3(\tilde\eta^{n,k}_{2m-2}) \ge \bfl^n -1.$ By the definition of $\beta^{n,k}_m$, we have $Q^n_2(\beta^{n,k}_m) \ge \frac{\bfD^n}{4}$. For $Q^n_2(\beta^{n,k}_m+s) \le \frac{1}{2}\frac{D^n}{4}$ for some $s\in [0, \frac{\theta D^n}{K}]$, we need at least $\frac{1}{2}\frac{D^n}{4}$ service completions for Buffer 2 in $\frac{\theta D^n}{K}$ time. Similarly, to make $Q^n_3(\beta^{n,k}_m+s) \le \frac{L^n}{2}$ for some $s\in [0, \frac{\theta D^n}{K}]$, we need at least $\frac{L^n}{4}$ service completions for Buffer 3 in $\frac{\theta D^n}{K}$ time. Thus we have  
\bes
\PP( \hmnk \cap \clc^{n,k}_m)\le \PP\left(\tilde S^n_3(\frac{\theta\bfD^n}{K}) \ge \frac{\bfl^n}{4}\right) +\PP\left(\tilde S^n_2(\frac{\theta\bfD^n}{K}) \ge \frac{1}{2}\frac{\bfD^n}{4}\right). 
\ees
Noting that 
\bes
\frac{\bfl^n}{4} \ge 2\mu_3 \frac{ 2\theta d }{K}\bfl^n \ge (\mu_3^n +\epsilon_1)\frac{2\theta (\bfD^n-d)}{K} \ge (\mu_3^n +\epsilon_1) \frac{\theta\bfD^n}{K},\;\;
\frac{1}{2}\frac{\bfD^n}{4} \ge \frac{2\mu_2 \theta \bfD^n}{K} \ge (\mu_2^n+\eps_1)\frac{\theta \bfD^n}{K},
\ees
we have when $n\ge \max\{n(\{\mu^n_3\}, \epsilon_1), n(\{\mu^n_2\}, \epsilon_1), n_1\}$ and $nt \ge 2/\epsilon_1$,
\begin{align}
\PP( \hmnk \cap \clc^{n,k}_m)
 & \le \PP\left(\tilde S^n_3(\frac{\theta\bfD^n}{K}) \ge (\mu^n_3 + \epsilon_1)\frac{\theta\bfD^n}{K}\right) +\PP\left(\tilde S^n_2(\frac{\theta\bfD^n}{K}) \ge (\mu^n_2 + \epsilon_1)\frac{\theta\bfD^n}{K}\right) \non \\
 & \le \sum_{i=2,3} \exp\left\{-( \frac{\mu_i\theta\bfD^n}{K} -1) \Theta^{s,i}_1(\mu_i, \eps_1) \right\}.\label{ldp-estimate-2}
\end{align}
Next,
\begin{align}
\PP((\hmnk)^c\cap \clc^{n,k}_m ) & \le \PP\left( \tilde\eta^{n,k}_{2m-1}-\beta^{n,k}_m -\frac{\theta\bfD^n}{K} > \frac{\bfD^n}{K}, (\hmnk)^c \cap \clc^{n,k}_m\right)\label{estimate-new-5}\\
& +\PP\left( \tilde\eta^{n,k}_{2m-1}-\beta^{n,k}_m - \frac{\theta\bfD^n}{K} \le \frac{\bfD^n}{K}, (\hmnk)^c \cap \clc^{n,k}_m
\right).\label{estimate-new-6}
\end{align}
Using a similar argument as for \eqref{inf-1}, we see that 
\begin{equation}
\PP\left( \tilde\eta^{n,k}_{2m-1}-\beta^{n,k}_m -\frac{\theta\bfD^n}{K} > \frac{\bfD^n}{K}, (\hmnk)^c \cap \clc^{n,k}_m\right)
\le 	\PP(\cla^{n,k}_m).
\label{eq:eq754}
\end{equation}
%
%
We now consider the probability in \eqref{estimate-new-6}. Define 
$
t_0 \doteq \frac{\bfl^n}{4 \mu_3},
$
and
\bes
N \doteq\left \lfloor \frac{\frac{\bfD^n}{K}}{t_0} \right \rfloor + 1 \in \left[ \left\lfloor \frac{8 c\mu_3}{\mu_2-\mu_3} \right\rfloor +1,  \left\lfloor \frac{16 c\mu_3}{\mu_2-\mu_3} \right\rfloor +1\right].
\ees
For $s\in [0, t_0], j=2,3$ and $l=0, \ldots, N-1$, define 
\begin{align*}
	S^n_{j, l+1}(s) &= S^n_j\left(T^n_j\left(\beta^{n,k}_m + \frac{\theta\bfD^n}{K} + lt_0\right)+ s\right) - S^n_j\left(T^n_j\left(\beta^{n,k}_m + \frac{\theta\bfD^n}{K} + lt_0 \right)\right)\\
	A^n_{2,l+1}(s) &=  A^n_2\left(\beta^{n,k}_m + \frac{\theta\bfD^n}{K} + lt_0+ s\right) - A^n_2\left(\beta^{n,k}_m + \frac{\theta\bfD^n}{K} + lt_0\right).
	\end{align*}
	 Consider the event
\bes
\cle^{n,k}_{m,1} = \left\{ \left|\frac{S^n_{3,1}(t_0) - \mu_3^n t_0}{t_0}\right| > \epsilon_1 \ \mbox{or} \ \left|\frac{ S^n_{2,1}(t_0) - \mu_2^n t_0}{t_0}\right| >  \epsilon_1 \ \mbox{or} \ \left|\frac{ A^n_{2,1}(t_0) - \lambda_2^n t_0}{t_0}\right| >  \epsilon_1  \right\}.
\ees
We first estimate $\PP(\cle^{n,k}_{m,1})$. Clearly,
\bes\ba
\PP(\cle^{n,k}_{m,1}) & \le \sum_{j=2,3} \PP\left( \left|\frac{S^n_{j,1}(t_0) - \mu_j^n t_0}{t_0}\right| > \epsilon_1\right)  + \PP\left( \left|\frac{A^n_{2,1}(t_0) - \lambda_2^n t_0}{t_0}\right| > \epsilon_1\right) \\
& = \sum_{j=2,3} \PP\left( S^n_{j,1}(t_0) > (\mu_j^n+\epsilon_1) t_0\right)  + \PP\left(A^n_{2,1}(t_0) > (\lambda_2^n+\epsilon_1) t_0\right)  \\
& \quad + \sum_{j=2,3} \PP\left( S^n_{j}(t_0) < (\mu_j^n-\epsilon_1) t_0\right)  + \PP\left(A^n_{2}(t_0) < (\lambda_2^n-\epsilon_1) t_0\right).
\ea\ees
Now from Lemmas \ref{multi-1}, \ref{multi-2} and \eqref{LDP-upper} in Lemma \ref{LDP-renewal}, we have for 
$$n\ge \max\{n(\{\mu^n_3\}, \epsilon_1), n(\{\mu^n_2\}, \epsilon_1), n(\{\lambda^n_2\}, \epsilon_1), n_1\} \mbox{ and } nt > 2/\epsilon_1,$$
\bes\ba
& \sum_{j=2,3} \PP\left( S^n_{j,1}(t_0) > (\mu_j^n+\epsilon_1) t_0\right)  + \PP\left(A^n_{2,1}(t_0) > (\lambda_2^n+\epsilon_1) t_0\right) \\
& \le \sum_{j=2,3}\exp\left\{-( \mu_jt_0 -1) \Theta^{s,j}_1(\mu_j,\epsilon_1) \right\}  + \exp\left\{-( \lambda_2t_0 -1) \Theta^{a,2}_1(\lambda_2,\epsilon_1) \right\}.
\ea\ees
Let $\check v^n_2, \check v^n_3$ denote the residual service times for jobs in service at time $\beta^{n,k}_m + \frac{\theta\bfD^n}{K}$ in Buffers 2 and 3, respectively, and $\check u^n_2$ the residual arrival time at the same instant for jobs to Buffer 2. From \eqref{LDP-lower} in Lemma \ref{LDP-renewal}, and Lemmas \ref{multi-1}, \ref{multi-2} again,
 we have when $n\ge \max\{n(\{\mu^n_3\}, \epsilon_1), n(\{\mu^n_2\}, \epsilon_1), n(\{\lambda^n_2\}, \epsilon_1), n_1\}$, 
\bes\ba
& \sum_{j=2,3} \PP\left( S^n_{j,1}(t_0) < (\mu_j^n-\epsilon_1) t_0\right)  + \PP\left(A^n_{2,1}(t_0) < (\lambda_2^n-\epsilon_1) t_0\right) \\
& \le \sum_{j=2,3} \exp\left(-(\mu_j  - 2\epsilon_1) t_0 \Theta^{s,j}_2(\mu_j, \epsilon_1) \right) + \sum_{j=2,3}\PP\left(\check v^n_j >  \frac{\epsilon t_0}{2\mu_j^n}  \right) \\
& \quad + \exp\left(-(\lambda_2  - 2\epsilon_1) t_0 \Theta^{a,2}_2(\lambda_2,\epsilon_1) \right) + \PP\left(\check u^n_2 >  \frac{\epsilon t_0}{2\lambda_2^n}  \right). 
\ea\ees
Using similar arguments as in \eqref{lower-prob-estimate} and \eqref{lower-prob-estimate-2}, we have for $j=2,3,$
\bes\ba
\PP\left(\check v^n_j >  \frac{\epsilon t_0}{2\mu_j^n}  \right)  \le \exp\left(-(\mu_j nt -1) 
\Theta^{s,j}_1(\mu_j,\epsilon_1) \right) 
 + (\mu^n_j + \epsilon_1) nt \exp\left(-\frac{p_0 \epsilon_1 t_0}{2\mu_j} \right) \exp(\Lambda_{s,j}(p_0)),
\ea\ees
and 
\bes\ba
\PP\left(\check u^n_2 >  \frac{\epsilon t_0}{2\lambda_2^n}  \right)  \le \exp\left(-(\lambda_2 nt -1) 
\Theta^{a,2}_1(\lambda_2,\epsilon_1) \right) 
 + (\lambda_2^n + \epsilon_1) nt \exp\left(-\frac{p_0 \epsilon_1 t_0}{2\lambda_2} \right) \exp(\Lambda_{a,2}(p_0)),
\ea\ees
where $0< p_0\in\clo.$ Thus
\bes\ba
\PP(\cle^{n,k}_{m,1}) & \le \sum_{j=2,3}\exp\left\{-( \mu_jt_0 -1) 
\Theta^{s,j}_1(\mu_j,\epsilon_1) \right\}  + \exp\left\{-( \lambda_2t_0 -1) 
\Theta^{a,2}_1(\lambda_2,\epsilon_1) \right\} \\
& \quad + \sum_{j=2,3} \exp\left(-(\mu_j  - 2\epsilon_1) t_0 \Theta^{s,j}_2(\mu_j,\epsilon_1) \right) 
 + \sum_{j=2,3} \exp\left(-(\mu_j nt -1) \Theta^{s,j}_1(\mu_j,\epsilon_1) \right)\\
&\quad  + 
\sum_{j=2,3} (\mu^n_j + \epsilon_1) nt \exp\left(-\frac{p_0 \epsilon_1 t_0}{2\mu_j} \right) \exp(\Lambda_{s,j}(p_0)) 
 + \exp\left(-(\lambda_2  - 2\epsilon_1) t_0 \Theta^{a,2}_2(\lambda_2,\epsilon_1) \right) \\
& \quad + \exp\left(-(\lambda_2 nt -1) \Theta^{a,2}_1(\lambda_2,\epsilon_1) \right)  
+ (\lambda_2^n + \epsilon_1) nt \exp\left(-\frac{p_0 \epsilon_1 t_0}{2\lambda_2} \right) \exp(\Lambda_{a,2}(p_0)) \\
& \doteq \Gamma(t_0, nt).
\ea\ees
On the intersection $(\cle^{n,k}_{m,1})^c$ and the event in \eqref{estimate-new-6}, we have $Q^n_2(\beta^{n,k}_m +\frac{\theta\bfD^n}{K}) > \frac{1}{2}\frac{D^n}{4}, Q^n_3(\beta^{n,k}_m +\frac{\theta\bfD^n}{K}) > \frac{\bfl^n}{2}$, and for $s\in [0, t_0]$, 
\be\label{estimate-new-7}\ba
Q^n_2\left(\beta^{n,k}_m +\frac{\theta\bfD^n}{K} +s\right)  & \ge Q^n_2\left(\beta^{n,k}_m +\frac{\theta\bfD^n}{K}\right) + A^n_{2,1}(s) -  S^n_{2,1}(s) \\
& > \frac{1}{2}\frac{D^n}{4} - (\mu^n_2 + \epsilon_1) t_0 \\
& = \frac{1}{2}\frac{D^n}{4} - \frac{\mu^n_2 + \epsilon_1}{2\mu_3} \frac{\bfl^n}{2} \\
& > \left(\frac{1}{8} - \frac{\mu_2}{2\mu_3d} \right) D^n > 0, 
\ea\ee
where the last inequality follows from noting that $K \ge \frac{4\mu_2(\mu_2 - \mu_3)}{\mu_3}$. 
Since $Q^n_2$ is nonempty on the time interval $[\beta^{n,k}_m +\frac{\theta\bfD^n}{K}, \beta^{n,k}_m +\frac{\theta\bfD^n}{K}+t_0]$, Server 1 will work on Buffer 2 continuously during this time interval. Consequently, we have 
\be\label{estimate-new-10}\ba
Q^n_3\left(\beta^{n,k}_m + \frac{\theta\bfD^n}{K} + t_0\right)  & \ge Q^n_3\left(\beta^{n,k}_m + \frac{\theta\bfD^n}{K}\right) + S^n_{2,1}(t_0) - S^n_{3,1}(t_0) \\
& >  \frac{\bfl^n}{2} + (\mu^n_2- \mu^n_3 - 2\epsilon_1) t_0 \\
& \ge  \frac{\bfl^n}{2}. 
\ea\ee
Furthermore, 
 for all $s \in [0, t_0]$,
\be\label{estimate-new-8}\ba
Q^n_3\left(\beta^{n,k}_m + \frac{\theta\bfD^n}{K}+s\right)  & = Q^n_3\left(\beta^{n,k}_m + \frac{\theta\bfD^n}{K}\right) +  S^n_{2, 1}(s) -  S^n_{3,1}(s) \\
& > \frac{\bfl^n}{2} - (\mu^n_3 + \epsilon_1) t_0 \\
& = \frac{\bfl^n}{2} - \frac{\mu^n_3 + \epsilon_1}{2\mu_3} \frac{\bfl^n}{2} \\
& > 0.
\ea\ee
We repeat the above analysis for the time interval $[t_0, 2t_0]$. 
First define the event $\cle^{n,k}_{m,2}$ as $\cle^{n,k}_{m,1}$ by replacing
$S^n_{j, 1}(s)$ and $A^n_{2, 1}(s)$ with $S^n_{j, 2}(s)$ and $A^n_{2, 2}(s)$.
We then obtain that for 
$$n\ge \max\{n(\{\mu^n_3\}, \epsilon_1), n(\{\mu^n_2\}, \epsilon_1), n(\{\lambda^n_2\}, \epsilon_1), n_1\} \mbox{ and } nt \ge 2/\epsilon_1,$$
 $\PP(\cle^{n,k}_{m,2})$ has the same upper bound $\Gamma(t_0, nt)$ as $\PP(\cle^{n,k}_{m,1})$.
On the intersection $(\cle^{n,k}_{m,1})^c\cap (\cle^{n,k}_{m,2})^c$ and the event in \eqref{estimate-new-6}, following the similar arguments to those in \eqref{estimate-new-10} and \eqref{estimate-new-8}, it can be shown that 
$$Q^n_3\left(\beta^{n,k}_m + \frac{\theta\bfD^n}{K}+ 2t_0\right) > \frac{L^n}{2}, \ \; Q^n_3\left(\beta^{n,k}_m + \frac{\theta\bfD^n}{K}+ t_0 + s\right) > 0, \ \mbox{for $s\in [0,t_0]$,}$$ 
and similar to \eqref{estimate-new-7}, we have for $s\in [0, t_0]$,
\begin{align*}
Q^n_2\left(\beta^{n,k}_m +\frac{\theta\bfD^n}{K} + t_0 + s\right)  & \ge Q^n_2\left(\beta^{n,k}_m +\frac{\theta\bfD^n}{K}+t_0\right) + A^n_{2,2}(s) -  S^n_{2,2}(s) \\
& > \left(\frac{1}{8} - \frac{\mu_2}{\mu_3d} \right) D^n > 0. 
\end{align*}
Repeating this argument $N$ times we see that on the intersection $\cap_{j=1}^N(\cle^{n,k}_{m,j})^c$ and the event in \eqref{estimate-new-6},
 for $s\in[0, t_0],$ and $l=0, \ldots, N-1$,
\be\label{queue3-nonempty}
 Q^n_3\left(\beta^{n,k}_m +\frac{\theta\bfD^n}{K} + lt_0+ s\right) > 0
\ee
and 
\begin{align*}
Q^n_2\left(\beta^{n,k}_m +\frac{\theta\bfD^n}{K} + lt_0 + s\right)  & \ge Q^n_2\left(\beta^{n,k}_m +\frac{\theta\bfD^n}{K}+lt_0\right) + A^n_{2,l+1}(s) -  S^n_{2,l+1}(s) \\
& > \left(\frac{1}{8} - \frac{(l+1)\mu_2}{2\mu_3d} \right) D^n > 0, 
\end{align*}
where the last inequality follows from noting that $K \ge 32 \mu_2+ \frac{4\mu_2(\mu_2 - \mu_3)}{\mu_3}$.

We note that, since Server 2 does not idle if there are jobs in Buffer 3, \eqref{queue3-nonempty} implies that the probability of the intersection $\cap_{j=1}^N(\cle^{n,k}_{m,j})^c$ and the event in \eqref{estimate-new-6} is $0$. Thus 
\bes\ba
\PP\left( \tilde\eta^{n,k}_{2m-1}-\beta^{n,k}_m - \frac{\theta\bfD^n}{K} \le \frac{\bfD^n}{K}, (\hmnk)^c \cap \clc^{n,k}_m
\right) \le \sum_{l=1}^N \PP(\cle^{n,k}_{m,l}) \le N \Gamma(t_0, nt).
\ea\ees
Combining the above estimate with \eqref{ldp-estimate-2}, \eqref{eq:eq754} and \eqref{eq:eq701},
\begin{align*}
 \PP(\clc^{n,k}_m) 
&\le \sum_{i=2,3} \exp\left\{-( \frac{\mu_i\theta\bfD^n}{K} -1) \Theta^{s,i}_1(\mu_i, \eps_1) \right\} + \sum_{i=2,3} \exp\left\{-( \frac{\mu_i(\theta+1)\bfD^n}{K} -1) 
\Theta^{s,i}_1(\mu_i,\epsilon_1) \right\} \\
& + \sum_{i=2,3} \exp\left(-(\mu_i  - 2\epsilon_1) \frac{(\theta + 1)\bfD^n}{K}   \Theta^{s,i}_2(\mu_i, \eps_1)  \right)+ \sum_{i=2,3} \exp\left(-(\mu_i nt -1) 
\Theta^{s,i}_1(\mu_2,\epsilon_1) \right) \\
& +\sum_{i=2,3} (\mu^n_i + \epsilon_1) nt \exp\left(-\frac{p_0 \epsilon_1 (\theta+1)D^n}{2\mu_iK} \right) \exp(\Lambda_{s,i}(p_0))  + N \Gamma(t_0, nt).
\end{align*}
Combining the above estimate with \eqref{LDP-estimate-1},  \eqref{estimate-new} and \eqref{estimate-second-1}, we have for large enough $n$ and $nt\ge \epsilon_1$, 

\begin{align}
 \PP(\clr(n,t)) 
&\le   \bfk^n\bfk_1^n \left[\exp\left\{-( \frac{\lambda_2\bfD^n}{K} -1) 
\Theta^{a,2}_1(\lambda_2,\epsilon_1) \right\} 
 +  2\sum_{i=2,3} \exp\left\{-( \frac{\mu_i(\theta+1)\bfD^n}{K} -1) 
\Theta^{s,i}_1(\mu_i,\epsilon_1) \right\} \right. \nonumber\\
& \quad + 2  \sum_{i=2,3} \exp\left\{-(\mu_i  - 2\epsilon_1) \frac{(\theta + 1)\bfD^n}{K}   \Theta^{s,i}_2(\mu_i, \eps_1)  \right\}   + 2 \sum_{i=2,3} \exp\left\{-(\mu_i nt -1) \Theta^{s,i}_1(\mu_2,\epsilon_1) \right\} \nonumber\\
&\quad +2\sum_{i=2,3} (\mu^n_i + \epsilon_1) nt \exp\left\{-\frac{p_0 \epsilon_1 (\theta+1)D^n}{2\mu_iK} \right\} \exp(\Lambda_{s,i}(p_0))\nonumber\\
&\quad + \left.\sum_{i=2,3} \exp\left\{-( \frac{\mu_i\theta\bfD^n}{K} -1) \Theta^{s,i}_1(\mu_i, \eps_1) \right\}   
+ N \Gamma(t_0, nt) \right]\nonumber\\
&\quad + \bfk^n \exp\left\{-(\mu_2 nt -1) \Theta^{s,2}_1(\mu_2,1) \right\} 
 +  \exp\left\{- (\lambda_1 nt -1) \Theta^{a,1}_1(\lambda_1,1) \right\}\nonumber\\
&\quad + \exp\left\{- (\lambda_2 nt -1) \Theta^{a,2}_1(\lambda_2,{1}) \right\}. \label{final-est-2}
\end{align}
Thus, we can find positive constants $\gamma_i, i=1,2,3,4,$ such that for large enough $n$ and $nt\ge \epsilon_1$, 
\bes
\PP(\clr(n,t)) \le \gamma_1(nt+1)^2 e^{-\gamma_2 nt} + \gamma_3 (nt+1)^3 n^{-\gamma_4 l_0}.
\ees

\end{proof}

\subsection{Proof of Lemma \ref{lower-bound-est}}

Recall the stopping times $\{\tau^n_k\}_{k\in \NN_0}$ defined in \eqref{stopping-times} and note from Lemma \ref{imp-fcn} that when 
$t\in [\tau^n_{2k-2}, \tau^n_{2k-1})$,
\begin{align}
\Delta^n(t)& \doteq \frac{Q^n_1(t)}{\mu^n_1}+\frac{Q^n_2(t)}{\mu^n_2} - \sqrt{n}\Psi\left( \frac{ Q^n_2(t)+Q^n_3(t)}{\sqrt{n}\mu^n_3}\right) \nonumber \\
& \ge \frac{Q^n_1(t)}{\mu^n_1}+\frac{ Q^n_2(t)}{\mu^n_2} - \frac{\mu_3}{\mu_2}\left( \frac{ Q^n_2(t)}{\mu^n_3}+\frac{Q^n_3(t)}{\mu^n_3}\right) \nonumber \\
& \ge \frac{Q^n_1(t)}{\mu^n_1} - \frac{\mu_3}{\mu_2\mu^n_3} \left( \frac{\mu^n_2}{\mu^n_1} {Q^n_1(t)} + L^n\right) + \left(\frac{1}{\mu^n_2} - \frac{\mu_3}{\mu_2 \mu^n_3} \right)  Q^n_2(t) \nonumber\\
& =\frac{1}{\mu^n_1} \left(1- \frac{\mu_3\mu^n_2}{\mu^n_3\mu_2} \right) Q^n_1(t) - \frac{\mu_3}{\mu_2\mu^n_3}L^n +\left(\frac{1}{\mu^n_2} - \frac{\mu_3}{\mu_2 \mu^n_3} \right) Q^n_2(t). \label{lower-1}
\end{align}
Thus for such $t$ the inequality in \eqref{eq:eq141} clearly holds (with a suitable choice of $C_1$).

Next consider intervals of the form $[\tau^n_{2k-1}, \tau^n_{2k})$.
 For $k, l\in\NN$, we  define a sequence of $\{\clg^n_1(t)\}_{t\ge 0}$ stopping times within $[\eta^{n,k}_{2l-2}, \eta^{n,k}_{2l-1}) \subset [\tau^n_{2k-1}, \tau^n_{2k})$ as follows: For $m\in\NN$, 
\begin{align*}
\zeta^{n,k,l}_0 & = \eta^{n,k}_{2l-2}, \\
\zeta^{n,k,l}_{2m-1} & = \tau^n_{2k}\wedge\eta^{n,k}_{2l-1}\wedge \inf\left\{ t\ge \zeta^{n,k,l}_{2m-2}: W^n(t)  - \sqrt{n}\Psi(W^n_2(t)/\sqrt{n}) < g_0 \right\}, \\
\zeta^{n,k,l}_{2m} & = \tau^n_{2k}\wedge\eta^{n,k}_{2l-1}\wedge \inf\left\{ t\ge \zeta^{n,k,l}_{2m-1}: W^n(t)  - \sqrt{n}\Psi(W^n_2(t)/\sqrt{n}) \ge g_0 \right\}.
\end{align*}
Then 
\begin{equation}
	\mbox{ for } t\in [\zeta^{n,k,l}_{2m-2}, \zeta^{n,k,l}_{2m-1}),\;  
W^n(t)  - \sqrt{n}\Psi(W^n_2(t)/\sqrt{n}) \ge g_0 \label{eq:eq147}
\end{equation}
and so for such $t$ the inequality in \eqref{eq:eq141} holds trivially.

Also, from the definition of $\{\eta^{n,k}_{l}\}_{l\in\NN_0}$ and \eqref{no-effect}, we see that for 
$t\in [\eta^{n,k}_{2l-2}, \eta^{n,k}_{2l-1})$,
\be
\label{eq:eq858}
Q^n_3(t) -  \frac{\mu^n_2}{\mu^n_1}Q^n_1(t) \ge \bfl^n,\; Q^n_1(t) < \frac{\mu^n_1}{\mu^n_2} (\bfc^n-\bfl^n+2).
\ee
Thus from Definition \ref{propolicy} we see that over this interval Server 1 does not work on Buffer 2. Thus letting
\begin{align*}
\tilde A^n_2(t) \doteq A^n_2(t) - A^n_2(\zeta^{n,k,l}_{2m-1}), \;\;
 \tilde S^n_3(t) &\doteq S^n_3(T^n_3(t)) - S^n_3(T^n_3(\zeta^{n,k,l}_{2m-1})),
 \end{align*} 
for $t \in [\zeta^{n,k,l}_{2m-1}, \zeta^{n,k,l}_{2m})$,
$Q_2^n(t) = Q_2^n(\zeta^{n,k,l}_{2m-1}) + \tilde A^n_2(t)$ and
$Q_3^n(t) = Q_3^n(\zeta^{n,k,l}_{2m-1}) - \tilde S^n_3(t)$.
 Using the nondecreasing and Lipschitz continuity property of $\Psi$, we now have for $t\in [\zeta^{n,k,l}_{2m-1}, \zeta^{n,k,l}_{2m})$,
\begin{align}
\Delta^n(t) & = \frac{Q^n_1(t)}{\mu^n_1}+\frac{Q^n_2(t)}{\mu^n_2} -\sqrt{n} \Psi\left( \frac{ Q^n_2(t)+Q^n_3(t)}{\sqrt{n}\mu^n_3}\right) \nonumber\\
& \ge  \frac{Q^n_2(\zeta^{n,k,l}_{2m-1})+ \tilde A^n_2(t) }{\mu^n_2} -\sqrt{n}\Psi\left( \frac{Q^n_2(\zeta^{n,k,l}_{2m-1})+ \tilde A^n_2(t)+Q^n_3(\zeta^{n,k,l}_{2m-1})- \tilde  S^n_3(t) }{\sqrt{n}\mu^n_3}\right)\nonumber\\
& = \frac{ \tilde A^n_2(t) }{\mu^n_2} - \frac{Q^n_1(\zeta^{n,k,l}_{2m-1})}{\mu^n_1} + \left[\frac{Q^n_1(\zeta^{n,k,l}_{2m-1})}{\mu^n_1} + \frac{Q^n_2(\zeta^{n,k,l}_{2m-1}) }{\mu^n_2} -\sqrt{n} \Psi\left( \frac{Q^n_2(\zeta^{n,k,l}_{2m-1})+Q^n_3(\zeta^{n,k,l}_{2m-1}) }{\sqrt{n}\mu^n_3}\right)\right] \nonumber\\
& \quad  -\sqrt{n}\left[ \Psi\left( \frac{Q^n_2(\zeta^{n,k,l}_{2m-1})+ \tilde A^n_2(t)+Q^n_3(\zeta^{n,k,l}_{2m-1})- \tilde  S^n_3(t) }{\sqrt{n}\mu^n_3}\right)- \Psi\left( \frac{Q^n_2(\zeta^{n,k,l}_{2m-1})+Q^n_3(\zeta^{n,k,l}_{2m-1}) }{\sqrt{n}\mu^n_3}\right) \right]\nonumber\\
& \ge  \frac{ \tilde A^n_2(t) }{\mu^n_2} - \frac{Q^n_1(\zeta^{n,k,l}_{2m-1})}{\mu^n_1}+ \Delta^n(\zeta^{n,k,l}_{2m-1})  - \frac{\mu_3}{\mu^n_3\mu_2} (\tilde A^n_2(t)-\tilde  S^n_3(t))^+. \nonumber 
\end{align}
Using the Lipschitz property of $\Psi$ again we see that there exists $C_0>0$ (depending only on $\Psi$ and the rate
parameters) such that 
\begin{align*}
\Delta^n(\zeta^{n,k,l}_{2m-1}) \ge \Delta^n(\zeta^{n,k,l}_{2m-1}-) - C_0 \ge g_0 - C_0,
\end{align*}
where the last inequality is from \eqref{eq:eq147}.
Furthermore, for all $t\in [\zeta^{n,k,l}_{2m-1}, \zeta^{n,k,l}_{2m})$, $\tilde A^n_2(t) \le Q^n_2(t)$
and (from \eqref{eq:eq858})
$Q^n_1(t) \le \frac{\mu^n_1}{\mu^n_2}(C^n-L^n+2)$.
Thus we have for such $t$ (in fact for all $t\in [\eta^{n,k}_{2l-2}, \eta^{n,k}_{2l-1})$),
\begin{align}
\Delta^n(t)  \ge  - \left|\frac{1}{\mu^n_2}-\frac{\mu_3}{\mu^n_3\mu_2}\right|Q^n_2(t)  -  \frac{1}{\mu^n_2} (\bfc^n-\bfl^n+2) + g_0-C_0. \label{lower-2} 
\end{align}
We have therefore shown that \eqref{eq:eq141} holds for all $t \in [\eta^{n,k}_{2l-2}, \eta^{n,k}_{2l-1})$.

We finally consider the case when $t\in [\eta^{n,k}_{2l-1}, \eta^{n,k}_{2l})$(with a suitable choice of $C_1, C_2$). 
As noted below \eqref{effect-1}, over this interval Server 1 works on Buffer 1 continuously, and 
since, from the definition of $\tau^n_{2k-1}, \tau^n_{2k}$, over this interval $Q_3^n(t) \ge L^n-1$, Server 2 works on Buffer 3 continuously. Let 
\begin{align*}
 \check A^n_2(t)  \doteq A^n_2(t) - A^n_2(\eta^{n,k}_{2l-1}), \;\; \check S^n_3(t)  \doteq S^n_3(t-\eta^{n,k}_{2l-1} + T_3^n(\eta^{n,k}_{2l-1})) - S^n_3(T_3^n(\eta^{n,k}_{2l-1})).
\end{align*}
 Thus again from the monotonicity and Lipschitz continuity property of $\Psi$, we have for $t\in  [\eta^{n,k}_{2l-1}, \eta^{n,k}_{2l}),$
\begin{align*}
\Delta^n(t) & = \frac{Q^n_1(t)}{\mu^n_1}+\frac{Q^n_2(t)}{\mu^n_2} - \sqrt{n}\Psi\left( \frac{ Q^n_2(t)+Q^n_3(t)}{\sqrt{n}\mu^n_3}\right) \\
& =  \frac{Q^n_1(\eta^{n,k}_{2l-1})+ [Q^n_1(t) - Q^n_1(\eta^{n,k}_{2l-1}) ]}{\mu^n_1}+ \frac{Q^n_2(\eta^{n,k}_{2l-1})+ \check A^n_2(t) }{\mu^n_2}  \\
& \quad -\sqrt{n}\Psi\left( \frac{Q^n_2(\eta^{n,k}_{2l-1})+ \check A^n_2(t)+Q^n_3(\eta^{n,k}_{2l-1})- \check S^n_3(t) }{\sqrt{n}\mu^n_3}\right) \\
&= \frac{Q^n_1(t) - Q^n_1(\eta^{n,k}_{2l-1}) }{\mu^n_1}+ \frac{ \check A^n_2(t) }{\mu^n_2} + \left[\frac{Q^n_1(\eta^{n,k}_{2l-1})}{\mu^n_1} + \frac{Q^n_2(\eta^{n,k}_{2l-1}) }{\mu^n_2} - \sqrt{n}\Psi\left( \frac{Q^n_2(\eta^{n,k}_{2l-1})+Q^n_3(\eta^{n,k}_{2l-1}) }{\sqrt{n}\mu^n_3}\right)\right] \\
& \quad  -\left[ \sqrt{n}\Psi\left( \frac{Q^n_2(\eta^{n,k}_{2l-1})+ \check A^n_2(t)+Q^n_3(\eta^{n,k}_{2l-1})- \check S^n_3(t) }{\sqrt{n}\mu^n_3}\right) - \sqrt{n}\Psi\left( \frac{Q^n_2(\eta^{n,k}_{2l-1})+Q^n_3(\eta^{n,k}_{2l-1}) }{\sqrt{n}\mu^n_3}\right)  \right]\\
& \ge   \frac{Q^n_1(t) - Q^n_1(\eta^{n,k}_{2l-1}) }{\mu^n_1}+ \frac{ \check A^n_2(t) }{\mu^n_2} + \Delta^n(\eta^{n,k}_{2l-1}) - \frac{\mu_3}{\mu^n_3\mu_2} (\check A^n_2(t)-\check  S^n_3(t))^+.
\end{align*}
Also, from \eqref{lower-2}, 
\begin{align*}
 \Delta^n(\eta^{n,k}_{2l-1})  \ge \Delta^n(\eta^{n,k}_{2l-1}-) -C_0  \ge  - \left|\frac{1}{\mu^n_2}-\frac{\mu_3}{\mu^n_3\mu_2}\right|Q^n_2(\eta^{n,k}_{2l-1}-)  -  \frac{1}{\mu^n_2} (\bfc^n-\bfl^n+2) + g_0-2C_0,
\end{align*}
and using \eqref{no-effect}
$$
Q^n_1(\eta^{n,k}_{2l-1})  \le Q^n_1(\eta^{n,k}_{2l-1}-) + 1 < \frac{\mu^n_1}{\mu^n_2}(C^n-L^n+2)+1.$$
Furthermore, for all $t\in [\eta^{n,k}_{2l-1}, \eta^{n,k}_{2l})$,
$$\check A^n_2(t)  \le Q^n_2(t), \mbox{ and }
Q^n_2(\eta^{n,k}_{2l-1}-)  \le Q^n_2(\eta^{n,k}_{2l-1}) + 1 \le Q^n_2(t) + 1.
$$
Using the above estimates, we have for $t\in  [\eta^{n,k}_{2l-1}, \eta^{n,k}_{2l}),$
\begin{align}
 \Delta^n(t) 
&\ge 
- \frac{1}{\mu^n_2}(C^n-L^n+2)-1 -  \left|\frac{1}{\mu^n_2}- \frac{\mu_3}{\mu^n_3\mu_2}\right|Q^n_2(t) + \Delta^n(\eta^{n,k}_{2l-1}-)  - C_0 \nonumber \\
& \ge - \frac{2}{\mu^n_2}(C^n-L^n+2)-1 -  \left|\frac{1}{\mu^n_2}- \frac{\mu_3}{\mu^n_3\mu_2}\right| (2Q^n_2(t)+ 1) +g_0 - 2 C_0. \label{lower-3} 
\end{align}
Thus for such $t$, \eqref{eq:eq141} holds with suitable choice of $C_1, C_2$ as well. The result follows. 

\qed

\setcounter{equation}{0}
\appendix
\numberwithin{equation}{section}

\section{One-dimensional Skorohod map}\label{SP}

We recall below the definition and basic properties of the $1$-dimensional Skorohod map \cite{skorohod61}. 
 Recall $\mathcal{D}_1=\{x\in \mathcal{D}([0,\infty): \RR): x(0)\geq 0\}$. 

\begin{definition}[One-dimensional Skorohod Problem (SP)]\label{sp-def}
Let $x\in \mathcal{D}_1$. A pair $(z,y)\in\mathcal{D}([0,\infty): \RR_+)\times \mathcal{D}([0,\infty): \RR_+)$ is a solution of the Skorohod problem for $x$ if the following hold.
\begin{itemize}
\item[\rm (i)] For all $t\geq 0, z(t)=x(t)+y(t)\geq 0$.
\item[\rm (ii)] $y$ satisfies the following: (a) $y(0)=0$, (b) $y$ is nondecreasing, and (c) $y$ increases only when $z=0$, that is, 
$\int_{[0,\infty)} z(t)dy(t) = 0.$
\end{itemize}
\end{definition}
The following proposition summarizes some well known properties of the $1$-dimensional SP (see \cite{chen91} for a proof).

\begin{proposition}\label{sp-property}\hfill
\begin{itemize}
\item[\rm (i)] Let $x\in \mathcal{D}_1$. Then there exists a unique solution $(z,y)\in\mathcal{D}([0,\infty): \RR_+)\times \mathcal{D}([0,\infty): \RR_+)$ of the SP for $x$, which is given as follows:
\bes
y(t) = - \inf_{0\leq s\leq t} (x(s) \wedge 0), \;  \; z(t) = x(t) - \inf_{0\leq s\leq t} (x(s)\wedge 0), \; t \ge 0.
\ees
We write $z = \Gamma(x)$, and refer to the map $\Gamma: \mathcal{D}_1 \to \mathcal{D}([0,\infty): \RR_+)$ as the Skorohod map. Let $\cli: \mathcal{D}([0,\infty): \RR) \to \mathcal{D}([0,\infty): \RR)$ be the identity functional. Then $y = (\Gamma - \cli)(x).$
\item[\rm (ii)] The  Skorohod map $\Gamma$ is Lipschitz continuous in the following sense: For all $t\geq 0$ and $x_1, x_2\in \mathcal{D}_1$, 
$$\sup_{0\leq s\leq t}|\Gamma(x_1)(s)-\Gamma(x_2)(s)|\leq 2 \sup_{0\leq s\leq t}|x_1(s)-x_2(s)|.$$
\item[\rm (iii)] Fix $x\in \mathcal{D}_1$. Let $(z,y)\in\mathcal{D}([0,\infty): \RR_+)\times \mathcal{D}([0,\infty): \RR_+)$ be such that 
\begin{itemize}
\item[\rm (a)] $z(t)=x(t)+y(t)\geq 0, \; t\geq 0$,
\item[\rm (b)] $y$ is nondecreasing with $y(0)=0$.
\end{itemize}
Then  $z(t)\geq \Gamma(x)(t), \; t\geq 0$.
\end{itemize}
\end{proposition}

\section*{Acknowledgements}

This work  has been supported in part by the National Science Foundation(DMS-1016441, DMS-1305120) and the Army Research
 Office (W911NF-10-1-0158, W911NF-14-1-0331).

\providecommand{\bysame}{\leavevmode\hbox to3em{\hrulefill}\thinspace}
\providecommand{\MR}{\relax\ifhmode\unskip\space\fi MR }
\providecommand{\MRhref}[2]{%
  \href{http://www.ams.org/mathscinet-getitem?mr=#1}{#2}
}
\providecommand{\href}[2]{#2}


\begin{thebibliography}{10}

\bibitem{bellwill01}
S.~L. Bell and R.~J. Williams, \emph{Dynamic scheduling of a system with two
  parallel servers in heavy traffic with resource pooling: Asymptotic
  optimality of a threshold policy}, Ann. Appl. Probab. \textbf{11} (2001),
  no.~3, 608 -- 649.

\bibitem{bellwill05}
\bysame, \emph{Dynamic scheduling of a parallel server system in heavy traffic
  with complete resource pooling: Asymptotic optimality of a threshold policy},
  Electron. J. Probab. \textbf{10} (2005), 1044--1115.

\bibitem{BSW}
V.~E. Bene{\v{s}}, L.~A. Shepp, and H.~S. Witsenhausen, \emph{Some solvable
  stochastic control problems}, Stochastics \textbf{4} (1980/81), no.~1,
  39--83.

\bibitem{AA}
A.~Budhiraja and A.~P. Ghosh, \emph{A large deviation approach to
  asymptotically optimal control of crisscross network in heavy traffic}, The
  Annals of Applied Probability \textbf{15} (2005), no.~3, 1887--1935.

\bibitem{AA1}
\bysame, \emph{Diffusion approximations for controlled stochastic networks: An
  asymptotic bound for the value function}, Ann. Appl Probab \textbf{16}
  (2006), no.~4, 1962--2006.

\bibitem{AA2}
\bysame, \emph{Controlled stochastic networks in
  heavy traffic: Convergence of value functions}, Ann. Appl. Probab.
  \textbf{22} (2012), no.~2, 734--791.
  
 \bibitem{BudLiu}
 A.~Budhiraja, A.~P. Ghosh, and X.~Liu, \emph{Scheduling control for Markov modulated single-server multiclass queueing systems in heavy traffic}, Queueing Systems \textbf{78} (2014), no.~1, 57--97. 

\bibitem{BudRos}
A.~Budhiraja and K.~Ross, \emph{Convergent numerical scheme for singular
  stochastic control with state constraints in a portfolio selection problem},
  SIAM J. Control Optim. \textbf{45} (2007), no.~6, 2169--2206.

\bibitem{AK}
\bysame, \emph{Optimal stopping and free boundary characterizations for some
  brownian control problems}, Ann Appl. Probab. \textbf{18} (2008), 2367--2391.



\bibitem{chen91}
H.~Chen and A.~Mandelbeaum, \emph{Leontief systems, {RBV\rq{}s} and
  {RBM\rq{}s}}, Applied Stochastic Analysis (M.~H.~A. Davis and R.~J. Elliott,
  eds.), Gordon and Breach, 1991, pp.~1 -- 43.

\bibitem{DaiLin}
J.~G. Dai and Wuqin Lin, \emph{Asymptotic optimality of maximum pressure
  policies in stochastic processing networks}, Ann. Appl. Probab. \textbf{18}
  (2008), no.~6, 2239--2299.

\bibitem{EthKur}
	S.~N. Ethier and T.~G. Kurtz, \emph{Markov processes: Characterization and convergence},
	 Wiley Series in Probability and Mathematical Statistics: Probability and Mathematical Statistics. John Wiley \& Sons, Inc., New York, 1986. 

\bibitem{harrison88}
J.~M. Harrison, \emph{Brownian models of queueing networks with heterogeneous
  customer population.}, Stochastic Differential Systems, Stochastic Control
  Theory and Applications (W.~Fleming and F.~L. Lion, eds.), Springer, New
  York, 1988, pp.~147 -- 186.

\bibitem{Har2}
J.~Michael Harrison, \emph{Heavy traffic analysis of a system with parallel
  servers: asymptotic optimality of discrete-review policies}, Ann. Appl.
  Probab. \textbf{8} (1998), no.~3, 822--848.

\bibitem{HarTak}
J.~Michael Harrison and Michael~I. Taksar, \emph{Instantaneous control of
  {B}rownian motion}, Math. Oper. Res. \textbf{8} (1983), no.~3, 439--453.

\bibitem{HarVan}
J.~Michael Harrison and Jan~A. Van~Mieghem, \emph{Dynamic control of {B}rownian
  networks: State space collapse and equivalent workload formulations}, Ann.
  Appl. Probab. \textbf{7} (1997), no.~3, 747--771.

\bibitem{HarWei}
J.~Michael Harrison and Lawrence~M. Wein, \emph{Scheduling networks of queues:
  heavy traffic analysis of a simple open network}, Queueing Systems Theory
  Appl. \textbf{5} (1989), no.~4, 265--279.

\bibitem{Kum}
Sunil Kumar, \emph{Two-server closed networks in heavy traffic: Diffusion
  limits and asymptotic optimality}, Ann. Appl. Probab. \textbf{10} (2000),
  no.~3, 930--961.

\bibitem{MutKum2}
Sunil Kumar and Kumar Muthuraman, \emph{A numerical method for solving singular
  stochastic control problems}, Oper. Res. \textbf{52} (2004), no.~4, 563--582.

\bibitem{KusMar}
H.~J. Kushner and L.~F. Martins, \emph{Numerical methods for stochastic
  singular control problems}, SIAM J. Control Optim. \textbf{29} (1991),
  1443--1475.

\bibitem{KusMar2}
Harold~J. Kushner and L.~Felipe Martins, \emph{Heavy traffic analysis of a
  controlled multiclass queueing network via weak convergence methods}, SIAM J.
  Control Optim. \textbf{34} (1996), no.~5, 1781--1797.

\bibitem{MSS}
L.~F. Martins, S.~E. Shreve, and H.~M. Soner, \emph{Heavy traffic convergence
  of a controlled, multiclass queueing system}, SIAM J. Control Optim.
  \textbf{34} (1996), 2133--2171.

\bibitem{MutKum1}
Kumar Muthuraman and Sunil Kumar, \emph{Solving free-boundary problems with
  applications in finance}, Found. Trends Stoch. Syst. \textbf{1} (2006),
  no.~4, 259--341.

\bibitem{ShSo1}
S.~E. Shreve and H.~M. Soner, \emph{A free boundary problem related to singular
  stochastic control}, Applied stochastic analysis ({L}ondon, 1989),
  Stochastics Monogr., vol.~5, Gordon and Breach, New York, 1991, pp.~265--301.

\bibitem{skorohod61}
A.~V. Skorohod, \emph{Stochastic equations for diffusions in a bounded region},
  Theory Probab. Appl. (1961), no.~6, 264 -- 274.

\bibitem{ShSo2}
H.~Mete Soner and Steven~E. Shreve, \emph{Regularity of the value function for
  a two-dimensional singular stochastic control problem}, SIAM J. Control
  Optim. \textbf{27} (1989), no.~4, 876--907.

\bibitem{YCY}
P.~Yang, H.~Chen, and D.~Yao, \emph{Control and scheduling in a two-station
  queueing network}, Queueing Syst. Theory Appl. \textbf{18} (1994), 301--332.

\end{thebibliography}

\skp

{\sc

\bigskip\noi
Amarjit Budhiraja\\
Department of Statistics and Operations Research\\
University of North Carolina\\
Chapel Hill, NC 27599, USA\\
email: budhiraj@email.unc.edu

\skp

\noi
Xin Liu\\
Department of Mathematical Sciences\\
    Clemson University\\
Clemson, SC 29634, USA\\
email:    xliu9@clemson.edu.

\skp

\noi
Subhamay Saha\\
Department of Electrical Engineering\\
Technion - Israel Institute of Technology\\
Haifa 32000, Israel\\
email: subhamay@tx.technion.ac.il

}

\end{document}